\documentclass[acmsmall]{acmart}

\usepackage{bm}
\usepackage{amsmath}
\usepackage{amsfonts}
\usepackage{amssymb}
\usepackage{algorithmic}
\usepackage{algorithm}
\usepackage{subfigure}
\usepackage{pbox}
\usepackage{xcolor}

\newcommand{\R}{\mathbb{R}}
\newcommand{\idc}{\mathbf{1}}
\newcommand{\GG}{\mathcal{G}}
\newcommand{\N}{\mathcal{N}}
\newcommand{\E}{\mathcal{E}}
\newcommand{\F}{\mathcal{F}}

\newcommand{\p}{\mathbf{P}}
\newcommand{\q}{\mathbf{Q}}
\newcommand{\A}{\mathbf{A}}
\newcommand{\pp}{\mathbf{\tilde{P}}}

\newcommand{\rt}{\right}
\newcommand{\lt}{\left}
\newcommand{\mmu}{\bm{\mu}}
\newcommand{\hmmu}{\hat{\bm{\mu}}}
\newcommand{\hmu}{\hat{\mu}}
\newcommand{\ppi}{\bm{\pi}}
\newcommand{\nnu}{\bm{\nu}}
\newcommand{\xx}{\mathbf{x}}

\newcommand{\uu}{{\mathbf{u}}}
\newcommand{\pr}{\mathbb{P}}
\newcommand{\Ex}{\mathbb{E}}
\newcommand{\tp}{\tilde{P}}
\newcommand{\var}{\mathrm{Var}}
\newcommand{\cov}{\mathrm{Cov}}

\allowdisplaybreaks

\def\BibTeX{{\rm B\kern-.05em{\sc i\kern-.025em b}\kern-.08emT\kern-.1667em\lower.7ex\hbox{E}\kern-.125emX}}

\setcopyright{acmcopyright}
\acmJournal{POMACS}
\acmYear{2019}
\acmMonth{1}

\begin{document}

%
\title{Non-Markovian Monte Carlo on Directed Graphs}

%

\author{Chul-Ho Lee}
\affiliation{%
  \department{Department of Computer Engineering and Sciences}
  \institution{Florida Institute of Technology}
  \city{Melbourne}
  \state{FL}
  \postcode{32901}}
\email{clee@fit.edu}

\author{Min Kang}
\affiliation{%
  \department{Department of Mathematics}
  \institution{North Carolina State University}
  \city{Raleigh}
  \state{NC}
  \postcode{27695}}
\email{mkang2@ncsu.edu}

\author{Do Young Eun}
\affiliation{%
  \department{Department of Electrical and Computer Engineering}
  \institution{North Carolina State University}
  \city{Raleigh}
  \state{NC}
  \postcode{27695}}
\email{dyeun@ncsu.edu}

%
\begin{abstract}
Markov Chain Monte Carlo (MCMC) has been the de facto technique for sampling and inference of large graphs such as online social networks. At the heart of MCMC lies the ability to construct an ergodic Markov chain that attains any given stationary distribution $\bm{\pi}$, often in the form of random walks or crawling agents on the graph. Most of the works around MCMC, however, presume that the graph is undirected or has reciprocal edges, and become inapplicable when the graph is directed and non-reciprocal. Here we develop a similar framework for directed graphs, which we call Non-Markovian Monte Carlo (NMMC), by establishing a mapping to convert $\bm{\pi}$ into the quasi-stationary distribution of a carefully constructed transient Markov chain on an extended state space. As applications, we demonstrate how to achieve any given distribution $\bm{\pi}$ on a directed graph and estimate the eigenvector centrality using a set of non-Markovian, history-dependent random walks on the same graph in a distributed manner. We also provide numerical results on various real-world directed graphs to confirm our theoretical findings, and present several practical enhancements to make our NMMC method ready for practical use in most directed graphs. To the best of our knowledge, the proposed NMMC framework for directed graphs is the first of its kind, unlocking all the limitations set by the standard MCMC methods for undirected graphs.
\end{abstract}

%
%


%

%
\maketitle

\section{Introduction}\label{se:intro}

Markov Chain Monte Carlo (MCMC) is a general technique for sampling from an arbitrary probability distribution by constructing a Markov chain that achieves the target distribution as its stationary distribution~\cite{Liu04}. By definition, every MCMC method requires the so-called balance condition, i.e., $\pi(j) \!=\! \sum_i \pi(i) P_{ij}$ for all $j$, where $P_{ij}$ is the transition probability from state $i$ to $j$ and $\ppi \!=\! [\pi(1), \pi(2), \ldots, \pi(n)]$ is the stationary distribution. Most popular MCMC methods such as Metropolis-Hastings algorithm and Gibbs sampler, often referred to as Glauber dynamics, however, require a simple yet more restrictive condition, called \emph{detailed} balance condition, i.e., $\pi(i) P_{ij} \!=\! \pi(j) P_{ji}$ for all $i,j$~\cite{Hastings70,Peskun73,Liu04}. The detailed balance condition is generally preserved under a proposal and acceptance framework for state transitions. Assuming the current state $i$, a transition to state $j$ is proposed with probability $Q_{ij}$, which is the transition probability of an irreducible (proposal) Markov chain, and then accepted with acceptance probability $\gamma_{ij}$. This two-step process ensures the convergence to the desired stationary distribution $\ppi$ via achieving the detailed balance condition, i.e., $\pi(i) Q_{ij} \gamma_{ij} \!=\! \pi(j) Q_{ji} \gamma_{ji}$ for all $i,j$~\footnote{While the proposal and acceptance framework is mostly used for the Metropolis-Hastings algorithm, (random scan) Gibbs sampler also fits into the framework where the conditional distribution of the Gibbs sampler is considered as the proposal distribution along with the acceptance probability of one. In this sense, the Gibbs sampler is a special case of the Metropolis-Hastings algorithm. Furthermore, when it comes to binary spin systems, it is essentially the Barker's method~\cite{Barker65} that was compared with the Metropolis algorithm under the proposal and acceptance framework in~\cite{Hastings70,Peskun73}.}. The detailed balance condition imposes a serious restriction, named `reciprocity' condition, that $Q_{ij} \!>\! 0$ if and only if $Q_{ji} \!>\! 0$~\cite{Diaconis98,Liu04}.

The detailed balance condition is a local condition that eases the construction of a Markov chain in MCMC methods, but is also criticized as a source of slow convergence to equilibrium and high sampling variability because every transition should be balanced out with its reverse transition. Thus, some recent advances in MCMC methods have been made to remove the requirement of the detailed balance condition in order to accelerate the convergence speed or improve sampling efficiency. A popular class of strategies is to transform a given reversible Markov chain to its related nonreversible chain by using the idea of lifting~\cite{Chen99,Neal00,Fernandes11,Turitsyn11,Schram15,Sakai16,Vucelja16}, nonbacktracking~\cite{Neal04,Alon07,Diaconis13}, or vorticity~\cite{NIPS10,Chen13,Ichiki13,Bierkens16,Rey-Bellet16,Kaiser17}. The former two ideas are based on expanding the state space to give rise to persistence of motion or directionality in state transitions. The other idea is that for the transition matrix of a given reversible chain, which is symmetric with respect to $\ppi$, a proper antisymmetric matrix with $\ppi$ (or more matrices) is added to introduce a non-trivial probability flux flowing over a cycle in one direction. They all, however, require the reciprocity condition, which is carried over from the original reversible chain before the transformation.

On the other hand, MCMC also fits into the problem of sampling from a target probability distribution over a large, \emph{unknown} graph, or a huge state space whose exploration can only be done by traversing along the edges between nodes (states). The sampling problem with such limited access to the space often arises when sampling and estimating structural properties of various types of large networks, including peer-to-peer networks and online social networks, whose complete information is generally \emph{unavailable} in advance. MCMC has then been a key enabler to obtain a sequence of samples by constructing an ergodic Markov chain or `random walk' on a given graph, whose equilibrium distribution equals the target distribution, with only a limited and \emph{local} knowledge of the graph structure~\cite{WillingerToN09,RibeiroIMC10,GjokaJSAC11,LeeSIGMETRICS12,HardimanWWW13,BrunoSIGMETRICS16,History-vldb15,Rewiring-database16,Xin-hybrid-sampling,Xin-Eun-Infocom17,Lee-Eun-Infocom17}.

Most of the MCMC methods for constructing random walks on graphs, however, become \emph{inapplicable to directed} graphs, though many real-world networks can still only be represented by directed graphs~\cite{Newman03,Newman10-book,Kleinberg-book,Mendes17}. The famous examples are Web graphs, online social networks having follower-followee relationship, neuronal and metabolic networks, to name a few. The directed graphs naturally violate the reciprocity condition, i.e., there is at least one directed edge $(i,j)$ for which $Q_{ij} \!>\! 0$ but $Q_{ji} \!=\! 0$ due to the absence of its reverse directed edge $(j,i)$. The lack of reciprocity condition simply prohibits us from using the usual MCMC methods and even obtaining a closed-form expression of the stationary distribution of a given Markov chain on a directed graph. Thus, the stationary distribution has been merely estimated based on the visit counts obtained by running the Markov chain on a directed graph without knowing the explicit form of the distribution~\cite{VLDB2000,WWW2000,AAAI2001}. This is also the very reason why popular social-network-analysis metrics defined on a directed graph such as PageRank and eigenvector centrality are generally computed and estimated by the power method, assuming that a global knowledge of the whole graph structure is available a priori~\cite{Newman10-book}.

In this paper, we present a general algorithmic framework called Non-Markovian Monte Carlo (NMMC) that constructs a set of non-Markovian random walks whose historical empirical distribution converges to any desired probability distribution $\ppi$ over a directed, non-reciprocal graph, using only local information on the graph structure. The main idea is to establish a non-trivial mapping to convert the desired distribution into the so-called \emph{quasi-stationary distribution} (QSD) of a duly constructed \emph{transient} Markov chain on an extended state space. We then leverage history-dependent random walks or reinforced walks to estimate the QSD ``on the fly'', which does not require the entire topological information. This novel construction leads to a large degree of freedom and does not require the reciprocity condition that has been implicitly or explicitly assumed in almost all MCMC methods, and thus makes our approach the only viable solution to directed graphs. As a byproduct, our NMMC framework also allows us to estimate the eigenvector centrality of a given directed graph, one of the key metrics used in ranking or identifying influential nodes in epidemics or information spreading, by using the same set of history-dependent random-walk crawling agents. 

The rest of this paper is structured as follows. In Section~\ref{se:prelim}, we collect several terms and definitions on nonreversible Markov chains on directed graphs as well as on the eigenvector centrality (EVC) of a directed graph, to be used throughout the paper. In Section~\ref{se:qsd-mapping}, we provide necessary background on the QSD, its key properties and characterizations, and then formally establish a mapping to convert a given distribution $\ppi$ and the EVC of a graph into the QSD of a carefully constructed transient Markov chain. In Section~\ref{se:qsd-algo}, we first explain recent developments in the literature to estimate the QSD and then build our NMMC method with theoretical guarantees to achieve a given $\ppi$ and to estimate the EVC using random crawlers on directed graphs. We also discuss several ways to improve the speed of convergence and the practicability including dynamically adjusted algorithmic parameters based on run-time estimates on the fly. Section~\ref{se:simu} presents our extensive simulation results using a set of real-life directed graphs to numerically support our theoretical findings. In Section~\ref{se:discuss}, we thoroughly compare our NMMC method to the relevant literature, ranging from the traditional MCMC methods to the PageRank algorithm. We also present how to relax (potentially impractical) assumptions with our NMMC method, which require the in-degree information and the strong connectedness of the underlying directed graph, and demonstrate the feasibility of our NMMC method as a truly working solution to most directed graphs in practice. We provide related prior studies around sampling of directed graphs in Section~\ref{se:related} and conclude with a brief summary capturing our novel contributions in Section~\ref{se:conclu}.

\section{Preliminaries}\label{se:prelim}

\subsection{Nonreversible Markov Chains on Directed Graphs}

Let $\GG \!=\! (\N,\E)$ be a finite, aperiodic, strongly connected graph with a set of nodes $\N \!=\! \{1,2,\ldots,n\}$ and a set of \emph{directed} edges $\E$. The graph $\GG$ is defined by an $n \times n$ adjacency matrix $\A$, with elements $A_{ij} \!=\! 1$ if there is an edge from node $i$ to $j$ (directed), i.e., $(i,j) \in E$, and $A_{ij} \!=\! 0$ otherwise. Here $\A$ is asymmetric, namely, there is at least one directed edge $(i,j)$ for which $A_{ij} \!=\! 1$ but $A_{ji} \!=\! 0$ (non-reciprocal). Let $d^+_i \!=\! \sum_{j\in \N} A_{ij}$ be the out-degree of node $i$, i.e., the number of out-bound edges emanating from $i$. Similarly, let $d_i^- \!=\! \sum_{j\in \N} A_{ji}$ be the in-degree of node $i$ (the number of edges pointing to node $i$). We shall demonstrate how the `strong connectedness' of the underlying graph $\GG$ can be relaxed in Section~\ref{subse:nonscc}.

Consider an ergodic Markov chain on $\N$ with its transition matrix $\q \!=\! [Q_{ij}]$ such that $Q_{ij} \!>\! 0$ for every $(i,j) \!\in\! \E$, or $A_{ij} \!=\! 1$, which is irreducible and aperiodic on $\N$. The irreducibility of $\q$ is granted from the connectivity of the graph $\GG$, or the irreducibility of the adjacency matrix $\A$, in that for any $i,j \!\in\! \N$, there exists a path following directed edges from $i$ to $j$, i.e., there exists $k\!>\!0$ with $[\A^k]_{ij} \!>\! 0$. Similarly for the aperiodicity of $\q$. Then, since $\A$ is asymmetric, the ergodic Markov chain is nonreversible without satisfying the reciprocity condition. Its unique stationary distribution $\mmu \!=\! [\mu(1),\mu(2),\cdots,\mu(n)]$, which satisfies $\mu(j) \!=\! \sum_{i\in \N} \mu(i) Q_{ij}$ for all $j\!\in\!\N$, is also simply \emph{unknown} for a general graph $\GG$. For example, if we define a simple random walk on $\GG$, it is an ergodic, nonreversible Markov chain on $\N$ with its transition matrix
$\q^{\textsf{srw}} \!=\! [Q^{\textsf{srw}}_{ij}]$ given by $Q^{\textsf{srw}}_{ij} \!=\! A_{ij}/d_i^+$, where $d_i^+\!>\!0$ for all $i$. There is, however, no closed-form expression of the stationary distribution $\mmu$ for the simple random walk on a general directed graph.

\subsection{Eigenvector Centrality}
\label{subse:EVC}

In addition to sampling from a target probability distribution $\bm{\pi}$ on a directed graph, another metric of interest to estimate on the directed graph is the so-called eigenvector centrality (EVC). Specifically, the EVC of node $i \!\in\! \N$, say $x(i)$, represents the notion of importance (centrality) of node $i$ and is proportional to the sum of the EVC values of nodes pointing to $i$, i.e.,
\begin{equation}
\sum_{i\in\N} x(i) A_{ij} = \lambda x(j) ~\text{ for all}~ j \in \N \; \Longleftrightarrow \; \xx \A = \lambda \xx,
\label{eq-EVC}
\end{equation}
with some positive constant $\lambda \!>\! 0$, where $\xx$ is a row vector $\xx \!=\! [x(1),x(2),\cdots,x(n)]$. Since the adjacency matrix $\A$ is irreducible, the Perron-Frobenius theorem~\cite{Carl2000} asserts that there exists a unique solution $\xx$ to (\ref{eq-EVC}), which is the left leading eigenvector of $\A$ with all positive entries $x(i) \!>\! 0$, $i \!=\! 1,2,\ldots, n$, corresponding to the largest eigenvalue $\lambda \!=\! \lambda(\A) \!>\! 0$, i.e., the spectral radius of $\A$. Throughout this paper, we refer to this left leading eigenvector $\xx$ as the EVC of $\A$. Since $\beta\xx$ ($\beta \neq 0$) can also be a leading eigenvector for $\lambda(\A)$, we also assume that $\xx$ is normalized such that $\sum_{i\in\N}x(i)\!=\!1$.

The EVC finds many important applications. It has been playing a crucial role in ranking nodes in social network analysis and identifying influential nodes (users) in epidemics or information spreading~\cite{Canright06,Carreras2007,Newman10-book,Kleinberg-book,Lu-PR16}, and the reciprocal of its corresponding eigenvalue $1/\lambda(\A)$ is known to be the epidemic threshold of a general graph for SIR/SIS epidemic processes~\cite{Ganesh05,Draief-AAP08,Chakrabarti08,Mieghem-ToN09,Tong-TKDD16}. Again, we note that there is no closed-form expression of the EVC $\xx$. While the EVC has also been used as a main kernel for the maximal entropy random walk in the statistical physics~\cite{MERW09,MERW11}, its practical construction based only on local information still remains an open issue.

Even worse, unlike the stationary distribution $\mmu$ of an ergodic Markov chain, e.g., the simple random walk, satisfying $\mmu \!=\! \mmu \q$ with its (stochastic) transition matrix $\q$, the EVC $\xx$ in (\ref{eq-EVC}) cannot be made as a stationary distribution of an ergodic Markov chain (or a random walk) unless the adjacency matrix $\A$ can be made stochastic by some constant multiple, i.e., $d_i^+ \!=\!d$ for all $i$. Even if this is the case, the solution $\xx$ to (\ref{eq-EVC}) defies any closed-form expression since the in-degrees $d_i^-$ can be all different, reducing to the same situation with an ergodic Markov chain on a directed graph whose stationary distribution $\mmu$ is unknown. Thus, the only viable approach to compute the EVC $\xx$ is again to rely on the power method for which the entire adjacency matrix $\A$ needs to be known a priori. To the best of our knowledge, it remains open or unknown as to how (or whether or not it is possible) to build crawling algorithms using only local information towards estimating the EVC $\xx$.

\section{NMMC: Mapping to QSD}\label{se:qsd-mapping}

The main goal of our proposed NMMC framework is to achieve any arbitrary probability distribution $\ppi$ on a general directed graph $\GG$ and also to estimate the EVC of its adjacency matrix $\A$ using a set of random walks in a distributed manner without requiring the entire matrix $\A$ a priori. The main enabler of our NMMC framework to tackle this challenge is to convert any given ergodic, nonreversible Markov chain (e.g., simple random walk) on $\N$ with $\q \!=\! [Q_{ij}]_{i,j \in \N}$ and \emph{unknown} $\mmu$ into a related suitable \emph{transient} Markov chain on an augmented state space $\N^+ \!:=\! \N \cup \{0\}$, where the node set $\N$ is a set of transient (non-absorbing) states and `0' is an (artificially added) absorbing state. This mapping is made based solely on local information in a way that the resulting \emph{quasi-stationary distribution} (QSD) becomes the target distribution $\ppi$. Later in Section~\ref{se:qsd-algo} we present how to efficiently estimate the required QSD using a set of non-Markovian, history-dependent random walks, which is yet another integral part of our NMMC framework. In what follows, we first collect several definitions and basic equations we need for later use.

\subsection{A Primer on QSD}\label{subse:qsd-def}

Suppose that we have a discrete time \emph{transient} Markov chain
$\{X_t\}_{t\geq 0}$ on a finite state space $\N^+ \!:=\! \N \cup \{0\}$ with its transition matrix $\p \!=\! [P_{ij}]_{i,j \in N^+}$, where $\N \!=\! \{1,2,\cdots,n\}$ is the set of transient (non-absorbing) states and `0' is the absorbing state with $P_{00} \!=\! 1$, i.e., once the chain hits $0$, it remains there forever. We have $P_{i0} \!>\! 0$ for at least one $i \!\in\! \N$ to ensure that the chain is transient. Clearly, $\p$ is a stochastic matrix on $\N^+$. We assume that this Markov chain is irreducible and aperiodic on the set of \emph{non-absorbing} states in the sense that $\pr\{X_t \!=\! j~\mbox{for some}~ t \!\geq\! 1 ~| X_0 \!=\! i\} \!>\! 0$ and $\mbox{g.c.d.}\{t \!>\! 0 : P^{t}_{ii} \!>\! 0 \} \!=\! 1$ for all $i,j \!\in\! \N$, where $P^t_{ii} \!=\! \pr\{X_t \!=\! i ~| X_0 \!=\! i\}$ is a $t$-step transition probability.\footnote{Equivalently, this is often termed as $\N$ being an irreducible class~\cite{Doorn09,QSD-review-Doorn2013,Benaim15}.} Then, we have $\lim_{t\to\infty} X_t \!=\! 0$ almost surely (i.e., absorption is certain eventually). If we restrict the stochastic matrix $\p$ only on $\N$, the resulting $n\times n$ matrix $\pp \!=\! [\tilde{P}_{ij}]_{i,j\in\N}$ is naturally a \emph{sub-stochastic} matrix.

Consider the dynamics of the transient Markov chain $\{X_t\}$ at any given time $t$ by conditioning on the event that the chain is still surviving, i.e., the chain has not hit the absorbing site `0' by time $t$. These conditional probabilities $\pr\{X_t \!\in\! \cdot ~| X_t \!\in\! \N\} \!=\! \pr\{X_t \!\in\! \cdot ~| X_t \!\neq\! 0\}$ are well-defined time-evolving full probability measures fully supported on the set of non-absorbing states $\N$. Now, we can consider the stationary probability measure with respect to these conditional probability measures; in other words, the invariant distribution $\nnu = [\nu(1),\nu(2),\cdots,\nu(n)]$ under the Markovian dynamics {\it conditioned to survive}. This stationary distribution $\nnu$ is called the quasi-stationary distribution (QSD)~\cite{QSD-survey,Doorn09,QSD-review-Doorn2013,Groisman13,Benaim15,Min-Kang13,Blanchet2016}, satisfying
\begin{equation}\label{qsd-def}
\nu(j) = \pr\{X_{t} = j ~|~ X_0 \sim \nnu, X_t \neq 0\}
\end{equation}
for all non-absorbing states $j \!\in\! \N$ and for all time $t \!\geq\! 0$, where $X_0 \sim \nnu$ means that the initial state $X_0$ is drawn from the distribution $\nnu$.

Exactly in the same way as we approximate the stationary
distribution of an ergodic Markov chain from its long-term limiting behavior, we can take a similar line of approach to acquire the desired QSD. Specifically, by taking the limit on the conditional distribution at time $t$ in (\ref{qsd-def}) starting from an initial distribution $\mmu_0$, we obtain
\begin{equation}
\lim_{t\to\infty} \pr\{ X_t = j~|~ X_0 \sim \mmu_0, X_t \neq 0\} = \nu'(j)
\label{def-QLD}
\end{equation}
for all $j \!\in\! \N$. If there exists an initial distribution $\mmu_{0}$ such that $\nnu' := [\nu'(1),\nu'(2),\cdots,\nu'(n)]$ exists, we call this limiting distribution a quasi limiting distribution (QLD). It then follows that under our finite state space setting with irreducible and aperiodic assumptions, the QSD $\nnu$ and the QLD $\nnu'$ coincide, i.e., $\nnu \!=\! \nnu'$, and they are uniquely determined for any initial distribution $\mmu_0$~\cite{QSD-review-Doorn2013}.

We can rewrite (\ref{qsd-def}) with $t \!=\! 1$ as, for any $j \!\in\! \N$,
\begin{equation*}
\nu(j) = \frac{ \pr\{X_{1} = j ~|~ X_0 \sim \nnu\} }{ \pr\{X_{1} \neq 0 ~|~ X_0 \sim \nnu\} } = \frac{\sum_{i\in \N} \nu(i) P_{ij}}{ 1 - \sum_{i\in \N}\nu(i) P_{i0} },
\end{equation*}
which then reads as
\begin{equation}
\sum_{i\in\N} \nu(i) P_{ij} = \lambda \nu(j),
\label{eq:qsd}
\end{equation}
where
\begin{equation}
\lambda = \pr\{X_{1} \neq 0 ~|~ X_0 \sim \nnu\} = 1 - \sum_{i\in\N}\nu(i) P_{i0}.
\label{eq:lambda}
\end{equation}
In other words, since $P_{ij} \!=\! \tp_{ij}$ for $i,j \!\in\! \N$, we can rewrite (\ref{eq:qsd}) as $\nnu \pp \!=\! \lambda \nnu$, where $\pp$ is the sub-stochastic matrix, which is a restriction on $\N$ from the original transition matrix $\p$ on $\N^+$. Then, by the Perron-Frobenius theorem for irreducible, non-negative matrices (not necessarily stochastic)~\footnote{The Perron-Frobenius theorem asserts that the leading (left) eigenvector of any irreducible, non-negative matrix corresponding to its largest eigenvalue is the only one with every component being positive~\cite{Carl2000}.}, the QSD $\nnu$ must be the uniquely determined left principal eigenvector of $\pp$ with $\nu(j)\!>\!0$ for all $j \!\in\! \N$, corresponding to the maximal eigenvalue $\lambda \!>\! 0$. Note that $1 - \lambda = \sum_{i\in \N}\nu(i) P_{i0} > 0$ is the one-step absorption probability starting from the QSD $\nnu$.

\begin{figure*}[t!]
    \includegraphics[width=0.95\textwidth]{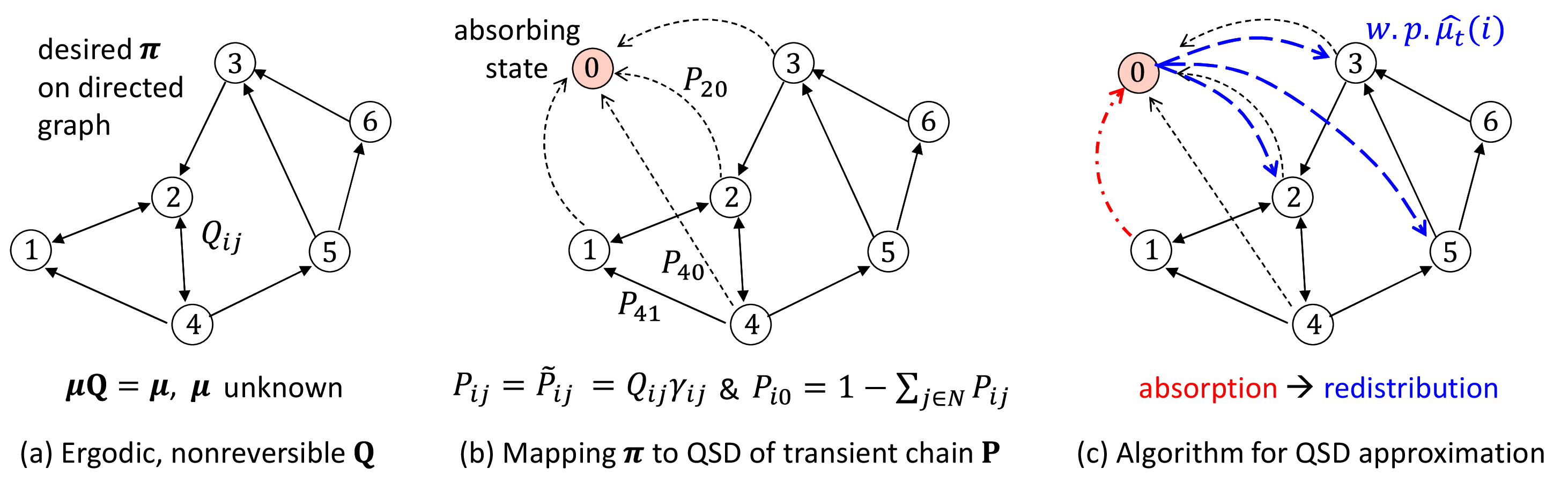}
    \vspace{-3mm}
    \caption{Summary of our Non-Markovian Monte Carlo method: With any ergodic, nonreversible Markov chain $\q$ defined over the node set $\N$ of a directed graph $\GG$, we construct a transient Markov chain $\p$ on $\N^+ \!=\! \N \cup \{0\}$ such that its QSD becomes our desired distribution $\ppi$ and then leverage non-Markovian, history-dependent random walks to estimate the desired QSD $\ppi$.\label{fig:qsd}}
    \vspace{-3mm}
\end{figure*}

\subsection{Achieving \texorpdfstring{$\ppi$}~~ via QSD}
\label{subse:qsd-mapping}

We next show how to construct a transient Markov chain $\p \!=\! [P_{ij}]$ on the augmented state space $\N^+ \!=\! \N \!\cup\! \{0\}$ with a virtual absorbing site `0' added to the original state space $\N$, such that the resulting QSD $\nnu$ becomes any desired target probability distribution $\ppi \!=\! [\pi(1),\pi(2),\cdots,\pi(n)]$ with $\pi(i) \!>\! 0$ for $i \!\in\! \N$, i.e., $\nnu \!=\! \ppi$. This mapping, as we shall show shortly, is always possible as long as the underlying nonreversible Markov chain $\q \!=\! [Q_{ij}]$ on the original state space $\N$ (or the directed graph $\GG$) is ergodic. Note that letting $\mmu \!=\! [\mu(1),\mu(2),\cdots,\mu(n)]$ be the stationary distribution of the ergodic chain $\q$, satisfying $\mmu \!=\! \mmu \q$, we do not require any knowledge on $\mmu$, i.e., $\mmu$ is \emph{unknown} in our construction, as illustrated in Figure~\ref{fig:qsd}(a).

Adopting the proposal and acceptance framework as in usual MCMC methods, we first use the transition probability $Q_{ij}$ as a proposal distribution from the current state (node) $i$ to $j$ ($i,j \!\in\! \N$). This proposed move is then `accepted' with probability $\gamma_{ij}$, which will be determined shortly. Unlike usual MCMC methods, however, with probability $1 \!-\! \gamma_{ij}$, this move is `rejected', in which case the chain gets absorbed to `0' and stays there forever. We thus have $P_{ij} = Q_{ij} \gamma_{ij}$ for $i,j \!\in\! \N$, and $P_{i0} = \sum_{j\in \N} Q_{ij}(1 - \gamma_{ij})$ for each $i \!\in\! \N$, with $P_{00} \!=\! 1$. Recall that $\tp_{ij} \!=\! P_{ij}$ for $i,j \!\in\! \N$. See for Figure~\ref{fig:qsd}(b) for an illustration.

Despite the similar structure, our proposed method is fundamentally different from the MCMC methods, including the famous Metropolis-Hastings (MH) algorithm. The crucial difference is that our proposed chain $\q$ is \emph{nonreversible} on a directed graph \emph{without the reciprocity condition}, and goes to the absorbing site `0' if rejected. In contrast, in the MH algorithm the proposed chain is on a undirected graph (typically reversible), and if rejected, it simply stays at the current node and tries again in the next time step.

The next step is to choose the acceptance probability $\gamma_{ij} \in (0,1]$ such that
\begin{equation}
\sum_{i\in \N} \pi(i) \tp_{ij} = \sum_{i\in \N} \pi(i) Q_{ij}\gamma_{ij} = \kappa\pi(j)
\label{QSD:eq3}
\end{equation}
is satisfied for all $j \!\in\! \N$, i.e., $\ppi \pp \!=\! \kappa \ppi$ for some $\kappa \!>\! 0$. Then, by the Perron-Frobenius theorem, the distribution $\ppi$ is the \emph{unique} left leading eigenvector of $\pp$ and thus coincides with the unique QSD $\nnu$ of the transient chain $\p$, corresponding to the largest eigenvalue  $\lambda \!=\! \kappa$. The key aspects here are (i) how to choose $\gamma_{ij}$ using only locally available information and (ii) how to efficiently estimate or approximate the resulting QSD, or to obtain samples from the QSD in a distributed manner. In what follows, we present the specific form of $\gamma_{ij}$ that maps any desired $\ppi$ on a given directed graph to the QSD $\nnu$ of the suitably constructed transient chain $\p$. The case of (ii) using crawling walks on a directed graph will be discussed later in Section~\ref{se:qsd-algo}.

For a given (proposed) ergodic, nonreversible chain $\q$, we define $S_j \!:=\! \{k \!\in\! \N ~|~ Q_{kj} \!>\! 0\}$ for each $j \!\in\! \N$. This is the set of states (nodes) from which the proposed chain $\q$ can enter into state $j$ in one step with positive probability. For each $S_j$, consider a probability distribution $\{\alpha_{kj}\}_{k \in S_j}$ such that $\alpha_{kj} \!>\! 0$ for all $k \!\in\! S_j$ and $\sum_{k \in S_j}\alpha_{kj} \!=\! 1$. Note that the set $S_j$ is the support of the probability distribution $\{\alpha_{\cdot j}\}$. We then have the following.

\begin{theorem}
\label{thm1:general-qsd}
For any choice of $\{\alpha_{kj}\}_{k \in S_j}$ on each $S_j$, $j \!\in\! \N$, if we set
\begin{equation}
\gamma_{ij} \propto \frac{\pi(j)}{\pi(i)} \cdot \frac{\alpha_{ij}}{Q_{ij}},
\label{r-qij-general}
\end{equation}
such that $\gamma_{ij}\leq 1$ for any pair $(i,j)$ with $Q_{ij}>0$, then the target distribution $\ppi$ becomes the QSD $\nnu$ of the transient chain $\p$, where $P_{ij} \!=\! Q_{ij} \gamma_{ij}$.
\end{theorem}
\begin{proof}
Let $\ppi$ be the given target distribution on $\N$. Fix $j \!\in\! \N$. Choose any arbitrary probability distribution $\{\alpha_{kj}\}_{k \in S_j}$ on the support $S_j$. By definition of $S_j$, we observe that for any $i \not\in S_j$, i.e., $Q_{ij} = 0$, the choice of $\gamma_{ij}$ is irrelevant since the proposed move from $i$ to $j$ will never take place. We can rewrite $\gamma_{ij}$ in (\ref{r-qij-general}) as $\gamma_{ij} = \frac{1}{c}\cdot\frac{ \pi(j)\alpha_{ij} }{ \pi(i)Q_{ij} }$ with some suitable
scaling constant $c > 1$ such that $\gamma_{ij} \leq 1$. Now, observe that
\begin{equation}
  \sum_{i \in \N} \pi(i) P_{ij} = \sum_{i \in S_j} \pi(i) Q_{ij}\gamma_{ij} = \sum_{i \in S_j} \frac{1}{c} \pi(j) \alpha_{ij} = \frac{1}{c} \pi(j), \label{tttt}
\end{equation}
where the last equality is from $\sum_{k \in S_j}\alpha_{kj} \!=\! 1$. Since this holds for all $j \!\in\! \N$ and $\tp_{ij} \!=\! P_{ij}$ for $i,j \!\in\! \N$, we have $\ppi \pp \!=\! (1/c) \ppi$ with $1/c \!>\! 0$. In view of the Perron-Frobenius theorem, $\ppi$ is the only positive eigenvector of $\pp$ and is the leading left eigenvector, corresponding to the maximal eigenvalue $1/c$.

On the other hand, as explained in Section~\ref{subse:qsd-def}, from the Perron-Frobenius theorem we have the existence of a unique QSD $\nnu$ of the transient chain $\p$ that verifies $\nnu \pp \!=\! \lambda \nnu$, with $\lambda$ in (\ref{eq:lambda}). Therefore, the unique QSD $\nnu$ must coincide with our desired probability distribution $\ppi$, with $\lambda \!=\! 1/c \in (0,1)$.
\end{proof}

As a special case, we can set $\alpha_{kj} \!=\! 1/|S_j| \!=\! [\sum_{k\in \N} \idc_{\{Q_{kj}>0\}}]^{-1}$ for all $k \!\in\! S_j$, and then have the following expression for $\gamma_{ij}$.
\begin{corollary}
\label{cor:special-qsd}
Consider
\begin{equation}
\gamma_{ij} \propto \frac{\pi(j)}{\pi(i)} \cdot \lt[Q_{ij}|S_j| \rt]^{-1} =: b_{ij},
\label{r-qij}
\end{equation}
such that $\gamma_{ij}\leq 1$ for each $(i,j)$ with $Q_{ij}\!>\!0$. Then, the QSD $\nnu$ of $\p$ becomes identical to the target distribution $\ppi$.
\end{corollary}

While the probability distribution $\{\alpha_{\cdot j}\}$ on the support $S_j$ can be made arbitrary as shown in Theorem~\ref{thm1:general-qsd}, we consider $\gamma_{ij}$ in (\ref{r-qij}) throughout the rest of the paper for its simplicity. Note that the acceptance probability $\gamma_{ij}$ in (\ref{r-qij}) requires $\pi(j)/\pi(i)$, which is essentially the desired target distribution up to multiplicative constants, the proposal probability $Q_{ij}$ (clearly known), and $|S_j|$ (the number of possible transitions into node $j$). These are all local information around $i,j$ when the chain attempts to make a move from $i$ to $j$.

As a special case, we consider $\q \!=\! \q^{\textsf{srw}}$, the transition matrix of the simple random walk on a given directed graph $\GG$. In this case, $Q_{ij} \!=\! A_{ij}/d_i^+$ and $|S_j| \!=\! d_j^-$, the in-degree of node $j$. This is the only required local information, not exact knowledge (identities) of the nodes whose outgoing edges point to each node. Such in-degree can be obtained without much difficulty in online social networks having follower-followee relationship such as Google++, Twitter, and Instagram, since the in-degree is usually available as part of a user profile, which is simply the number of followers~\cite{Vesdapunt-CIKM16}. It is worth noting that the social networks are often much more stringent on retrieving the IDs of inlinks and outlinks of each user than looking up user profiles~\cite{Vesdapunt-CIKM16}. We shall even demonstrate in Section~\ref{subse:online} that an online estimation of the in-degree of each node would suffice for our proposed NMMC method, even if such in-degree information may not be available.

Thanks to Theorem~\ref{thm1:general-qsd} and Corollary~\ref{cor:special-qsd}, without knowing the stationary distribution $\mmu$ of the simple random walk, which is \emph{unknown} in general, we can achieve any desired distribution $\ppi$ by estimating the QSD of the transient chain $\p \!=\! [P_{ij}]$ on $\N^+$ with $P_{ij} \!=\! \gamma_{ij}Q_{ij} \!=\! \gamma_{ij}A_{ij}/d_i^+$, where
\begin{equation}
\gamma_{ij} \propto b_{ij} = \frac{\pi(j)}{\pi(i)} \cdot \frac{d_i^+}{d_j^-},
\label{eq:gamma-SRW}
\end{equation}
with some scaling constant $c$ such that $b_{ij} / c  \leq 1$ for all $(i,j)\!\in\!\E$. Then, if we want to achieve the uniform distribution $\ppi$ on $\GG$, i.e., $\pi(i) \!=\! 1/n$ for all $i \!\in\! \N$, then $c \!\geq\! \max_{(i,j)\in \E} d_i^+ / d_j^-$. Since $d_j^- \!\geq\! 1$ for any $j\in\N$ (since the graph is irreducible), one easy choice would be $c \!=\! \max_{i\in \N} d_i^+$, i.e., the maximum out-degree of the graph. When such a global constant is known in advance, our Theorem~\ref{thm1:general-qsd}, or Corollary~\ref{cor:special-qsd}, applies as is. In case that such constant is not available, later in Section~\ref{subse:our-algo}, we will demonstrate that a real-time (dynamic) estimation of the global constant can also be safely used.

\subsection{Mapping EVC to QSD}
\label{subse:EVC-mapping}

In addition to achieving any desired distribution $\ppi$ as a QSD, we demonstrate that we can go one step further to estimate the EVC $\xx$ of the adjacency matrix $\A$ of a given directed graph $\GG$ in a similar fashion. As explained in Section~\ref{subse:EVC}, by definition, the EVC $\xx$ is the (normalized) left leading eigenvector of the adjacency matrix $\A$ with the corresponding maximal eigenvalue $\lambda(A)$.
Then, by simply adjusting $\gamma_{ij}$ using local information, we can equate the resulting QSD with the EVC $\xx$ of $\A$, as shown next.

\begin{corollary}
\label{cor:mapping-EVC}
If we set $\gamma_{ij} \!=\! A_{ij}/(cQ_{ij})$ for any $(i,j)$ pair with $Q_{ij}\!>\!0$, where $c$ is some suitable scaling constant such that $\gamma_{ij} \!\leq\! 1$, then the QSD $\nnu$ of the transient $\p$ with elements $P_{ij} \!=\! Q_{ij}\gamma_{ij} \!=\! A_{ij}/c$ for $i,j \!\in\! \N$, becomes the EVC $\xx$ of the adjacency matrix $\A$.
\end{corollary}
\begin{proof}
Using $\gamma_{ij} \!=\! A_{ij}/(cQ_{ij})$, and from (\ref{eq:qsd}) with $P_{ij}$ replaced by $\tp_{ij}$, we observe that
\begin{equation}
\sum_{i\in \N} \nu(i) A_{ij} = (\lambda c) \cdot \nu(j) \; \Longleftrightarrow \; \nnu \A = (\lambda c) \nnu,
\label{qsd-evc}
\end{equation}
with $\lambda$ given in (\ref{eq:lambda}), i.e., the resulting QSD $\nnu$ is the left eigenvector of the adjacency matrix $\A$ with $\nu(i) \!>\! 0$, corresponding to the positive eigenvalue $\lambda c$. The assertion then follows from the Perron-Frobenius theorem.
\end{proof}

Corollary~\ref{cor:mapping-EVC} says that the well known EVC $\xx$ of the adjacency matrix $A$ of a given directed graph $\GG$ can now be mapped to the unique QSD $\nnu$ of a transient Markov chain $\p$ driven by a general nonreversible chain $\q$ using only local information for $\gamma_{ij}$, and that the corresponding largest eigenvalue $\lambda(\A) \!=\! \lambda c$ becomes the inverse of the epidemic threshold of the given graph  (See Section~\ref{subse:EVC}.) For example, when the simple random walk is used, i.e., $Q_{ij} \!=\! Q^{\textsf{srw}}_{ij} \!=\! A_{ij}/d_i^+$, we see that $\gamma_{ij} \!=\! d_i^+/c$, where $c \!=\! \max_{i\in\N} d_i^+$ can be set as the maximum out-degree of the graph.

So far, we have shown how to map any desired distribution $\pi$ and the EVC $\xx$ on a directed graph onto the QSD of some transient Markov chain $\p$, which is obtained from any ergodic, nonreversible chain $\q$, with the simple random walk as an example. Now, the question is how to efficiently simulate or estimate the given QSD, or to obtain samples from the given QSD, via the same underlying nonreversible chain $\q$ used for graph exploration. In the next section, we present several methods to simulate or estimate the QSD and complete our NMMC method using a set of non-Markovian random walks (or crawling agents) to accomplish the goal.

\section{NMMC: QSD Estimation}\label{se:qsd-algo}

\subsection{Non-Markovian Random Walks to Approximate QSD}
\label{subse:qsd-simu-standard}

We are often interested in calculating the expectation of a function $f: \N \!\to\! \R$ defined on the state space $\N$ with respect to a desired distribution $\ppi$, i.e., $\Ex_{\ppi}\{f\}\!=\!\sum_{i\in\N} f(i) \pi(i)$, by constructing an ergodic Markov chain $\{X_t\}$ that has the desired $\ppi$ as its equilibrium distribution. This is one of the main applications of the MCMC methods. By the ergodic theorem, we have
\begin{equation}
\lim_{t\to\infty} \frac{1}{t}\sum_{k=1}^{t} f(X_k) = \Ex_{\ppi}\{f\}  \quad \mbox{a.s.}, \label{ergodic}
\end{equation}
for any initial distribution for $X_0$. Thus, one can directly utilize this defining relationship by simulating the ergodic chain $\{X_t\}$ for a long time and taking the time average in order to estimate or approximate $\Ex_{\ppi}\{f\}$.

In case of the QSD $\nnu$ of a transient Markov chain $\{X_t\}$ on $\N^+ \!=\! \N\cup\{0\}$, we would also consider the long-time averages to estimate $\Ex_{\ppi}\{f\}$ driven by the desired QSD $\nnu \!=\! \ppi$ (or the QSD $\ppi$ itself) in view of
\begin{equation}
  \lim_{t\to\infty} \Ex\lt\{f(X_t) ~|~ X_0 \sim \mmu_0, X_t \neq 0 \rt\} = \Ex_{\ppi}\{f\}, \label{QLD2}
\end{equation}
which is from (\ref{def-QLD}) and the identity between QSD and QLD in the limit for any initial distribution $\mmu_0$. Unfortunately, however, this does not help much. In order to rely on the regularity inherent in the limiting relationship in (\ref{QLD2}), we would have to generate many trajectories for $\{X_t\}$ for a long time, only to find out that most of them get absorbed or trapped to the absorbing site `0' since $\pr\{X_t \!\neq\! 0\}$ goes to zero exponentially fast, and we are left with very few trajectories that survive for a long time and possibly contribute to useful samples for the estimation. In short, this brute-force approach simply would not work for numerical estimation of the QSD in general.

In the mathematics literature, one practical approach in approximating the QSD is based on a historical urn-based process or a reinforced random walk on the state space. Although such a process loses the Markovian property due to the dependency on its own history or trajectory for the next move, using the past trajectory of the process is amenable to practical implementation. First, observe that from (\ref{eq:qsd}) and (\ref{eq:lambda}), the QSD $\nnu$ of the transient chain $\p$ satisfies
\begin{equation}
\nu(j) = \sum_{i\in \N} \nu(i) \lt[P_{ij} + P_{i0} \nu(j)\rt], \quad j \in \N.
\label{eq:qsd-redist}
\end{equation}
This can be interpreted as an ergodic Markov chain on the state space $\N$ with its transition matrix $\p^{\nnu} \!=\! [P^{\nnu}_{ij}]$ given by
\begin{equation}\label{eq:redist}
  P^{\nnu}_{ij} := P_{ij} + P_{i0}\nu(j) = \tilde{P}_{ij} + P_{i0}\nu(j), \quad i,j \in \N.
\end{equation}
That is to say, assuming that $\nnu$ is known, the process evolves according to $P_{ij} \!=\! \tilde{P}_{ij}$ until it gets absorbed, upon which it instantaneously `redistributes' to a new node in $\N$ according to $\nnu$. Let $\Phi(\nnu)$ be the invariant (stationary) distribution of this ergodic Markov chain with its transition matrix $\p^{\nnu} \!=\! [P^{\nnu}_{ij}]$, indexed with $\nnu$. Then, we can see that the QSD $\nnu$ in (\ref{eq:qsd-redist}) becomes precisely the fixed point of the mapping from $\nnu$ to $\Phi(\nnu)$. In other words, $\nnu$ is the QSD for $\p$ if and only if $\Phi(\nnu) \!=\! \nnu$. The construction of such an ergodic chain is, however, practically infeasible in estimating the target QSD $\nnu \!=\! \ppi$, since lack of the entire topological information prohibits direct access into the state space $\N$ according to $\ppi$, or $\ppi$ is simply unknown, e.g., for EVC of $\A$.

The above observation nonetheless motivated one of the earliest efforts in approximating the QSD of a given transient chain by Aldous \textit{et al.}~\cite{Aldous88} to replace $\nnu$ (unknown) in~(\ref{eq:qsd-redist}) with the (known) empirical distribution of the history of the original process.\footnote{In the literature, there are some study on Fleming-Viot (FV) processes and the evolution of the associated spatial empirical measures converging to the QSD~\cite{Min-Kang04,Asselah2011-JAP,GK-ip,QSD-survey,Min-Kang13,Groisman13}. It is essentially to replace the unknown $\nnu$ in~(\ref{eq:qsd-redist}) by the `spatial' occupancy measure of a large number of particles undergoing identical dynamics with mutual interaction upon particle absorption. To be more precise, in our graph setting, it can be described as follows. There are $K$ particles or agents, each of which moves according to $\p$ of the transient Markov chain on $\N^+$. When any of the particles gets absorbed by `0', it is instantaneously relocated to one of the remaining particles, chosen uniformly at random, on $\N$. Then, the empirical measure of the positions of the particles converges to the QSD as \emph{both} the number of particles $K$ and time $t$ go to infinity. The number of particles required under the FV method, however, should be much larger than the size of the graph (i.e., $K \!\gg\! n$), making the FV method impractical to our scenario. Thus, we do not consider the FV method in this paper.} To be precise, we define its historical empirical distribution $\hmmu_t$ by time $t$ of the underlying stochastic process $\{Z_t\}_{t\geq 0}$ as
\begin{equation} \label{hist-emp}
\hmmu_{t} = \frac{1}{t\!+\!1}  \sum_{k=0}^{t}  \delta_{Z_{k}},
\end{equation}
with $\hmmu_{t} \!=\! [\hmu_t(1),\cdots,\hmu_t(n)]$, where $\hmu_t(i)$ is simply given by
\begin{equation*}
  \hmu_t(i) = \frac{1}{t \!+\! 1}  \sum_{k=0}^{t}  \delta_{Z_{k}}(i) = \frac{1}{t\!+\!1} \sum_{k=0}^{t}\idc_{\{Z_{k}=i\}},
\end{equation*}
which records the relative visit frequency of $\{Z_k\}$ to the state $i \!\in\! \N$ over the first $t\!+\!1$ time steps. Note that $\hmu_t(i)$ is random while $\sum_{i\in\N}\hmu_t(i) \!=\! 1$ for each $t$, i.e., $\hmmu_{t}$ is in fact a probability distribution (vector) on $\N$ for each $t$.

Consider the natural filtration $\{\F_{t}\}_{t \geq 0}$
associated with the historical process $\{Z_{t}\}$, i.e., $\F_t := \sigma\{Z_{k}\,|\, k \leq t \}$. We then have
\begin{equation*}\label{def-hist-proc}
\pr\{Z_{t+1} = j\,|\, \F_{t}\} = \sum_{i \in \N}
P_{ij}^{\hmmu_t} \cdot \idc_{\{Z_{t}=i\}}
\end{equation*}
for any given $j \!\in\! \N$, where $P_{ij}^{\hmmu_t}$ is from (\ref{eq:redist}) with $\nnu$ replaced by $\hmmu_t$. That is to say, given the full history $\F_t$ of the process $\{Z_t\}$ up to time $t$, the process will choose its next state $Z_{t+1}$ by the original
transition matrix $\p$ if not getting absorbed. When it hits the absorbing state `0', it will instantaneously jump to another state in $\N$ with probability determined
by the historical empirical distribution $\hmmu_t \!=\! \frac{1}{t+1} \sum_{k = 0}^{t} \delta_{Z_{k}}$, observed by the current time instant $t$ and thus readily available to us.

This method has then been recently generalized in~\cite{Benaim15,Benaim16} by utilizing the stochastic approximation theory and its induced differential equations developed for reinforced random walks~\cite{Benaim97,Pemantle07}, allowing different forms of the historical empirical distribution $\hmmu_t$ that substitutes $\nnu$ in (\ref{eq:redist}). Specifically, the distribution $\hmmu_t$ is generalized as
\begin{equation} \label{hist-emp2}
\hmmu_{t} = \frac{\sum_{k=0}^{t} w_k\delta_{Z_{k}}}{\sum_{k=0}^{t}  w_k} = \sum_{k=0}^{t} \eta_t(k) \delta_{Z_{k}},
\end{equation}
with
\begin{equation} \label{time-weight}
\eta_t(k) := \frac{w_k}{\sum_{k'=0}^{t}w_{k'}}, \quad k = 0, 1, 2, \ldots, t.
\end{equation}
Here, $\hmu_t(i)$ indicates the visit frequency to node $i$ by $\{Z_k\}$ up to time $t$, each of which is \emph{weighted} by $\eta_t(k)$. This makes the redistribution dynamics of the process $\{Z_t\}$ different, while the other remains the same as above. When the process hits the absorbing state at time $t\!+\!1$, it relocates directly to one of the previously-visited nodes according to $\hmmu_{t}$ in (\ref{hist-emp2}). This can also be interpreted as choosing $Z_{t+1} \!:=\! Z_k$ with probability $\eta_t(k)$ in (\ref{time-weight}), $k \!=\! 0,1,\ldots,t$. Setting $w_k \!=\! 1$, or $\eta_t(k) \!=\! 1/(t\!+\!1)$, for all $k$ in (\ref{hist-emp2}) reduces to (\ref{hist-emp}).

The resulting process $\{Z_t\}$ is clearly non-Markovian and history-dependent. It is shown in~\cite{Benaim15,Benaim16} that for a large class of non-decreasing weight functions $w_k$, or sequences $\{\eta_t(t)\}_{t \geq 0}$,
\begin{equation}
\lim_{t \to \infty} \hmmu_{t} = \nnu \ \; \mbox{a.s.}, \ \mbox{and} \ \; \lim_{t \to \infty} \Ex\{\hmmu_{t}\} = \nnu,
\end{equation}
for any arbitrary initial state $Z_0$ with $\hmmu_{0} \!=\! \delta_{Z_{0}}$. Note that the allowed weight function $w_k$ includes, for example, $w_k \!=\! 1$ (constant), $w_k \!=\! k^a$ (polynomial), and $2^{k^b}$ (sub-exponential), where $a \!>\!1$ and $b \!\in\! (0,1)$. These weight functions make the sequence $\{\eta_t(t)\}_{t \geq 0}$ non-increasing in $t$ and satisfy the required conditions in~\cite{Benaim15,Benaim16}. We adopt this history-based process with reinforcement to estimate the desired QSD $\nnu \!=\! \ppi$ as an integral part of our NMMC method.

The basic operation of our NMMC method to achieve any desired distribution $\ppi$ on a directed graph $\GG$ or to estimate the EVC $\xx$ of the adjacency matrix $\A$ can be summarized as follows. Under our NMMC method, a set of non-Markovian random-walk agents are launched to crawl the graph, and each of them independently moves over $\GG$ by choosing the next node $Z_{t+1}$ at the current node $Z_t \!=\! i$ according to the transition probability $P_{ij} \!=\! Q_{ij}\gamma_{ij}$, before it gets absorbed and relocated. In other words, a move from node $i$ to $j$ is \emph{proposed} with probability $Q_{ij}$ following a directed edge $(i,j)$ on $\GG$. This move is then \emph{accepted} with probability $\gamma_{ij}$ as given in Theorem~\ref{thm1:general-qsd} and Corollary~\ref{cor:special-qsd} for achieving $\ppi$ (or Corollary~\ref{cor:mapping-EVC} for estimating the EVC $\xx$). Otherwise, it is \emph{rejected} in which case the agent is instantaneously relocated (or redistributed) to one of the previously-visited nodes according to the historical empirical distribution $\hmmu_{t}$ in (\ref{hist-emp2}). Figure~\ref{fig:qsd} illustrates the overall operation of our NMMC method.

We note that the history-dependent/reinforced random walks are quite relevant and amenable to practical implementation on most popular online social networks, since the crawler (e.g., a simple random walk) continues to collect the ID (e.g., URL) of each node visited as it explores the graph and can always relocate/redistribute to one of the previously-visited nodes, while keeping track of which nodes have been visited up until now. The historical empirical distribution governs not only the dynamics of each crawling agent, but it also enables us to estimate the target QSD $\nnu \!=\! \ppi$. Therefore, our NMMC method provides the same benefit and fruitful denouement that we cherished from the standard MCMC methods that are rooted in the ergodic theorem in (\ref{ergodic}).

\subsection{Performance Enhancements}
\label{subse:our-algo}

So far, we have demonstrated how any ergodic (inherently nonreversible) Markov chain $\q$ with unknown $\mmu$ can be used to achieve a given desired distribution $\ppi$ on a directed graph $\GG$ and to estimate the EVC $\xx$ of the adjacency matrix $\A$. It is essentially achieved by combining our mapping $\ppi$ or $\xx$ to the QSD $\nnu$ of a related transient Markov chain and the recent development in the mathematics literature as to estimating the QSD via a set of non-Markov chain, history-dependent random walks.

As shall be seen later in Section~\ref{se:simu}, the direct application of the non-Markovian random walks with historical empirical distribution $\hmmu_{t}$ may not be satisfactory due to a possible imbalance between their diffusion and redistribution, both of which are the necessary integral components. In this section, we present several ways to make the crawling walks achieve a right balance between their diffusion and redistribution, which translates into speeding up the convergence of $\hmmu_{t}$ to the desired QSD $\nnu \!=\! \ppi$ (or $\nnu \!=\! \xx$).

\vspace{1mm}
\noindent \textbf{Initial empirical measure $\hmmu_{0}$.} We next revisit the historical empirical distribution $\hmmu_{t}$ in (\ref{hist-emp2}). Letting $W_t \!:=\! \sum^t_{k=0} w_k$ for a given positive weight sequence,
 $\{ w_{k} \}_{k \ge 0}$, we observe that (\ref{hist-emp2}) can be written as, for $t \geq 1$,
\begin{equation}
  W_t \hmmu_{t}  = W_{t-1} \hmmu_{t-1}  + w_t \delta_{Z_t},
\end{equation}
which then leads to
\begin{align}
  \hmmu_{t}  &=  \frac{W_{t-1}}{W_t}\hmmu_{t-1} + \frac{w_t}{W_t} \delta_{Z_t} = \lt(1 - \frac{w_t}{W_t}\rt) \hmmu_{t-1} + \frac{w_t}{W_t} \delta_{Z_t} \nonumber \\
  &= \lt[1 - \eta_{t}(t)\rt] \hmmu_{t-1} + \eta_{t}(t)\delta_{Z_t}, \label{recursive}
\end{align}
from the identities $W_t \!=\! W_{t-1} \!+\! w_t$ and $\eta_{t}(t) \!=\! w_t/W_t$. This implies that the historical empirical distribution $\hmmu_{t}$ can be computed recursively based on (\ref{recursive}). Although the convergence of $\hmmu_{t}$ to the QSD $\nnu$ was shown for any arbitrary initial state $Z_0$ with $\hmmu_{0} \!=\! \delta_{Z_{0}}$ in~\cite{Benaim15,Benaim16}, its proof based on the stochastic approximation theory and the induced differential equations in fact does not require the initial distribution $\hmmu_{0}$ in (\ref{recursive}) to be in the form of $\delta_{Z_{0}}$, where $\hmmu_{0}$ has a value of 1 at the element $Z_0$ and zero at all other elements, but can be made arbitrary as long as $\hmmu_{0}$ is a probability distribution (vector) on $\N$.

\begin{algorithm}[t!]
\caption{Non-Markovian Monte Carlo (at time $t \ge 1$)}\label{alg1}
\begin{algorithmic}[1]

\REQUIRE $\q$, $\nnu$, $\hmmu_{0}$, and $w_k$

\STATE $i \leftarrow Z_{t-1}$ \qquad \ // the node at time $t\!-\!1$

\STATE Select node $j$ with probability $Q_{ij}$

\STATE Generate $u \sim U(0,1)$

\IF{$u \leq \gamma_{ij}$}

\STATE $Z_{t} \leftarrow j$ \qquad // the proposed move is accepted

\ELSE

\STATE /* the agent is relocated according to $\hmmu_{t}$ */

\STATE Select node $k$ with probability $\hmu_{t}(k)$

\STATE $Z_{t} \leftarrow k$

\ENDIF

\STATE /* update the historical empirical distribution $\hmmu_t$ */

\STATE $ \hmmu_t \leftarrow \lt[1 - \eta_{t}(t)\rt] \hmmu_{t-1} + \eta_{t}(t)\delta_{Z_t} $

\end{algorithmic}
\end{algorithm}

This freedom on the choice of $\hmmu_{0}$ gives us a great degree of flexibility. Clearly, the choice of $\hmmu_{0}$ affects the whole evolution of the process $\{Z_t\}$ for all time $t$ via redistribution mechanism, and thus can be made to our advantage. For example, a judicious choice of $\hmmu_{0}$ can translate into a situation where a crawling agent $\{Z_t\}$ can start with an `already-established' history to some degree or more generally, with any \emph{arbitrary (normalized) visit counts} assigned to a subset of nodes on the graph to begin with. Thus, our proposed NMMC method can be generalized with arbitrary $\hmmu_{0}$ and its overall procedure is now summarized in Algorithm~\ref{alg1}. Here the proposed chain $\q$, the initial empirical measure $\hmmu_{0}$, and the (unnormalized) weight function $w_k$ should be specified as inputs into Algorithm~\ref{alg1}. Note that the initial position of a crawling agent $Z_0$ is randomly chosen according to $\hmmu_{0}$.

\vspace{1mm}
\noindent \textbf{Dynamic and adaptive choice of parameters.} Our proposed NMMC method is mathematically guaranteed but its speed of convergence depends on the form of the acceptance probability $\gamma_{ij}$, along with $Q_{ij}$ and the target QSD, which interweaves diffusion and redistribution -- two seemingly opposite (yet integral) components of the method. As we have seen from Section~\ref{se:qsd-mapping}, the specific form of $\gamma_{ij}$ is different depending on the choice of $\q \!=\! [Q_{ij}]$ and what to estimate via the target QSD, i.e., the target distribution $\ppi$ or the EVC $\xx$ of the adjacency matrix $\A$. It is particularly desirable to maintain some reasonable values of $\gamma_{ij}$ over almost every $(i,j) \in \E$, or more precisely, transition pair $(i,j)$ with $Q_{ij} \!>\! 0$, so as to strike the right balance between diffusion and redistribution to achieve faster convergence.

We take a closer look at the acceptance probability $\gamma_{ij}$ as a function of $Q_{ij}$ and the target QSD. Recall that from Theorem~\ref{thm1:general-qsd} and Corollary~\ref{cor:special-qsd}, it is given by $\gamma_{ij} \!=\! b_{ij}/c$ in (\ref{r-qij}) for achieving (or sampling from) a given distribution $\ppi$ on the directed graph $\GG$, and the corresponding normalizing constant $c$ can be obtained based on $c \!=\! \max_{i,j} b_{ij}$ (the smallest possible value for $c$). Similarly, for estimating the EVC $\xx$ of the adjacency matrix $\A$, it can be chosen as $c \!=\! \max_{i,j} 1/Q_{ij}$ for $\gamma_{ij} \!=\! A_{ij}/(cQ_{ij})$ from Corollary~\ref{cor:special-qsd}. The global constant $c$ over the graph, however, may not be available in advance or could take a long time to estimate, although its real-time (dynamic) estimation can be incorporated as shall be explained shortly. Furthermore, the resulting values of $\gamma_{ij}$ may vary over $(i,j)\!\in\!\E$ and they can also be small for some $(i,j)$ pairs. Such small acceptance probabilities in turn make the agent get absorbed to `0' and redistributed to the already visited nodes. The redistribution mechanism is of course an important and necessary component of our algorithm to shape up the historical empirical distribution $\hmmu_t$ toward the target QSD. But, if it is too strong (very small $\gamma_{ij}$), the agent may virtually get stuck on a small set of visited nodes before diffusing out to a potentially vast unvisited portion of the graph (state space) for exploration.

\begin{algorithm}[t!]
\caption{Dynamic Non-Markovian Monte Carlo}\label{alg2}
\begin{algorithmic}[1]
\STATE /* This pseudocode only replaces lines 3--10 of Algorithm~\ref{alg1} */

\STATE Generate $u_1, u_2 \sim U(0,1)$

\IF{$u_1 \leq p$ and $c_{t-1} < b_{ij}$}

\STATE $c_t \leftarrow b_{ij}$  \qquad // $c_t$ is updated by $b_{ij}$

\ELSE

\STATE $c_t \leftarrow c_{t-1}$ \quad \; // $c_{t-1}$ is kept

\ENDIF

\IF{$u_2 \leq \min\{1,b_{ij}/c_t\}$}

\STATE $Z_{t} \leftarrow j$  \qquad \; // the proposed move is accepted

\ELSE

\STATE Select node $k$ with probability $\hmu_{t}(k)$

\STATE $Z_{t} \leftarrow k$  \qquad \ // the agent is redistributed

\ENDIF

\end{algorithmic}
\end{algorithm}

To remedy this possibly slow diffusion problem, we observe that the target QSD $\nnu \!=\! \ppi$ when achieving any desired distribution $\ppi$ is completely controllable to our benefit. To be precise, if it is of primary concern to estimate $\Ex_{\ppi}\{f\} \!=\! \sum_{i \in \N} f(i)\pi(i)$ for a given sampling function $f: \N \!\to\! \R$, which is often the case for `undirected' graph sampling~\cite{WillingerToN09,RibeiroIMC10,GjokaJSAC11,LeeSIGMETRICS12,HardimanWWW13,BrunoSIGMETRICS16}, the target distribution can be \emph{controlled} together with a so-called importance sampling method in MCMC~\cite{Liu04}. With an alternative sampling distribution $\ppi'$, it is to estimate
\begin{equation}
  \Ex_{\ppi'}\lt\{f \frac{\ppi}{\ppi'}\rt\} = \sum_{i \in \N} f(i) \frac{\pi(i)}{\pi'(i)}\pi'(i) = \Ex_{\ppi}\{f\}.
  \label{importance}
\end{equation}
Thus, we can choose the target distribution $\ppi$ along with $\q$ so that it translates into more balanced $\gamma_{ij}$. As shall be demonstrated later in Section~\ref{se:simu}, with $\q \!=\! \q^{\textsf{srw}}$, setting the target $\ppi$ to be proportional to node in-degrees, i.e., $\pi(i) \propto d^-_i, i \!\in\! \N$, turns out to be very effective in achieving fast convergence speed. This improvement is somewhat expected since the simple random walk would visit high in-degree nodes more frequently, which helps the historical empirical distribution $\hmmu$ get closer to this target $\ppi$ quickly.

To further alleviate the slow diffusion problem and also enable online estimation of the normalizing constant $c$, we allow $c$ to be time-varying, say $c_t$, so that it can be learned through graph exploration by the crawling agent. For ease of exposition, we focus on achieving the desired $\ppi$ over the graph, but everything can easily be carried over to estimating the EVC $\xx$ of $\A$. We set the acceptance probability $\gamma_{ij}(t) \! := \! b_{ij}/c_t$, which is updated in a way that $c_t \!:=\! b_{ij}$ at every time $t$ if $c_{t-1} \!<\! b_{ij}$ for
a pair of the current node $i$ and its neighbor $j$ as a potential next node, and $c_t \!:=\! c_{t-1}$ if otherwise. In other words, the agent keeps track of the maximum $c_t$ over all the values of $b_{ij}$ for node pairs that have been discovered up to time $t$. Note that $c_0 \!=\! 1$ initially.

We can take one step further to foster the initial diffusion speed to a greater extent, by intentionally slowing down the growth of $c_t$ for higher values of $\gamma_{ij}(t)$ (being closer to one). Specifically, with probability $p \!>\! 0$, we perform the aforementioned update of $c_t$, while with probability $1 \!-\! p$, the time-varying normalizing constant $c_t \!=\! c_{t-1}$ is kept the same even if $c_{t-1} \!<\! b_{ij}$, in which case the proposed move to $j$ is always accepted. The entire procedure of `dynamic' NMMC method is summarized in Algorithm~\ref{alg2}, which replaces the lines 3--10 of Algorithm~\ref{alg1}. Note that the resulting acceptance probability becomes $\gamma_{ij}(t) \!=\! \min\{1, b_{ij}/c_t \}$ as seen from Algorithm~\ref{alg2}.

While the NMMC method with time-varying $c_t$ is seemingly heuristic, it is in fact well supported by Theorem~\ref{thm1:general-qsd} and our observation on the initial empirical measure $\hmmu_{0}$, i.e., \emph{the theoretical guarantee remains intact}. To see this, we define a stopping time as $t^* \!:=\! \min\{t \!\geq 0\! ~|~ c_t \!=\! c\}$. Since the graph is finite, this $t^*$ is almost surely finite for $p \!>\! 0$. We would run the NMMC method with $c_t$ until $t^*$, and then reset the time clock to zero at $t^*$ and treat the historical empirical distribution $\hmmu_{t^*}$ to be newly set as $\hmmu_{0}$. It then follows from our observation pertaining to $\hmmu_{0}$, along with Theorem~\ref{thm1:general-qsd}, that the historical empirical measure $\hmmu_t$ converges to the given target QSD $\nnu \!=\! \ppi$. In addition to the theoretical guarantees, we next demonstrate that our NMMC method with a suitable choice of the $c_{t}$-updating probability $p$ indeed greatly improves the speed of convergence to the target QSD $\nnu \!=\! \ppi$ by effectively combating the slow-diffusion problem triggered by overly frequent redistribution to already-visited sites.

\begin{figure}[t!]
    \centering
    \vspace{-0mm}
    \hspace{-0mm}
    \subfigure[$\ppi = \uu$; static case ]{\includegraphics[width=0.31\textwidth]{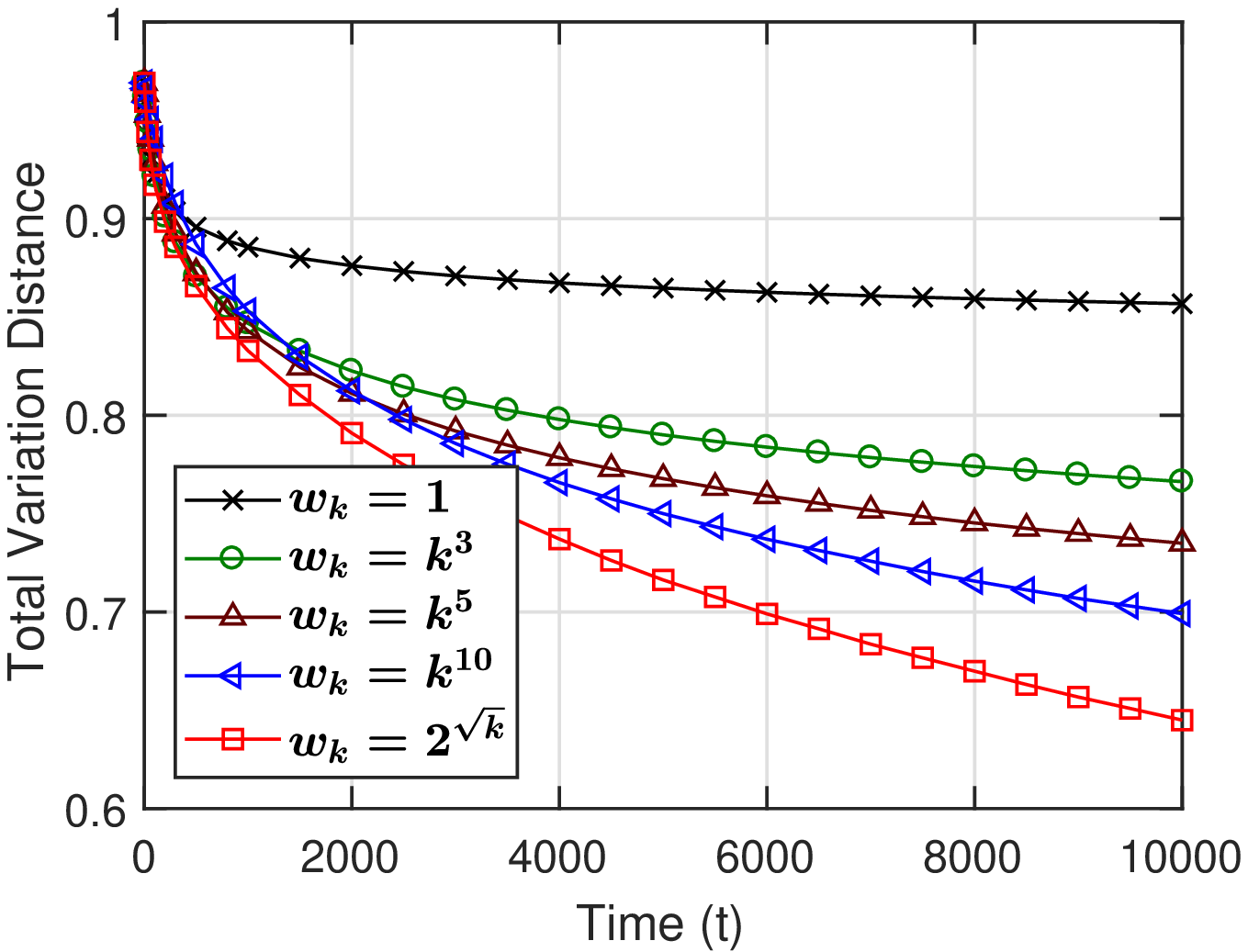}}
    \hspace{1mm}
    \subfigure[$\ppi = \uu$; $p = 1$]{\includegraphics[width=0.31\textwidth]{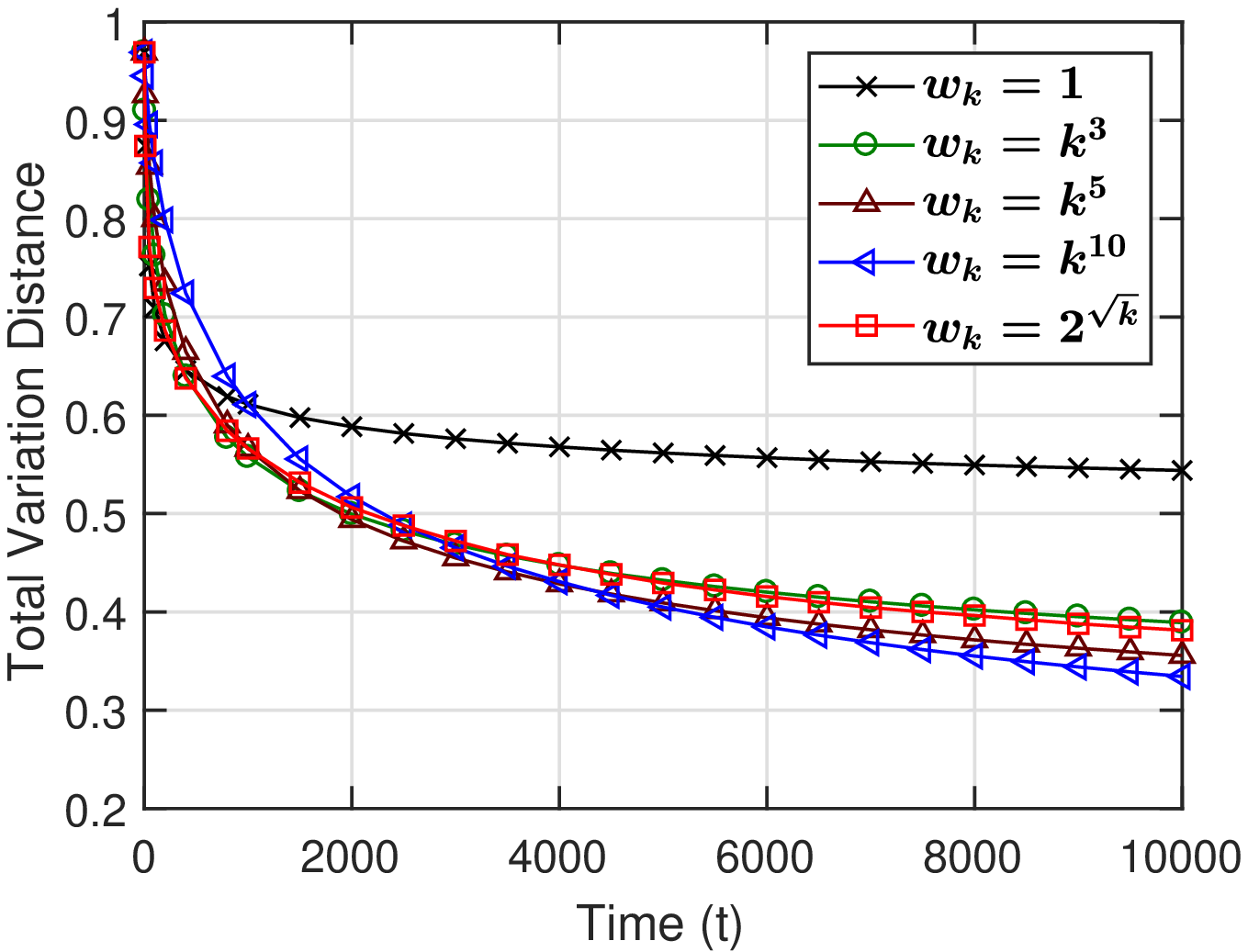}}
    \hspace{1mm}
    \subfigure[$\ppi = \uu$; $p = 0.1$]{\includegraphics[width=0.31\textwidth]{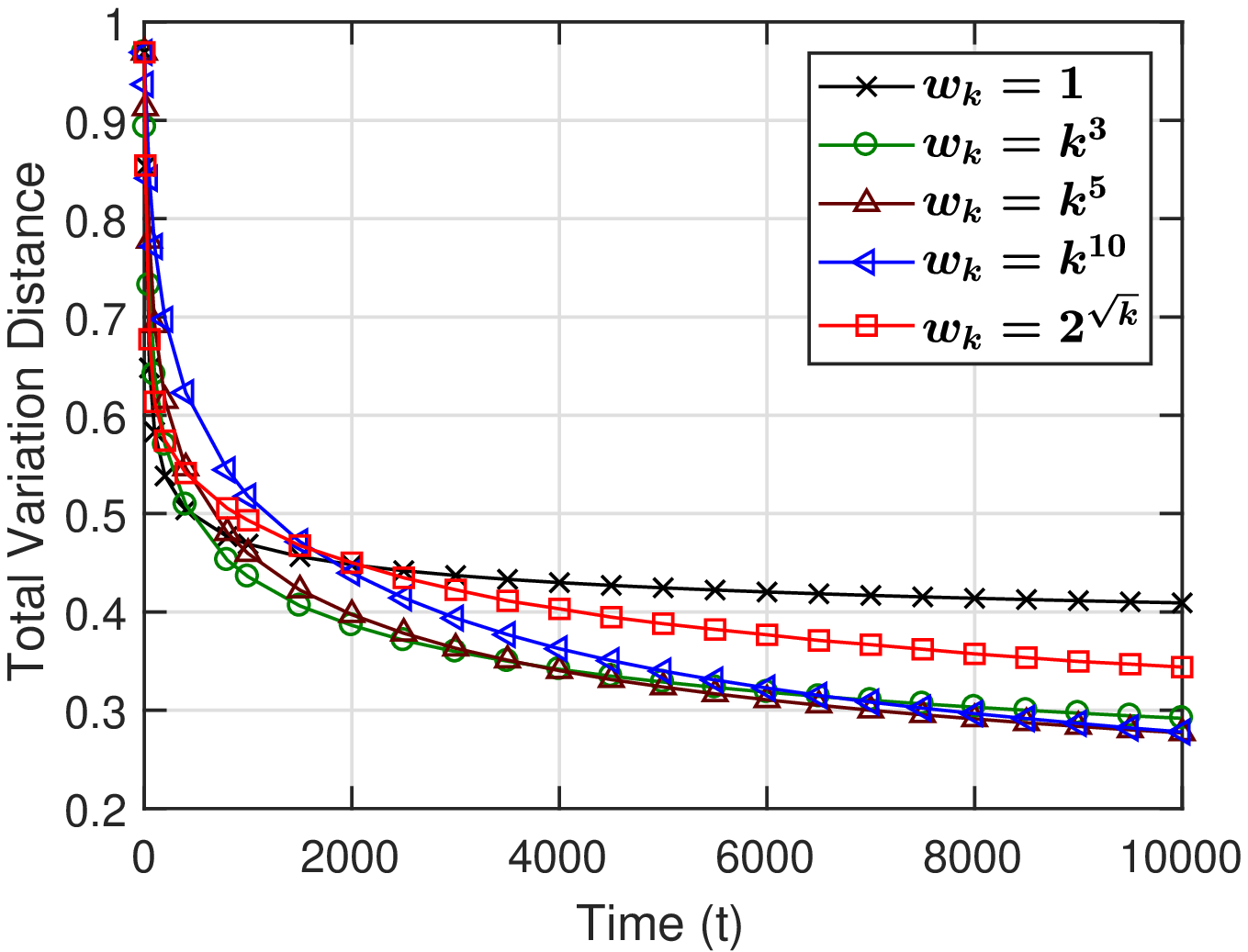}}
    \vspace{-1mm}\\
    \subfigure[$\ppi = \uu$; $p = 0.01$]{\includegraphics[width=0.31\textwidth]{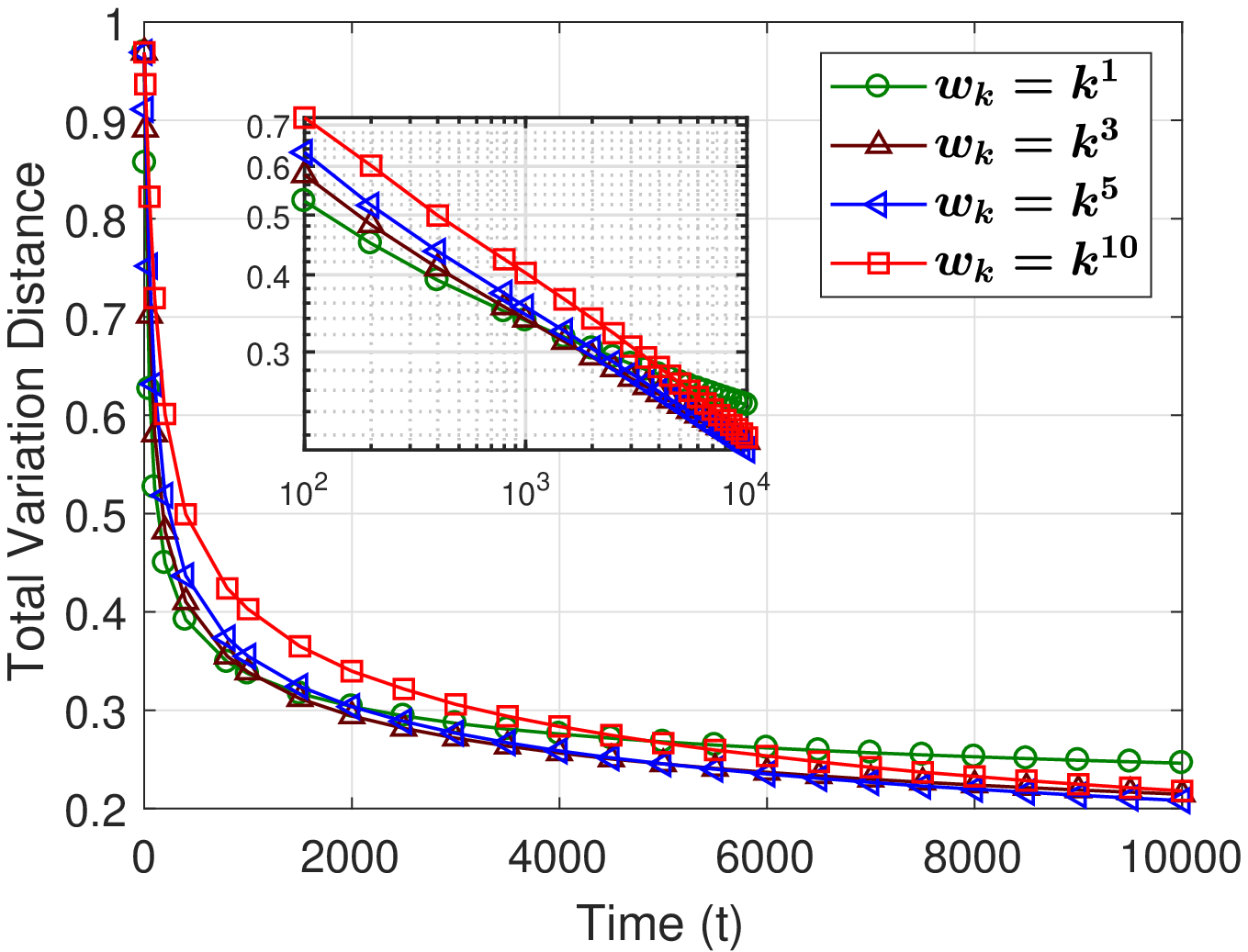}}
    \hspace{1mm}
    \subfigure[$\pi(i) \propto d^-_i$; $p = 0.01$; linear scale ]{\includegraphics[width=0.31\textwidth]{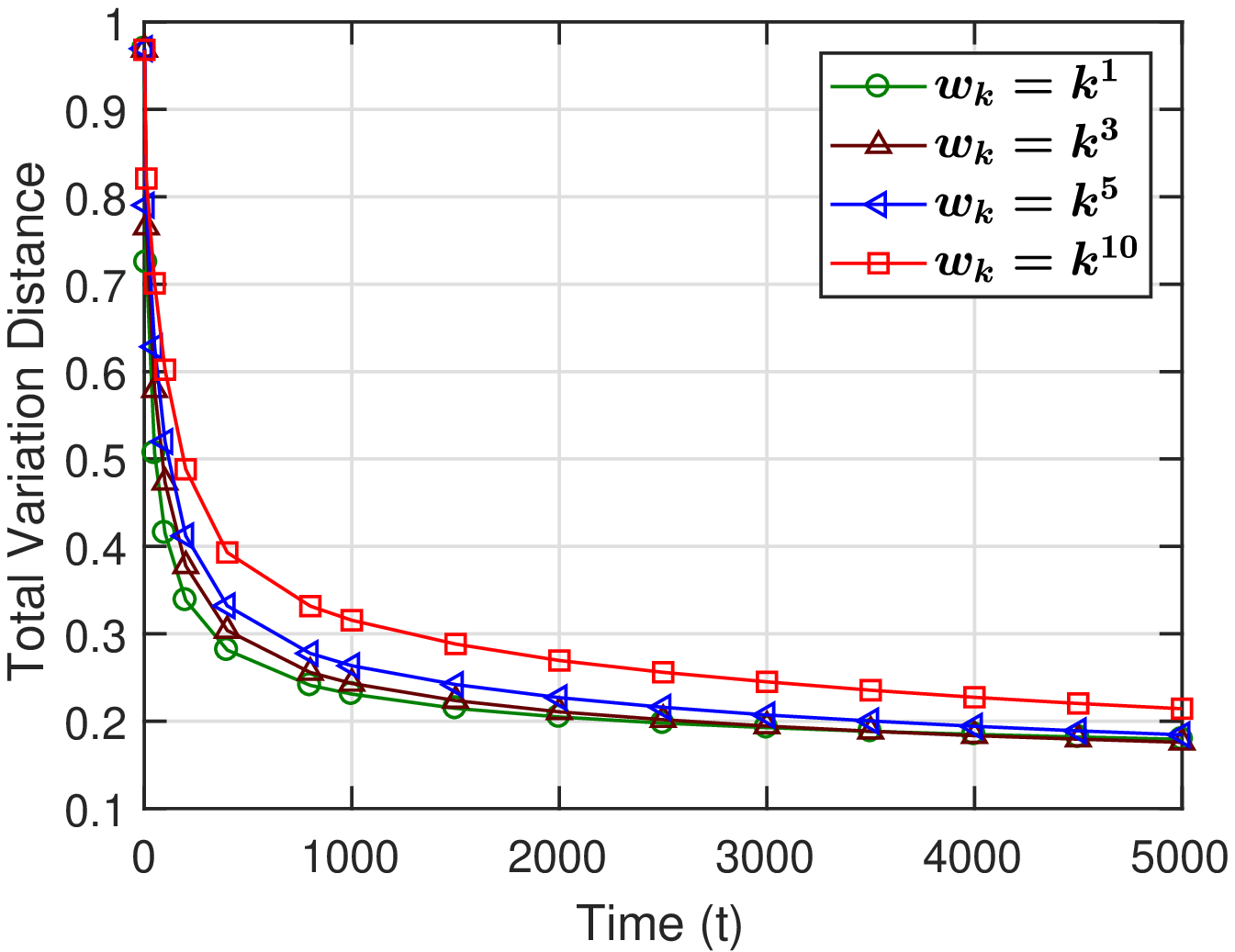}}
    \hspace{1mm}
    \subfigure[$\pi(i) \propto d^-_i$; $p = 0.01$; log-log scale ]{\includegraphics[width=0.3\textwidth]{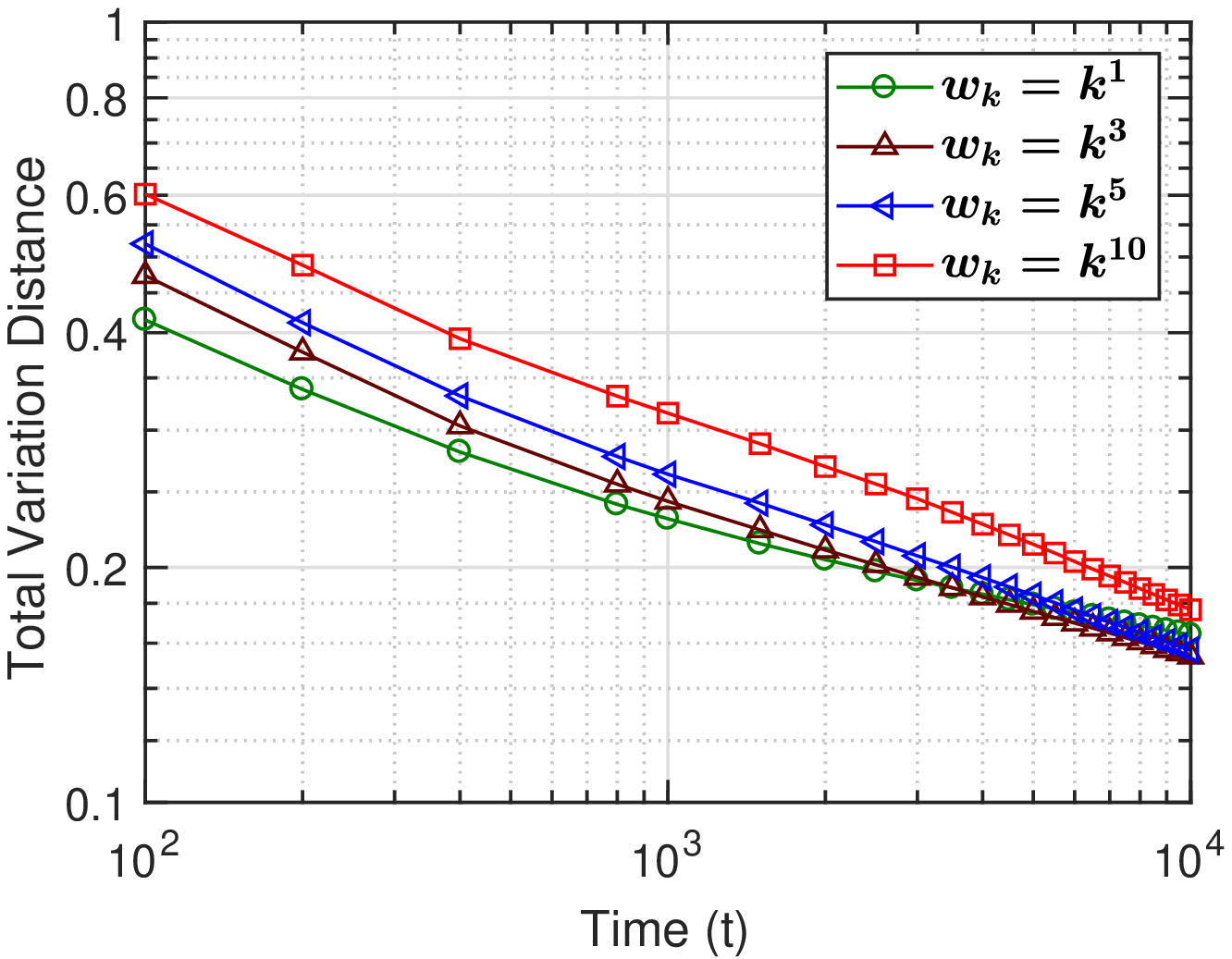}}
    \vspace{-1mm}\\
    \subfigure[EVC $\xx$; $p = 0.01$; linear scale ]{\includegraphics[width=0.31\textwidth]{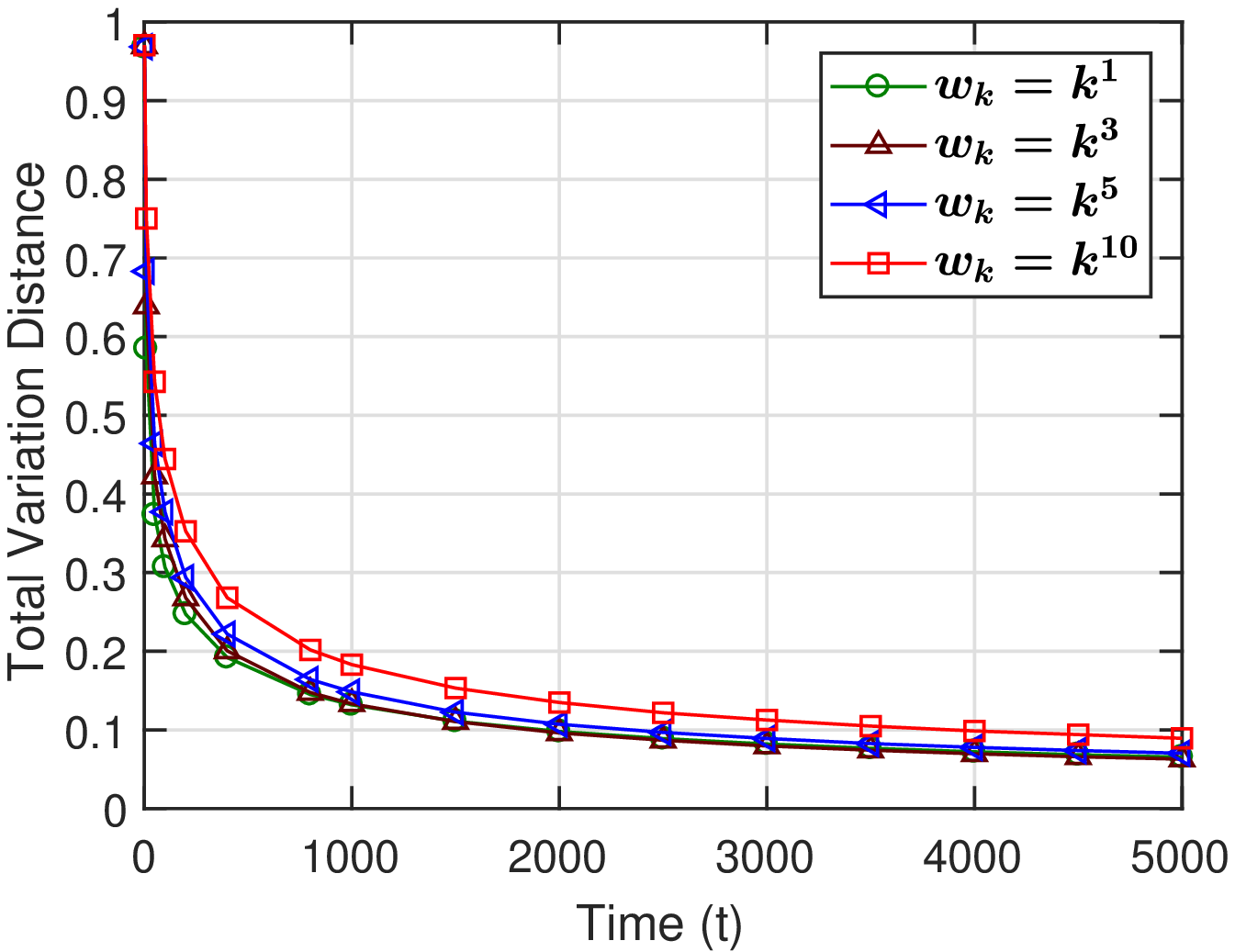}}
    \hspace{1mm}
    \subfigure[EVC $\xx$; $p = 0.01$; log-log scale]{\includegraphics[width=0.31\textwidth]{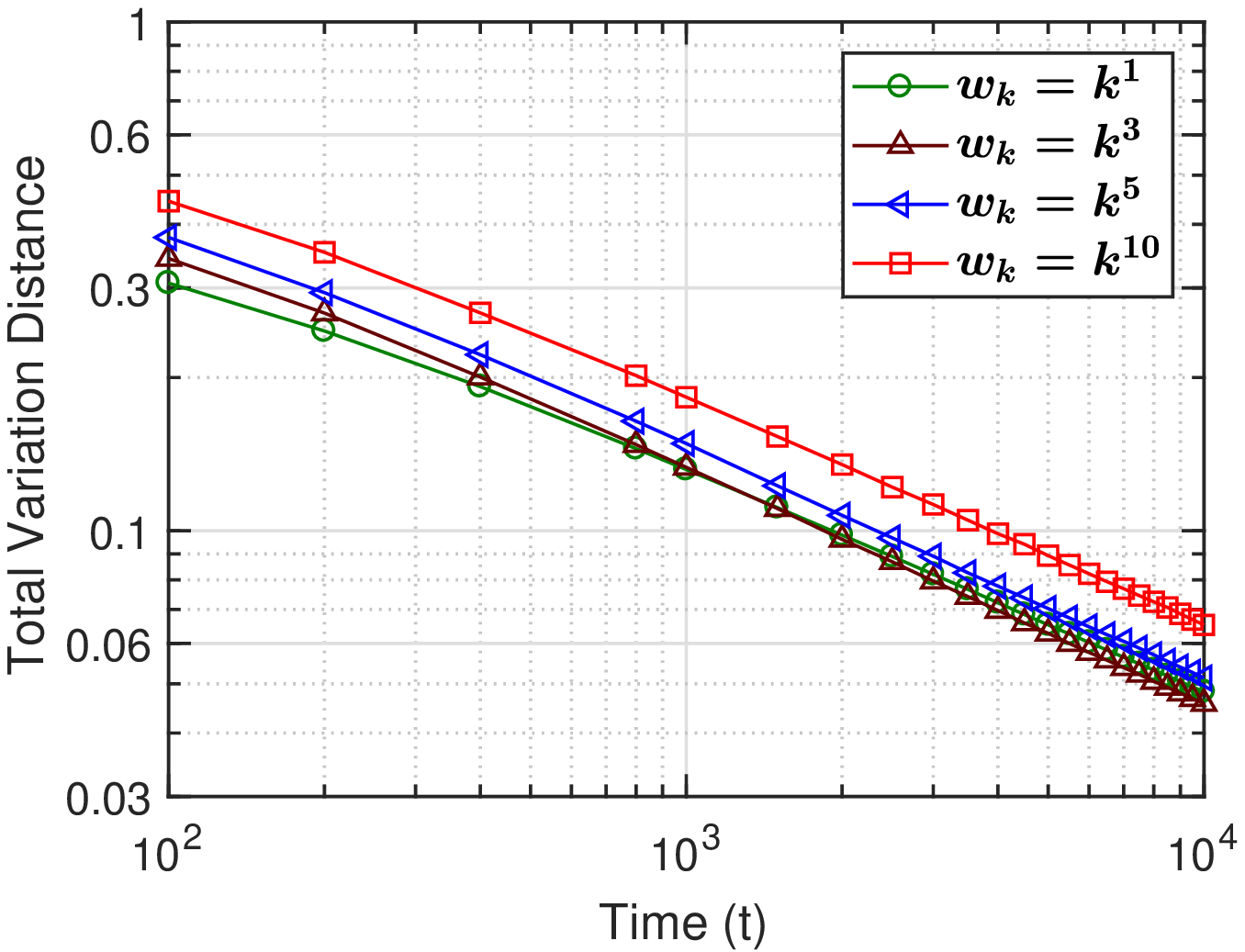}}
    \vspace{-3mm}
    \caption{The TVD results obtained under the Gnutella graph.} \label{fig:gnutella05}
    \vspace{-3mm}
\end{figure}

\section{Simulation Results}\label{se:simu}

In this section, we present simulation results to support our theoretical findings. To this end, we consider four real-world network datasets available in the SNAP repository~\cite{snap} as test-case directed graphs, whose statistics are summarized in Table~\ref{graph}. The numbers of nodes and edges in the largest strongly connected component (LSCC) of each graph are given in the fourth and fifth columns respectively. Note that self-loops, if present, e.g., in the original dataset of Slashdot graph, are removed for our simulations.

\begin{table}[htbp]
\centering
\begin{tabular}{|c|c|c|c|c|c|}
\hline & \; $\#$ nodes \; & \; $\#$ edges \; & $\#$ nodes (LSCC) & $\#$  edges (LSCC) \\ \hline
Gnutella  & 8,846    & 31,839   & 3,234   & 13,453  \\ \hline
Slashdot  & 82,168   & 948,464   & 71,307 & 912,381 \\ \hline
Wiki-Talk  & 2,394,385  & 5,021,410   & 111,881 & 1,477,893 \\ \hline
Amazon  & 400,727  & 3,200,440   & 380,167 & 3,069,889 \\ \hline
\end{tabular}
\vspace{2mm}
\caption{Statistics of the datasets }\label{graph}
\vspace{-6mm}
\end{table}

The primary focus here is, through numerical simulations, to confirm that the historical empirical distribution $\hmmu_t$ by our NMMC method in Algorithms~\ref{alg1} and~\ref{alg2} (the original one and its dynamic extension) converges to the target QSD $\nnu \!=\! \ppi$ (or $\nnu \!=\! \xx$) and to evaluate the speed of convergence of $\hmmu_t$ to $\nnu$. As a performance metric, we use the total variation distance (TVD)\footnote{Note that our purpose here is completely different from the typical use of TVD for characterizing the speed of convergence of an ergodic, reversible Markov chain on a finite state-space, say $\q$, to its stationary distribution, say $\mmu$. To be precise, for any initial $\mmu_0$, $\| \mmu_0 \q^t - \mmu \|_{TV} \leq \kappa \rho^t$ for some constant $\kappa$, where $\rho$ is the second largest eigenvalue modulus of $\q$ with $\rho \!<\! 1$~\cite{Bremaud99}. This exponential convergence speed is for the `marginal' distribution $\mmu_0 \q^t$ at time $t$ to $\mmu$ as $t$ increases. In contrast, we are interested in the convergence of the historical empirical distribution $\hmmu_t$, which encompasses the entire history of the underlying process with proper time weights.} between the historical empirical distribution $\hmmu_t$ and the target QSD $\nnu$, which is given by
\begin{equation}
\label{def:dtv}
\| \hmmu_t - \nnu \|_{TV} \triangleq \max_{E\subseteq\N} |\hmu_t(E) - \nu(E)| =
\frac{1}{2}  \sum_{i\in\N} |\hmu_{t}(i) - \nu(i)|.
\end{equation}

The inputs of Algorithms~\ref{alg1} and~\ref{alg2} are chosen as follows. The initial position $Z_0$ of each agent is randomly determined, with the initial empirical measure $\hmmu_0 \!=\! \delta_{Z_0}$. For the target QSD $\nnu$ to be mapped and estimated, we consider (i) uniform distribution, i.e., $\pi(j) \!=\! 1/n$, (ii) in-degree distribution, i.e., $\pi(j) \!\propto\! d^-_j$, and (iii) the EVC $\xx$ of the adjacency matrix $\A$. We also consider $\q \!=\! \q^{\textsf{srw}}$, where $Q_{ij} \!=\! A_{ij}/d^+_i$ for $i,j \!\in\! \N$. We use the LSCC of each graph for simulations of all these cases. In order to understand how different choices of unnormalized weight functions $w_k$ would affect the speed of convergence, we vary weight functions as $w_k \!=\! 1, k^1, k^3, k^5, k^{10}, 2^{\sqrt{k}}$, with increasing tendency to relocate to recently-visited nodes upon redistribution. For each simulation we use a set of independent crawling agents, which are independently moving over the graph. While each agent maintains its own historical empirical distribution $\hmmu_t$ for redistribution in case of individual absorption, all the historical empirical distributions are combined and normalized as a unified distribution for the convergence test.

We first present simulation results for Gnutella graph. We use 100 independent crawling agents for all the results. Figure~\ref{fig:gnutella05}(a) shows the TVD results
when achieving a uniform distribution $\uu$ under the NMMC method in Algorithm~\ref{alg1}, where the global constant $c$ is known a priori. Recall that for achieving $\ppi \!=\! \uu$, the acceptance probability $\gamma_{ij}$ in (\ref{r-qij}) becomes $\gamma_{ij} \!=\! b_{ij}/c$ with $b_{ij} \!=\! d^+_i/d^-_j$ and $c \!=\! \max_{i,j}b_{ij}$. This global constant $c$ can be as large as $c \!=\! \max_{i} d_i^+$ and this was the case in this simulation. Figure~\ref{fig:gnutella05}(a) shows that $\hmmu_t$ is getting closer to the target QSD $\nnu \!=\! \ppi$ for all considered weight functions $w_k$. In addition, Figure~\ref{fig:gnutella05}(a) exhibits that
heavier emphasis on latter visited sites leads to faster convergence of $\hmmu_{t}$ to $\ppi$. We would expect that it is attributed to higher weights of being relocated toward more recently-visited nodes that gives rise to faster diffusion on $\GG$ and quicker exploration of the graph. Nonetheless, the overall performance may be considered unsatisfactory as the resulting TVD values are still high. This means that the redistributions still take place too often, resulting in an imbalance between diffusion and redistribution.

In Figure~\ref{fig:gnutella05}(b)--(d), we present the TVD results for estimating the same uniform $\uu$ under the \emph{dynamic} NMMC method in Algorithm~\ref{alg2} while varying the value of the $c_t$-updating probability $p$ from 1 to 0.01. It turns out that, with the choice of $p \! = \! 0.01$, the TVD is reduced by more than 0.1 at every time step when compared to that of $p \!=\! 1$. To achieve the same precision in terms of the TVD, the number of time steps $t$ required for $p \!=\! 1$ would be much larger than that of $p \!=\! 0.01$ by at least an order of magnitude. Observing a clear decreasing trend with $p$, one may think that $p\!=\!0$ would be the best. However, when $p\!=\!0$, $c_t$ is never updated over $t$ as seen from our Algorithm~\ref{alg2}, i.e., $c_t\!=\!c_0\!=\!1$ for all $t$. Still, it is possible that $\gamma_{ij}(t)\!=\! \min\{1, b_{ij}/c_t\} \!=\! \min\{1, b_{ij}\}$ is less than one, meaning that a proposed move is not always accepted. Nonetheless, since $c_t$ never approaches the true value of $c$ when $p\!=\!0$, there is no theoretical guarantee for the convergence either. Accordingly, the resulting TVD does not decrease as shown in Figure~\ref{fig:gnutella05-problem}(a), where $w_k \!=\! 1$ is considered. Furthermore, in view of faster diffusion with smaller $p$ leading to a better speed of convergence, we can take another extreme case by simply running the simple random walk all the time, which always diffuses without any redistribution mechanism. It turns out that the resulting TVD again does not decrease as shown in Figure~\ref{fig:gnutella05-problem}(b). The observations so far collectively indicate an importance of having a right balance between diffusion and redistribution to shape up the historical empirical distribution $\hmmu_t$ toward the target QSD and to achieve faster convergence.

\begin{figure}[t!]
    \centering
    \vspace{-0mm}
    \subfigure[$p = 0$ vs. $p = 0.01$ ]{\includegraphics[width=0.35\textwidth]{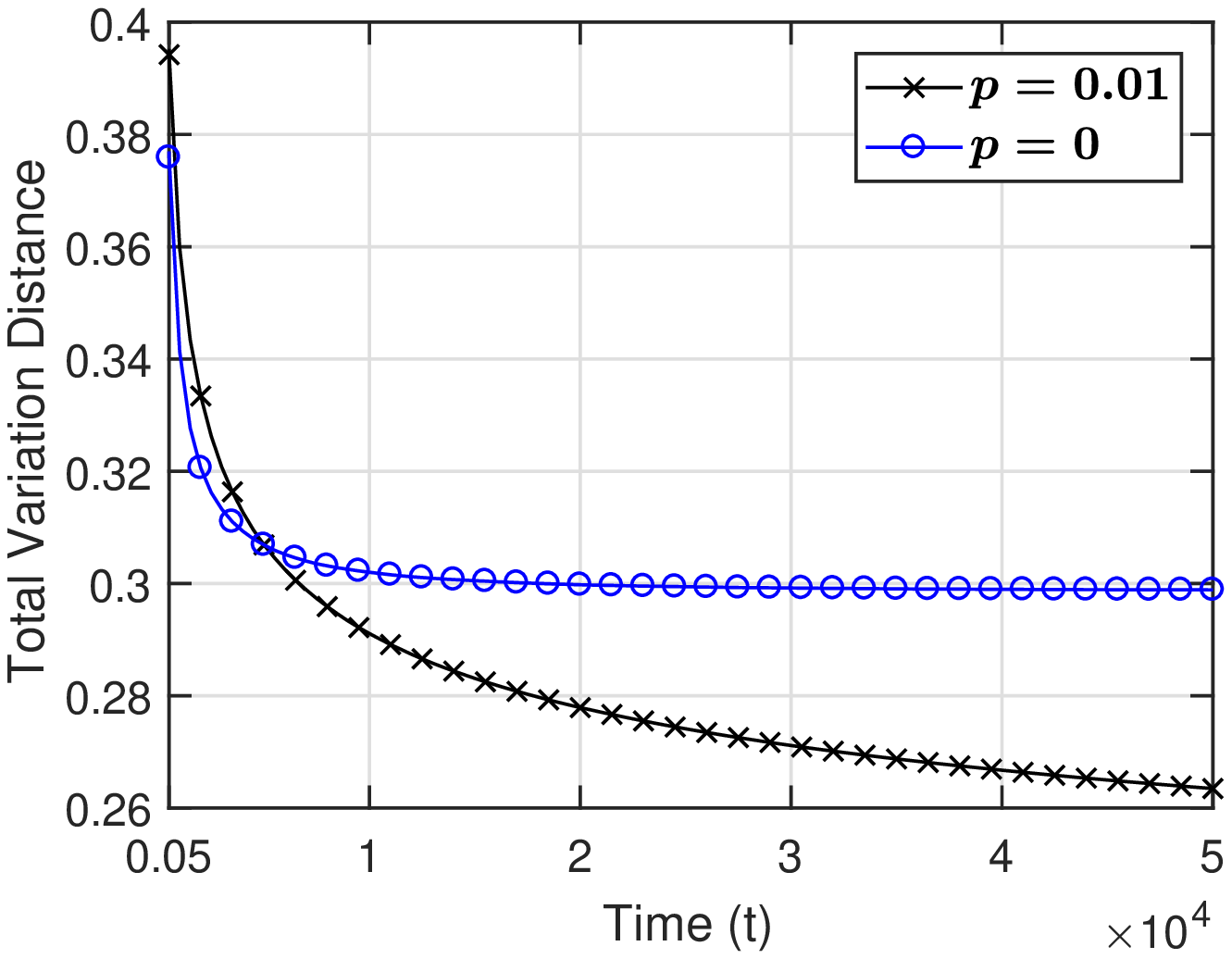}}
    \hspace{6mm}
    \subfigure[$\q^{\textsf{srw}}$ with no redistribution]{\includegraphics[width=0.35\textwidth]{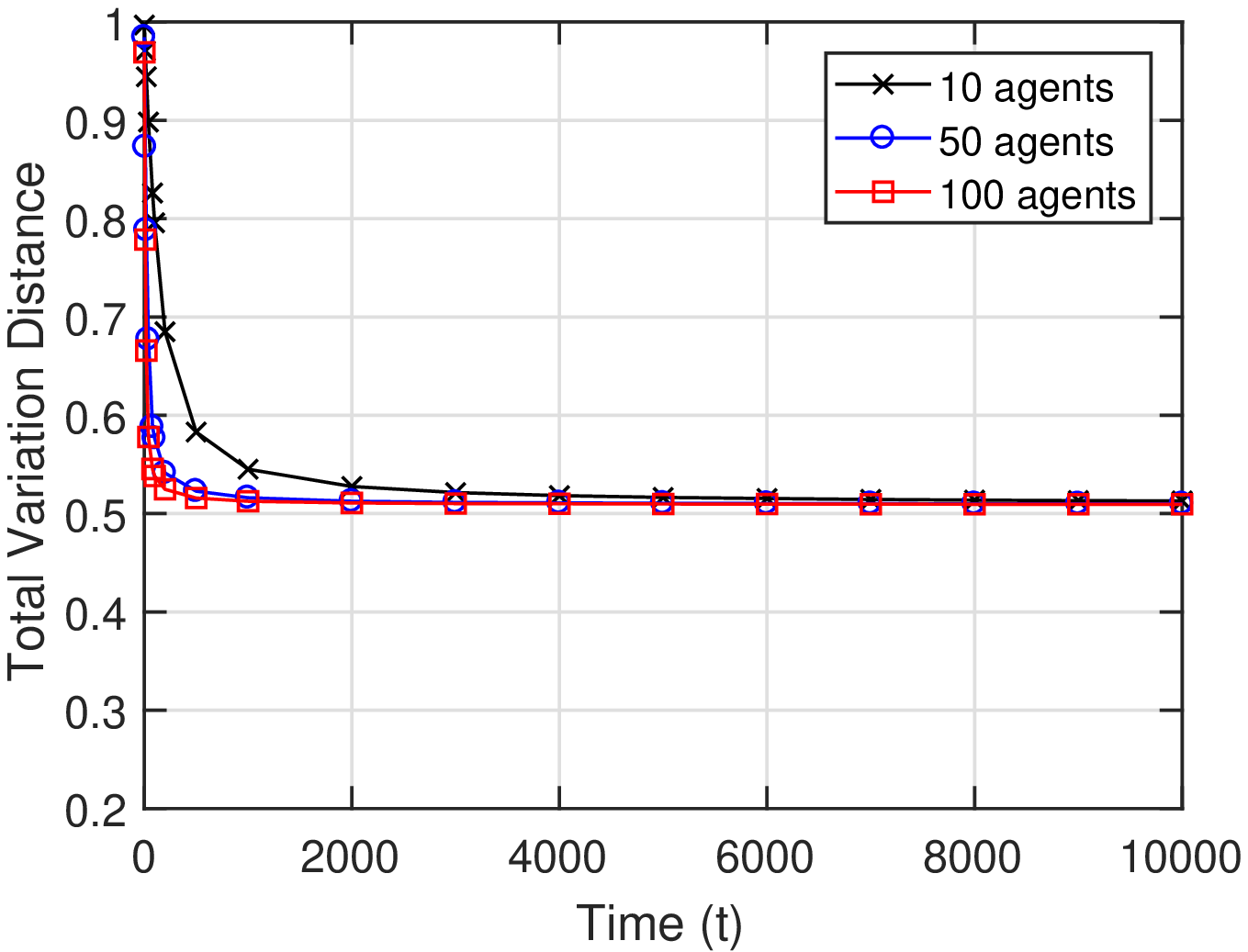}}
    \vspace{-3mm}
    \caption{Failure of convergence $\hmmu_t$ to $\ppi \!=\! \uu$ when using two possible extreme methods that put too much weight on diffusion of the underlying crawling agent.} \label{fig:gnutella05-problem}
    \vspace{-3mm}
\end{figure}

We further take a closer look into the cases with $w_k \!=\! k^\alpha$ and $w_k \!=\! 2^{\sqrt{k}}$. While the case with $w_k \!=\! 2^{\sqrt{k}}$ is the best with no dynamic update of $c$ as seen from Figure~\ref{fig:gnutella05}(a), it becomes worse than the case with $w_k \!=\! k^\alpha$ when using the dynamic NMMC method with $c_t$-updating probability $p$ as shown in Figure~\ref{fig:gnutella05}(b)--(c). The case with $w_k \!=\! 2^{\sqrt{k}}$ when combined with the dynamic updates of $c_t$ places over-emphasis on diffusion, which causes more harm than help. This again signifies the importance of striking the right balance between diffusion and redistribution of the reinforced walk $Z_t$ toward faster convergence, although the presence of such a crossover would depend on the choice of the directed graph $\GG$, target QSD $\nnu$ and the proposal chain $\q$. Furthermore, we have observed that the case with $w_k \!=\! k^\alpha$ is always better with the use of $c_t$-updating probability $p$ than the case with $w_k \!=\! 2^{\sqrt{k}}$, when achieving the in-degree distribution, i.e., $\pi(j) \!\propto\! d^-_j$, and the EVC $\xx$ of the adjacency matrix $\A$. We omit the results for brevity and exclude the case with $w_k \!=\! 2^{\sqrt{k}}$ for the subsequent numerical evaluations. Note that the acceptance probability $\gamma_{ij}$ in (\ref{r-qij}) becomes $\gamma_{ij} \!=\! d^+_i/(cd^-_i)$ with $c \!=\! \max_i d^+_i/d^-_i$ for the former, and $\gamma_{ij} \!=\! d^+_i/c$ with $c \!=\! \max_i d^+_i$ for the latter.

We instead report the results under the (unnormalized) weight function $w_k \!=\! k^\alpha$ with different values of $\alpha$ and $p \!=\! 0.01$ for all three cases of the target QSD in Figure~\ref{fig:gnutella05}(d)--(h). For achieving $\ppi \!=\! \uu$, the case with $\alpha \!=\! 1$ exhibits the fastest decaying rate in its TVD curve initially, but later becomes slower than the other cases. Similar tendency with $\alpha \!=\! 1$ is observed for achieving $\pi(j) \!\propto\! d^-_j$ and estimating the EVC $\xx$, as shown in Figure~\ref{fig:gnutella05}(e)-(h). On the other hand, we see that the case with $\alpha \!=\! 10$ eventually becomes the best with the fastest decaying rate, as we can see from the inset figure of Figure~\ref{fig:gnutella05}(d). We also observe that the eventual winner is the case with $\alpha \!=\! 5$ for achieving $\pi(j) \!\propto\! d^-_j$, while the case with $\alpha \!=\! 3$ eventually appears to be the best for estimating the EVC $\xx$. This set of simulation results indicate that target QSD $\nnu$ clearly affects the extent of diffusion to achieve the right balance between diffusion and redistribution.

In addition, we observe that TVD values decay to zero polynomially fast, as evidenced in Figure~\ref{fig:gnutella05}(f), Figure~\ref{fig:gnutella05}(h), and the inset figure of Figure~\ref{fig:gnutella05}(d), all drawn on a log-log scale. We have also observed such a polynomial decaying behavior for the other choices of parameters. This behavior can be explained as follows.

Suppose that we have a trajectory of an ergodic Markov chain, say $\{X_t\}_{t \geq 0}$, with its stationary distribution $\ppi$. For a function $f: \N \!\to\! \R$, letting $S_t \!:=\! \sum_{k=1}^t f(X_k)$, $S_t/t$ converges to $\Ex_{\ppi}\{f\}$ almost surely as seen from (\ref{ergodic}) as $t$ grows. In addition, the standard Central Limit Theorem (CLT) for ergodic Markov chains says that for any bounded function $f$ and any initial distribution on $X_0$,
\begin{equation}\label{clt}
  \sqrt{t}\lt(S_t/t - \Ex_{\ppi}\{f\} \rt) \overset{d}{\longrightarrow} N(0, \sigma^2), ~~\text{as}~ t \to \infty,
\end{equation}
where $N(0, \sigma^2)$ is a Gaussian random variable with zero mean and variance $\sigma^2$. The (asymptotic) variance $\sigma^2$ is given by~\cite{Liu04,LeeSIGMETRICS12,BrunoSIGMETRICS16}
\begin{equation*}
  \sigma^2 = \lim_{t \to \infty} \frac{1}{t}\var\lt(S_t\rt) = \var_{\ppi}\{f\} + 2\sum_{k=1}^{\infty}\cov_{\ppi}\lt(f(X_0), f(X_k)\rt),
\end{equation*}
where the first term in RHS is the marginal variance of $f$ with respect to $\ppi$ and the second term in RHS is the lag-$k$ autocovariance of the stationary sequence $\{f(X_t)\}$. From (\ref{clt}), we see that the fluctuations of $S_t/t$ around the mean value $\Ex_{\ppi}\{f\}$ are roughly within the interval of length of order $\sigma/\sqrt{t}$. Even when $X_t$'s are $i.i.d.$ over $t$, each having distribution $\ppi$ (i.e., perfect independent samples all the time), the asymptotic order of
convergence from $S_t/t$ to $\Ex_{\pi}\{f\}$ still remains the same as $O(1/\sqrt{t})$, while $\sigma^2$ reduces to $\var_{\ppi}\{f\}$.

What we have in this paper is the convergence of $\hmmu_t \!\in\! \Omega$ to the target $\nnu \!\in\! \Omega$ almost surely and the typical size of the fluctuations of the difference $\hmmu_t - \nnu$, where $\Omega \!=\! \{ \nnu \in [0,1]^n |  \sum_{i=1}^n \nu(i) \!=\! 1 \}$ is the set of all probability distributions on $\N$. If the underlying process on $\N$ were to be a Markov chain, by resorting to similar CLT, we would obtain similar $O(1/\sqrt{t})$ convergence speed where the typical size of the fluctuations is now characterized by that of a $n$-dimensional multivariate Gaussian random vector with some $n \times n$ covariance matrix. However, the underlying process $Z_t$ in our framework is history-dependent and thus non-Markovian, so we would not expect anything faster than $O(1/\sqrt{t})$ for its error as seen from Markovian or $i.i.d.$ cases. Nonetheless, the figures drawn on a log-log scale in our setting demonstrate that the error decays still polynomially fast $O(t^{-\beta})$. These observations are further supported by the mathematical results in~\cite{Benaim15}, which apply to the static NMMC method. Specifically, Corollary 1.3 in~\cite{Benaim15} implies that if one considers a class of polynomial weight functions $w_k \!=\! k^\alpha$, $\alpha \!\geq\! 0$, then the difference $\hmmu_t - \nnu$ in the $l^1$ norm is bounded above by $C t^{-\theta}$ almost surely for some (random) constant $C \!>\! 0$, where $\theta \!\leq\! \frac{1}{2}$.\footnote{In other words, for every possible sample path $\omega$ of the underlying stochastic process $\{Z_t\}$ by the crawler, there exists some random constant $C(\omega)$ such that the difference $\hmmu_t - \nnu$ in the $l^1$ norm is bounded above by $C(\omega) t^{-\theta}$ almost surely. Note also that $\theta$ depends on the extremal value of the spectral functional, which requires the knowledge of the entire spectrum of the underlying graph structure. We refer to~\cite{Benaim15} for more details.} In addition, Theorem 1.4 in~\cite{Benaim15} shows that the CLT from $\hmmu_t$ to $\nnu$ holds with respect to a \emph{rescaled} time $\eta_t(t)^{-1} \!\sim\! t$, as it would hold for Markovian or $i.i.d.$ cases.

\begin{figure*}[t!]
    \centering
    \vspace{-0mm}
    \subfigure[$\ppi = \uu$; $p = 0.01$]{\includegraphics[width=0.31\textwidth]{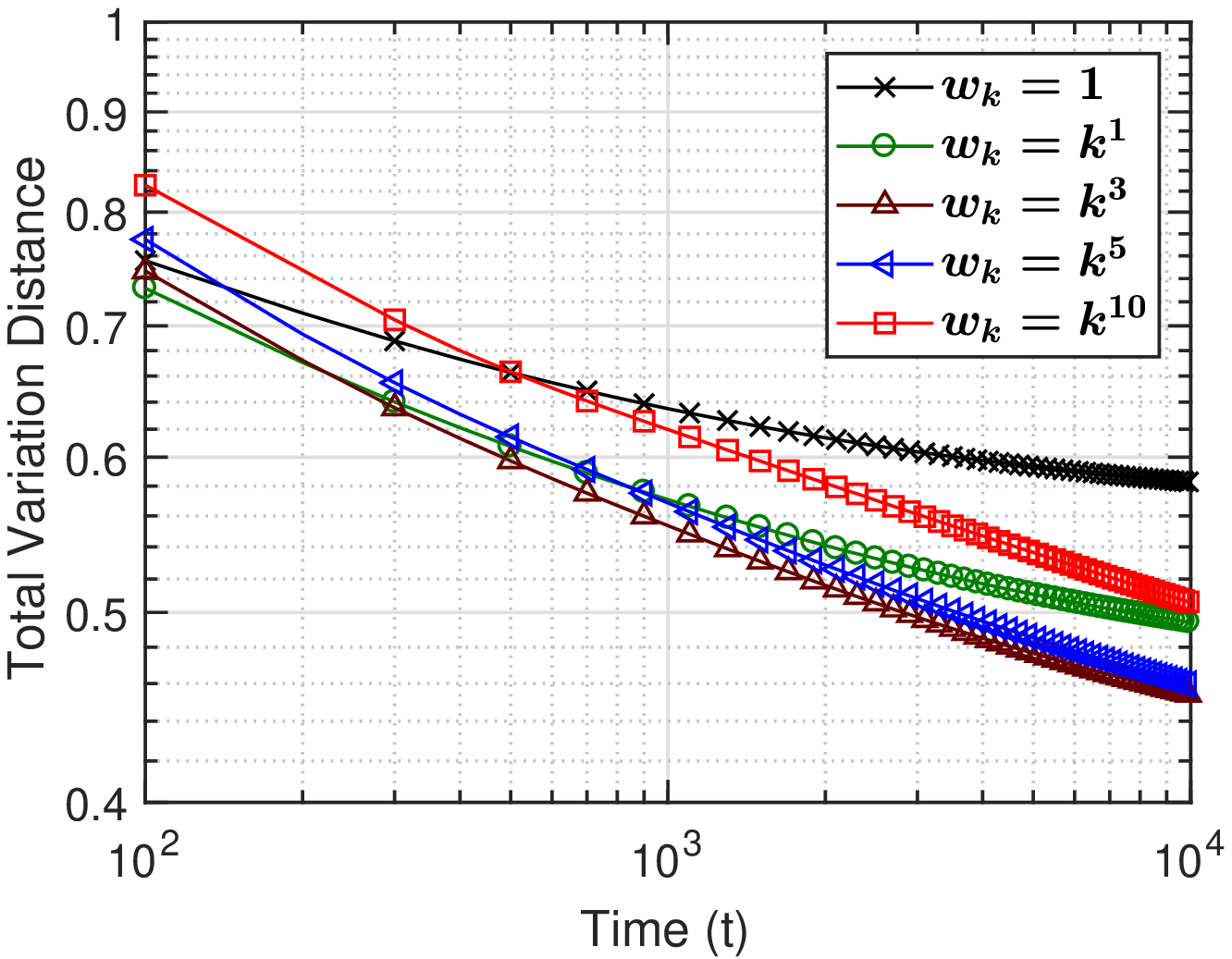}}
    \hspace{1mm}
    \subfigure[$\pi(i) \propto d^-_i$; $p = 0.01$]{\includegraphics[width=0.31\textwidth]{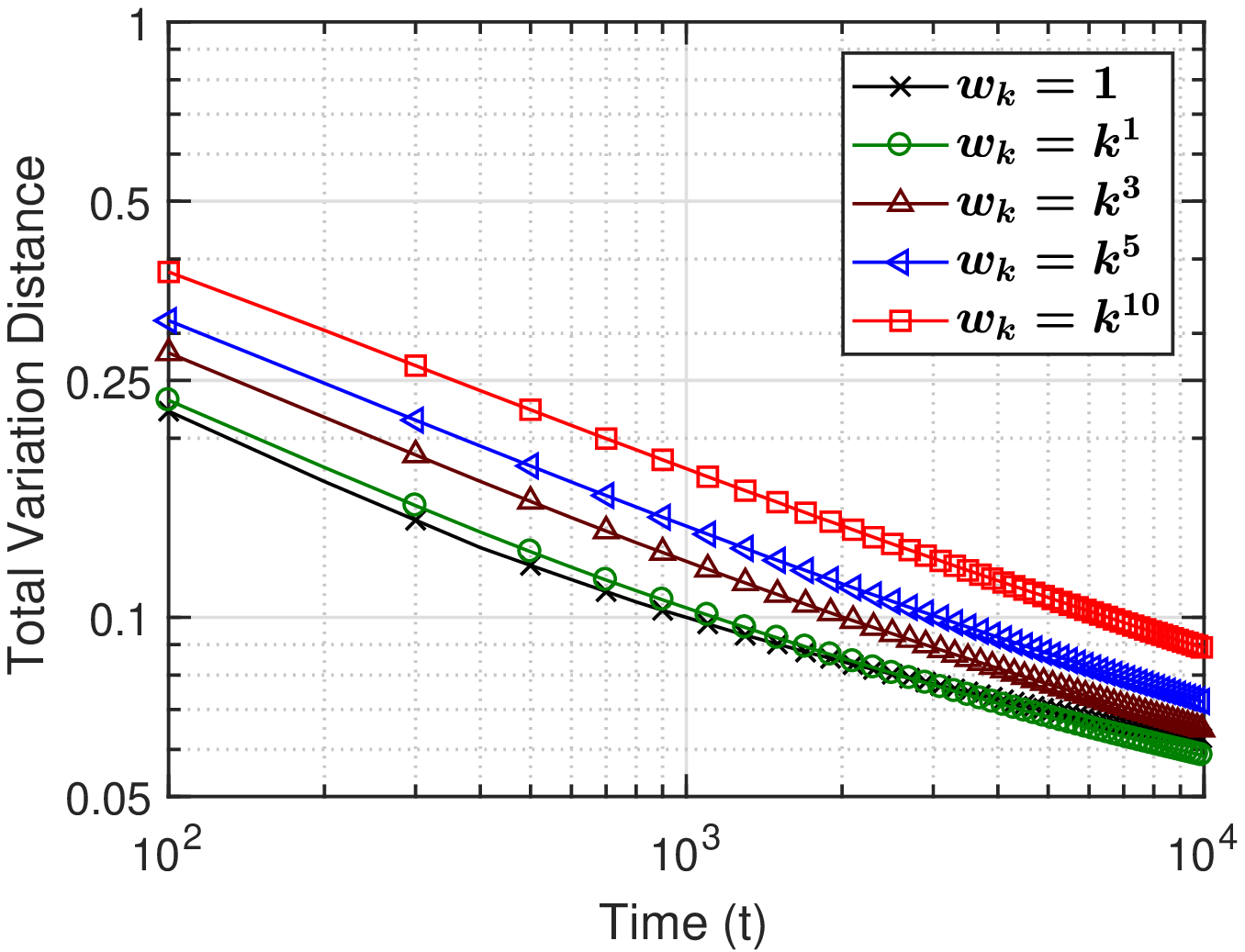}}
    \hspace{1mm}
    \subfigure[EVC $\xx$; $p = 0.01$]{\includegraphics[width=0.31\textwidth]{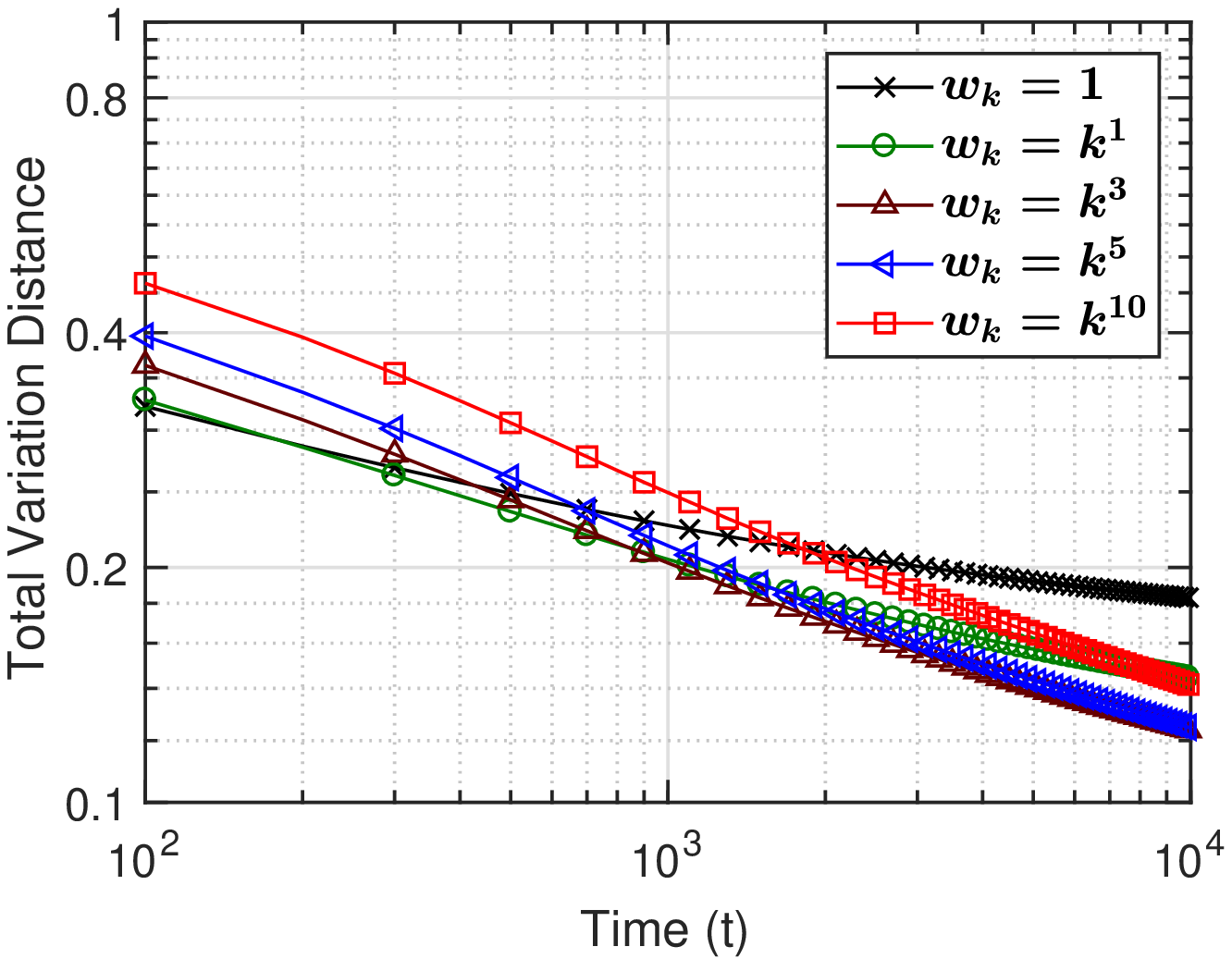}}
    \vspace{-3mm}
    \caption{The TVD results obtained under Slashdot graph.} \label{fig:slashdot}
    \vspace{-3mm}
\end{figure*}
\begin{figure*}[t!]
    \centering
    \vspace{-0mm}
    \subfigure[$\ppi = \uu$; $p = 0.01$]{\includegraphics[width=0.31\textwidth]{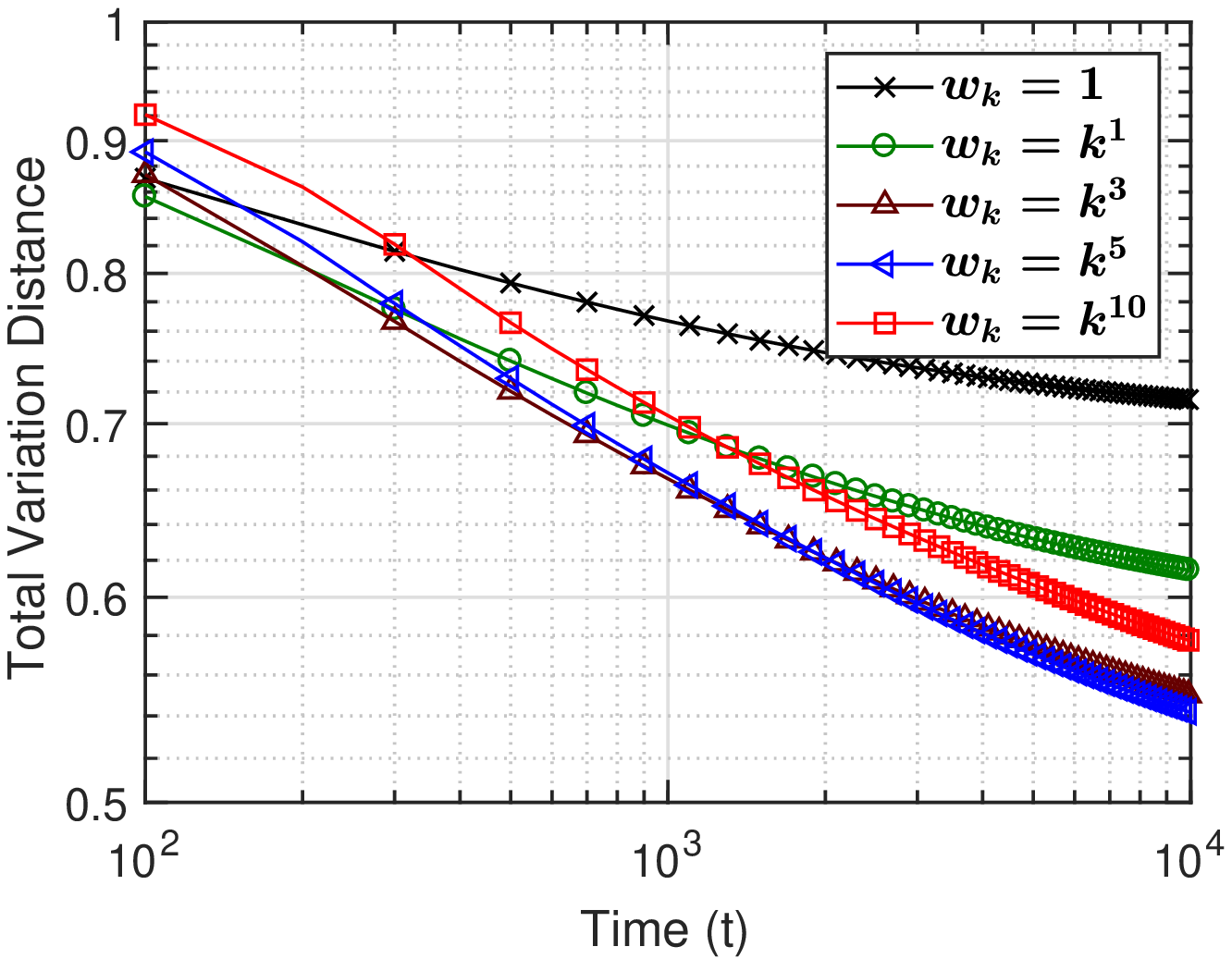}}
    \hspace{1mm}
    \subfigure[$\pi(i) \propto d^-_i$; $p = 0.01$]{\includegraphics[width=0.31\textwidth]{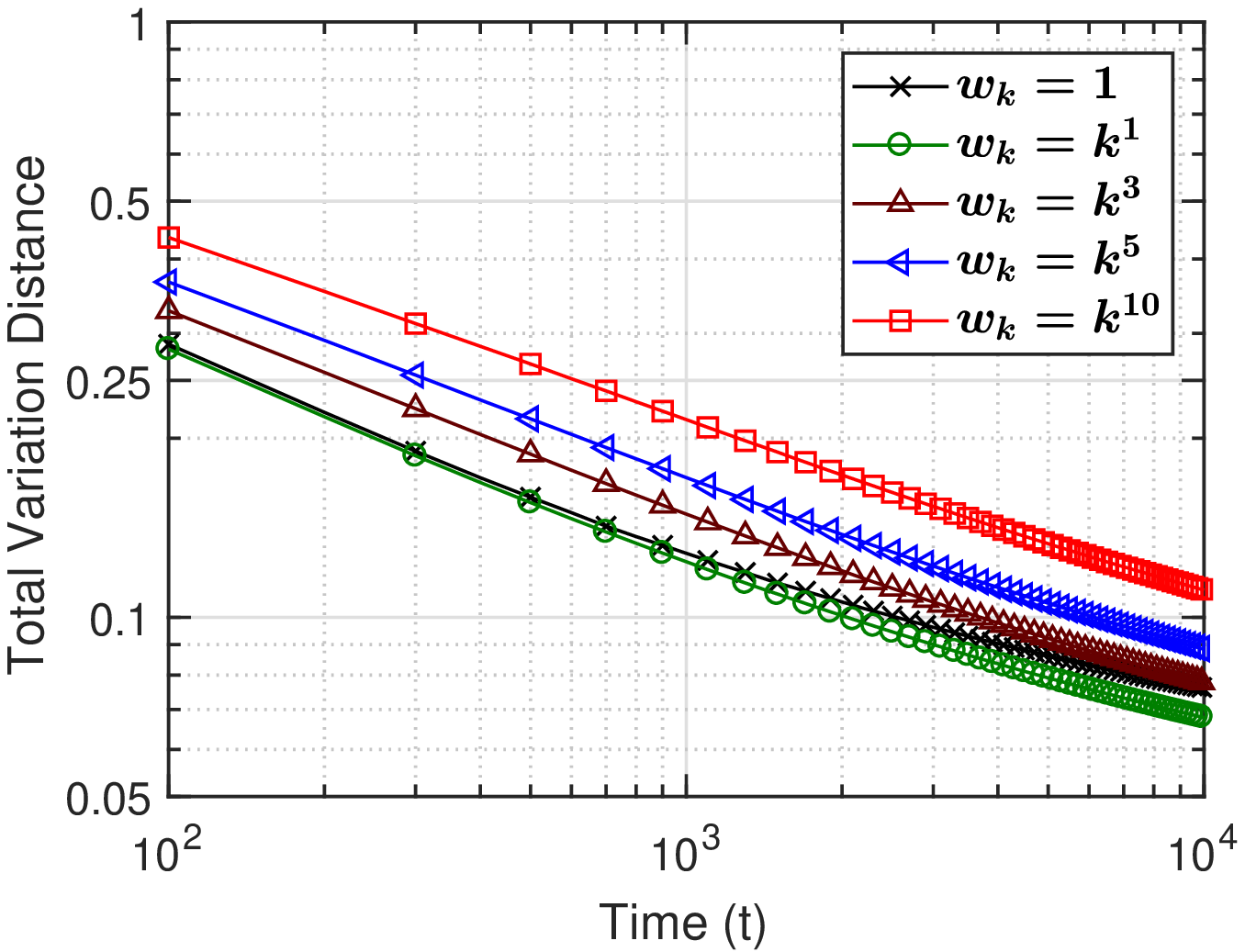}}
    \hspace{1mm}
    \subfigure[EVC $\xx$; $p = 0.01$]{\includegraphics[width=0.31\textwidth]{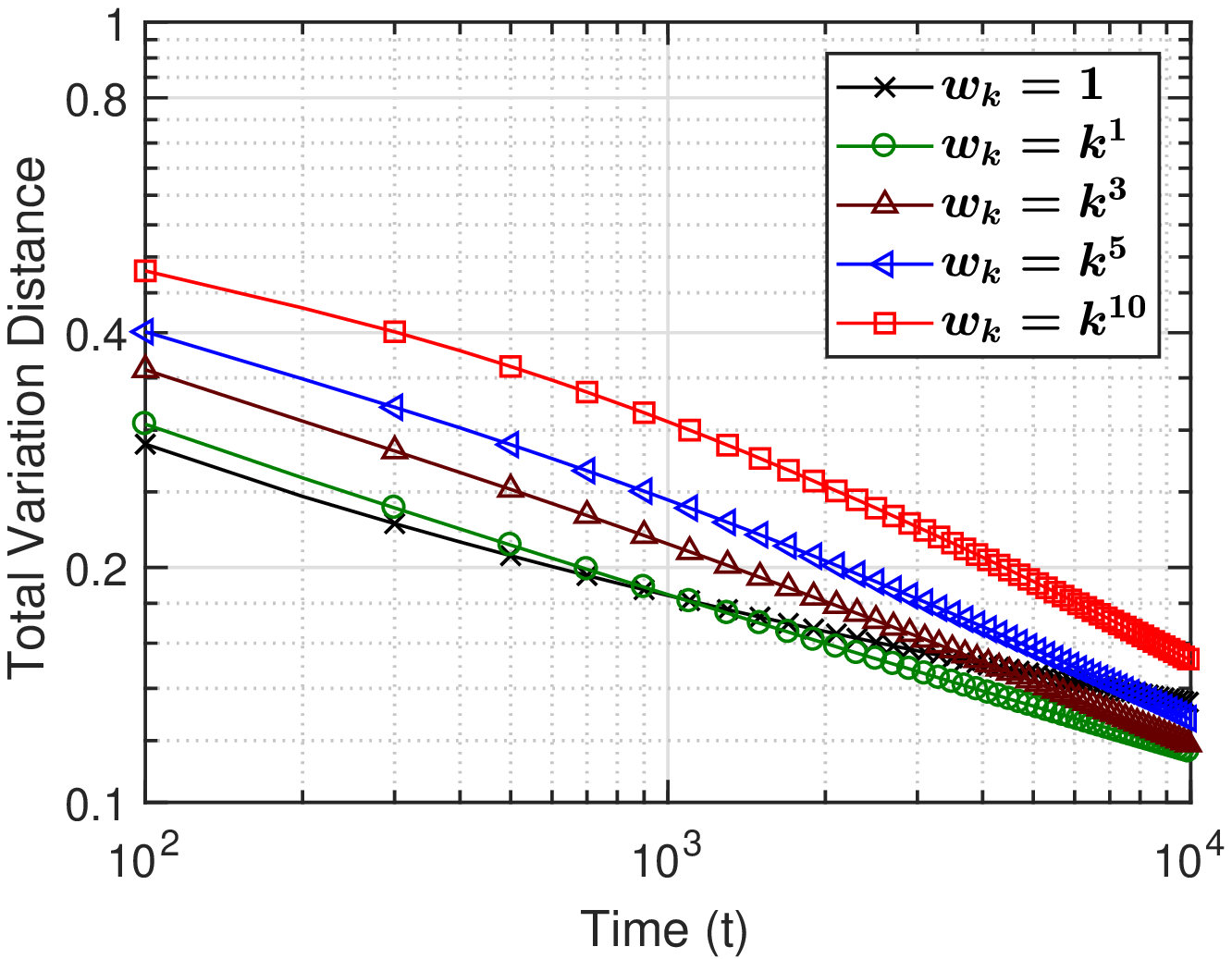}}
    \vspace{-3mm}
    \caption{The TVD results (on a log-log scale) obtained under Wiki-Talk graph.} \label{fig:wikitalk}
    \vspace{-3mm}
\end{figure*}

We now turn our attention to the simulation results under Slashdot and Wiki-Talk graphs. As the graphs are much bigger than the Gnutella graph, we properly increase the number of agents from 100 to 2000. We observe the similar trends as in the Gnutella graph and only present the TVD results for the cases with $w_k \!=\! k^\alpha$ ($\alpha \!=\! 0, 1,3,5,10$) when using the dynamic NMMC method with $c_t$-updating probability $p \!=\! 0.01$ in Figures~\ref{fig:slashdot} and~\ref{fig:wikitalk}, which consistently outperforms the other choices of $p \!=\! 1, 0.1$ for all three cases of the target QSD. We still observe a polynomial decrease in the convergence speed of $\hmmu_t$, which is asymptotically linear on a log-log scale. We see that the case with $\alpha \!=\! 3$ or $\alpha \!=\! 5$ exhibits the fastest decaying rate in its TVD curve for achieving $\ppi \!=\! \uu$ under both Slashdot and Wiki-Talk graphs and for estimating the EVC $\xx$ under Slashdot graph. On the other hand, the case with $\alpha \!=\! 1$ becomes the best with the fastest decaying rate for the other cases. This again shows that to strike a balance between diffusion and redistribution, the extent of diffusion, which can be controlled by the choice of $\alpha$, should be properly chosen based on the underlying graph structure and the target QSD. We also note that the convergence becomes much faster when the target QSD is set to be the EVC $\xx$ of $\A$ or the in-degree distribution, i.e., $\pi(j) \!\propto\! d^-_j$, as compared to the case of the uniform distribution $\ppi \!=\! \uu$.

We finally report the TVD results of the dynamic NMMC method with $c_t$-updating probability $p \!=\! 0.01$ under Amazon graph for the cases with $w_k \!=\! k^\alpha$ ($\alpha \!=\! 0,1,3,5,10$) when achieving $\pi(i) \propto d^-_i$. Since the size of the LSCC of the Amazon graph is over $\times100$ bigger than that of the Gnutella graph, we use $10^4$ agents. We again observe the polynomial decay in the convergence speed of $\hmmu_t$. We also see that the case with $w_k\!=\! 1$ or $w_k \!=\! k$ leads to the fastest decaying rate. This suggests that too strong a bias in weights toward more recently-visited nodes can have an adversary effect, leading to an imbalance between diffusion and redistribution.

\begin{figure}[h!]
    \includegraphics[width=0.37\textwidth]{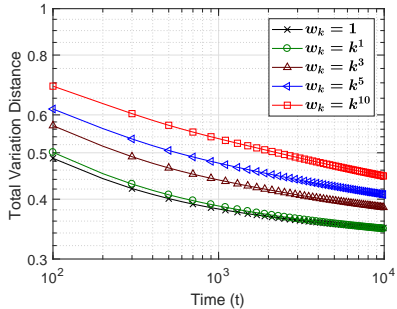}
    \vspace{-3mm}
    \caption{The TVD results (on a log-log scale) for achieving $\pi(i) \propto d^-_i$ with $p = 0.01$ under Amazon graph.}
    \label{fig:amazon}
    \vspace{-3mm}
\end{figure}

\section{Discussion and Practical Considerations}\label{se:discuss}

In this section, we thoroughly discuss the comparison between our NMMC method and the relevant literature, ranging from the MCMC methods to the PageRank algorithm. We also address practical issues with our NMMC method and demonstrate that our NMMC method is readily applicable to most direct graphs in practice.

\subsection{MCMC vs. NMMC}

We first discuss the key differences between the popular MCMC method and our NMMC method for sampling from a given probability distribution $\ppi$ on a graph $\GG$. In essence, the MCMC method is about how to build a Markov chain $\p$ on a given state space (e.g., graph) such that it attains any desired stationary distribution $\ppi$, i.e., $\ppi = \ppi \p$ on the same space~\cite{Liu04,Richey10}.
The Metropolis-Hastings (MH) algorithm, undeniably the most famous one in the MCMC literature, can be described as follows. For a given target distribution $\ppi$ to achieve on a graph, consider an underlying Markov chain $X_t$ on this graph with its transition matrix $\q = [Q_{ij}]$. At the current state $X_t=i$, the chain proposes to move to $X_{t+1}=j$ with probability $Q_{ij}$, which is then accepted with probability
\[
\gamma_{ij} = \min\lt\{1, \frac{\pi(j) Q_{ji}}{\pi(i) Q_{ij}}\rt\},
\]
and rejected with probability $1-\gamma_{ij}$ in which case the chain stays at the same state, i.e., $X_{t+1} = i$. Clearly, it then immediately follows that
\begin{equation}
\pi(i) \gamma_{ij}Q_{ij} = \pi(j) \gamma_{ji}Q_{ji},
\label{MH-rev-assump}
\end{equation}
for all $i,j\in\N$, i.e., the resulting stochastic matrix $\p = [P_{ij}]$ with $P_{ij} = \gamma_{ij}Q_{ij}$ for $i\neq j$ and $P_{ii} = 1-\sum_{k\ne i}P_{ik}$ becomes reversible with respect to $\ppi$, thus satisfying the original equation $\ppi = \ppi \p$. This method has proven to be extremely versatile and powerful, using only local information to find $\gamma_{ij}$ with any transition matrix $Q_{ij}$ on the given state space.
From an algebraic point of view, this amounts to constructing a reversible \emph{stochastic} matrix $\p$ such that its leading left-eigenvector matches a given $\ppi$. As mentioned in Section~\ref{se:intro}, again, the most important yet implicit assumption in (\ref{MH-rev-assump}) is that $Q_{ji}>0$ whenever $Q_{ij}>0$, i.e., nodes $i$ and $j$ have reciprocal relationship. If not, i.e., there exists a directed edge from $i$ to $j$ but not vice versa, we would have $Q_{ji}=0$ even if $Q_{ij}>0$, for which the MH algorithm no longer applies. What we have established in this paper is to show how we can remove the requirements of (i) being stochastic and (ii) reciprocity in the underlying process.

\begin{figure*}[t!]
    \centering
    \vspace{-0mm}
    \subfigure[TVD; MH]{\includegraphics[width=0.31\textwidth]{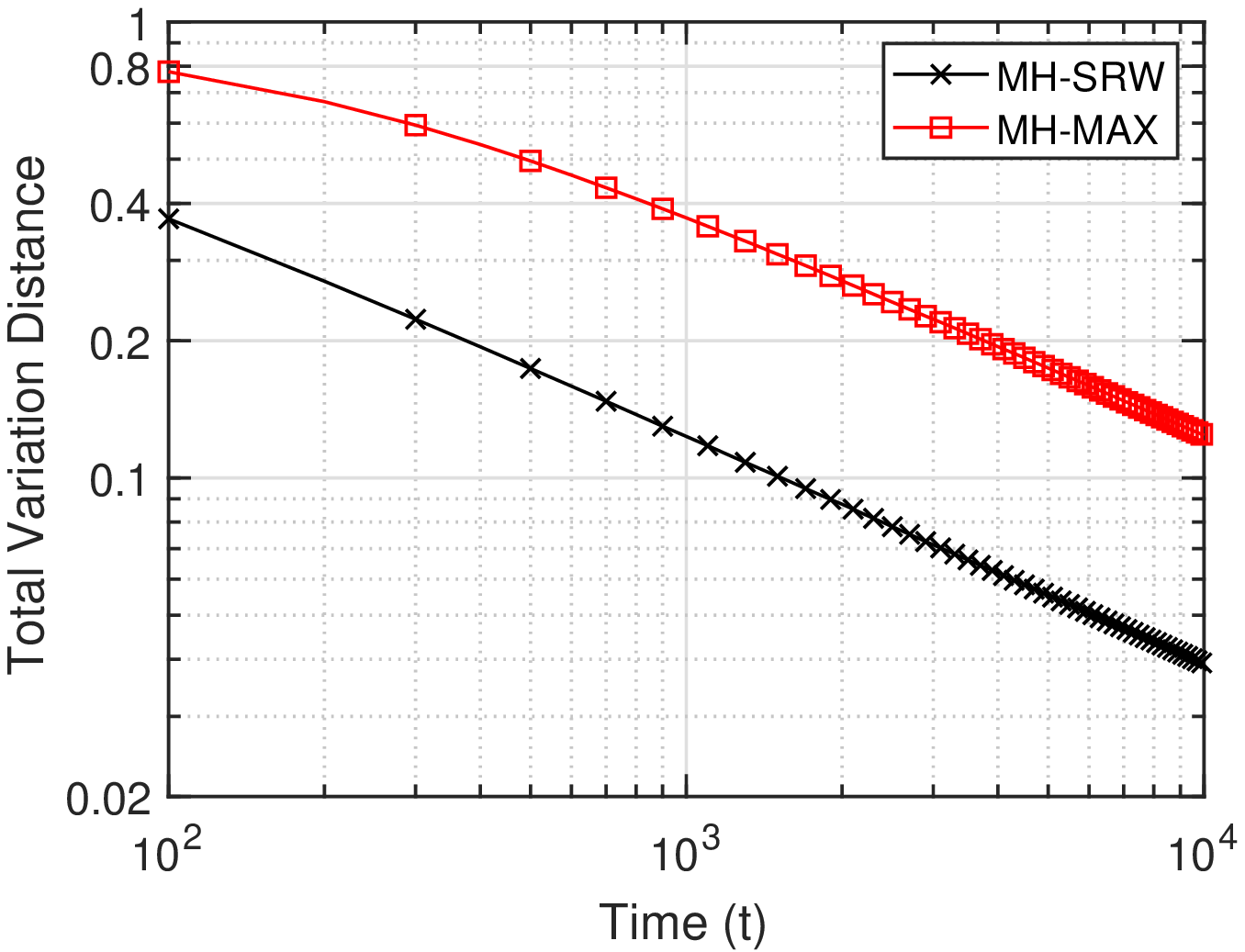}}
    \hspace{6mm}
    \subfigure[TVD; NMMC]{\includegraphics[width=0.31\textwidth]{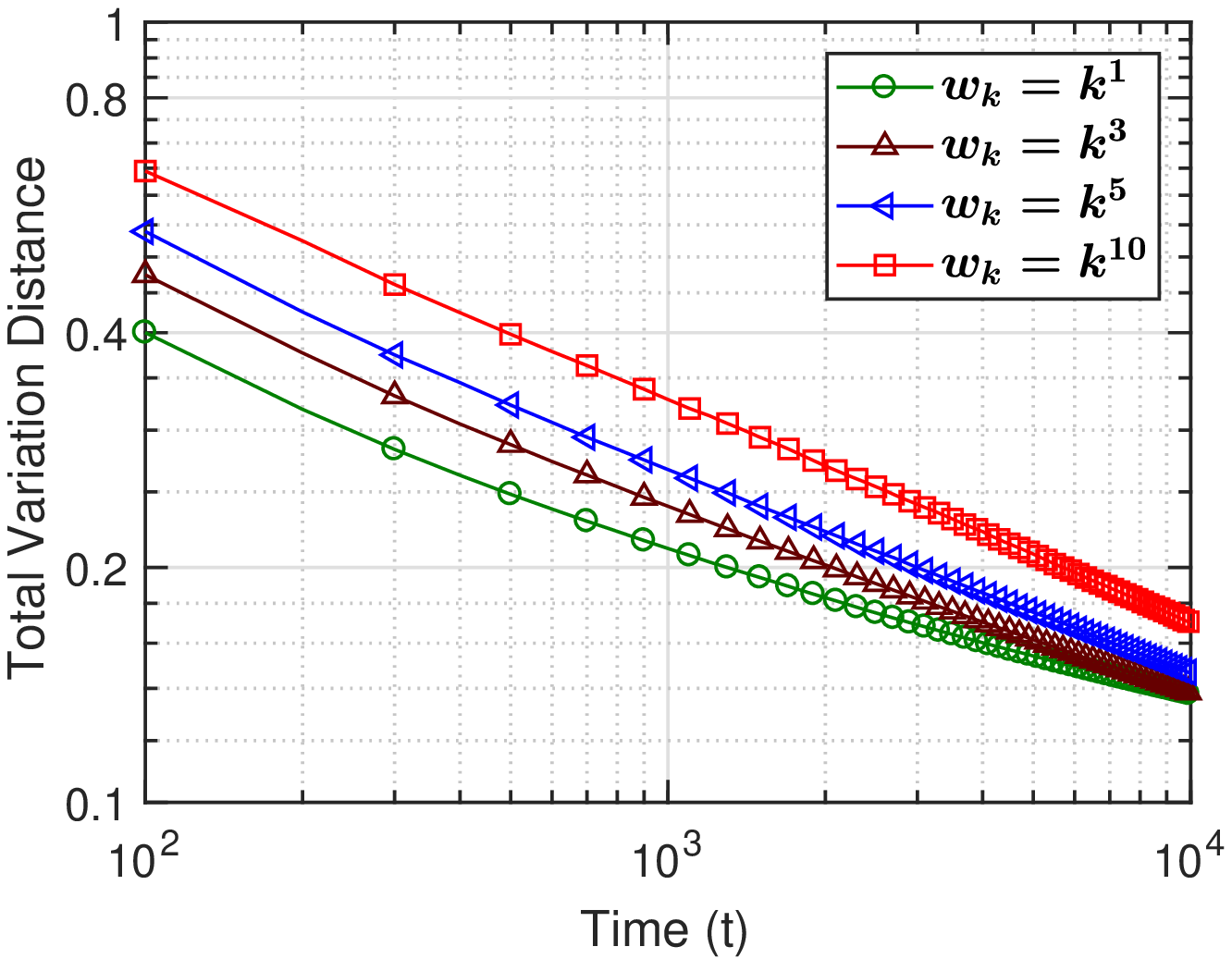}}
    \\\vspace{-1mm}
    \subfigure[NRMSE; MH]{\includegraphics[width=0.31\textwidth]{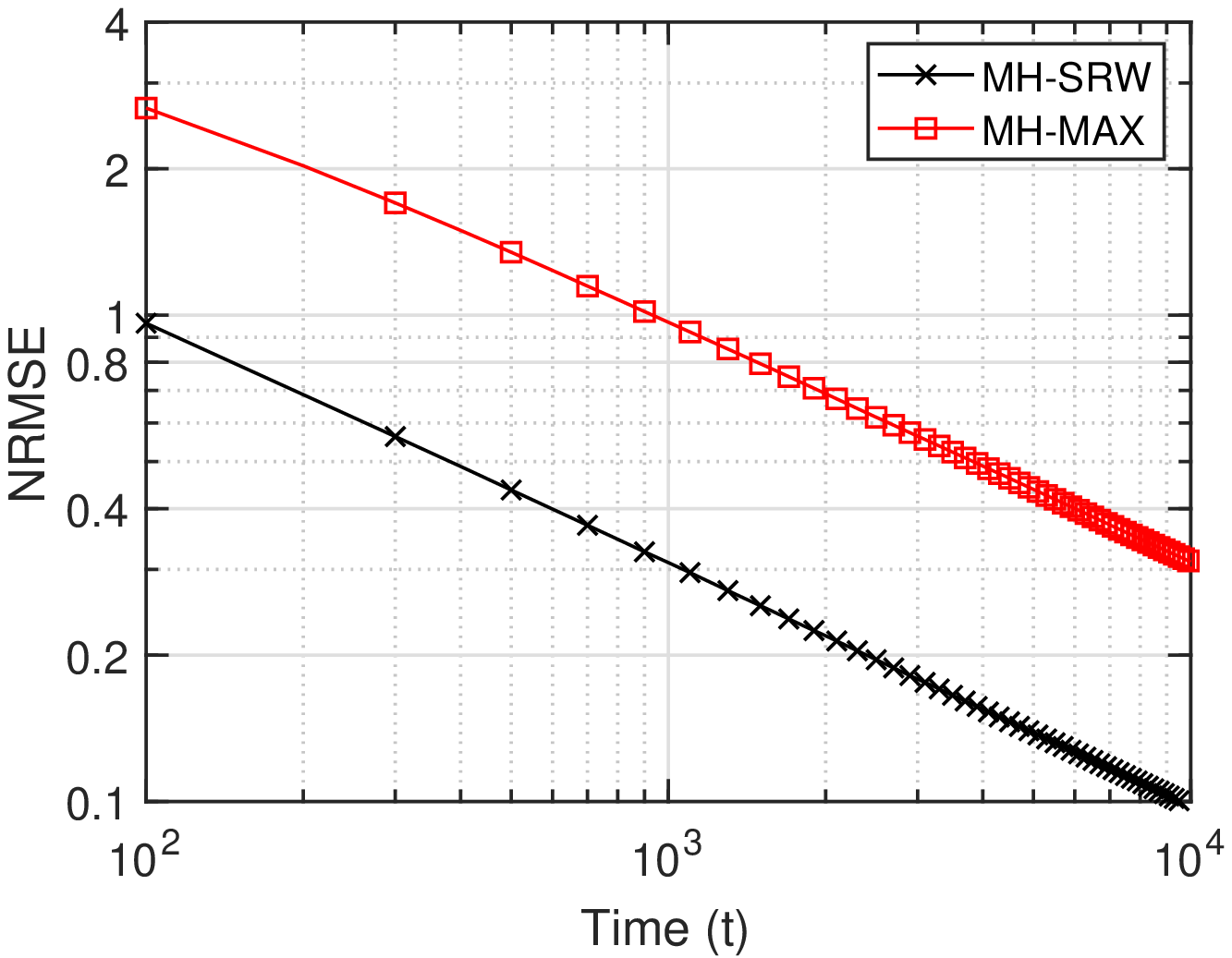}}
    \hspace{6mm}
    \subfigure[NRMSE; NMMC]{\includegraphics[width=0.31\textwidth]{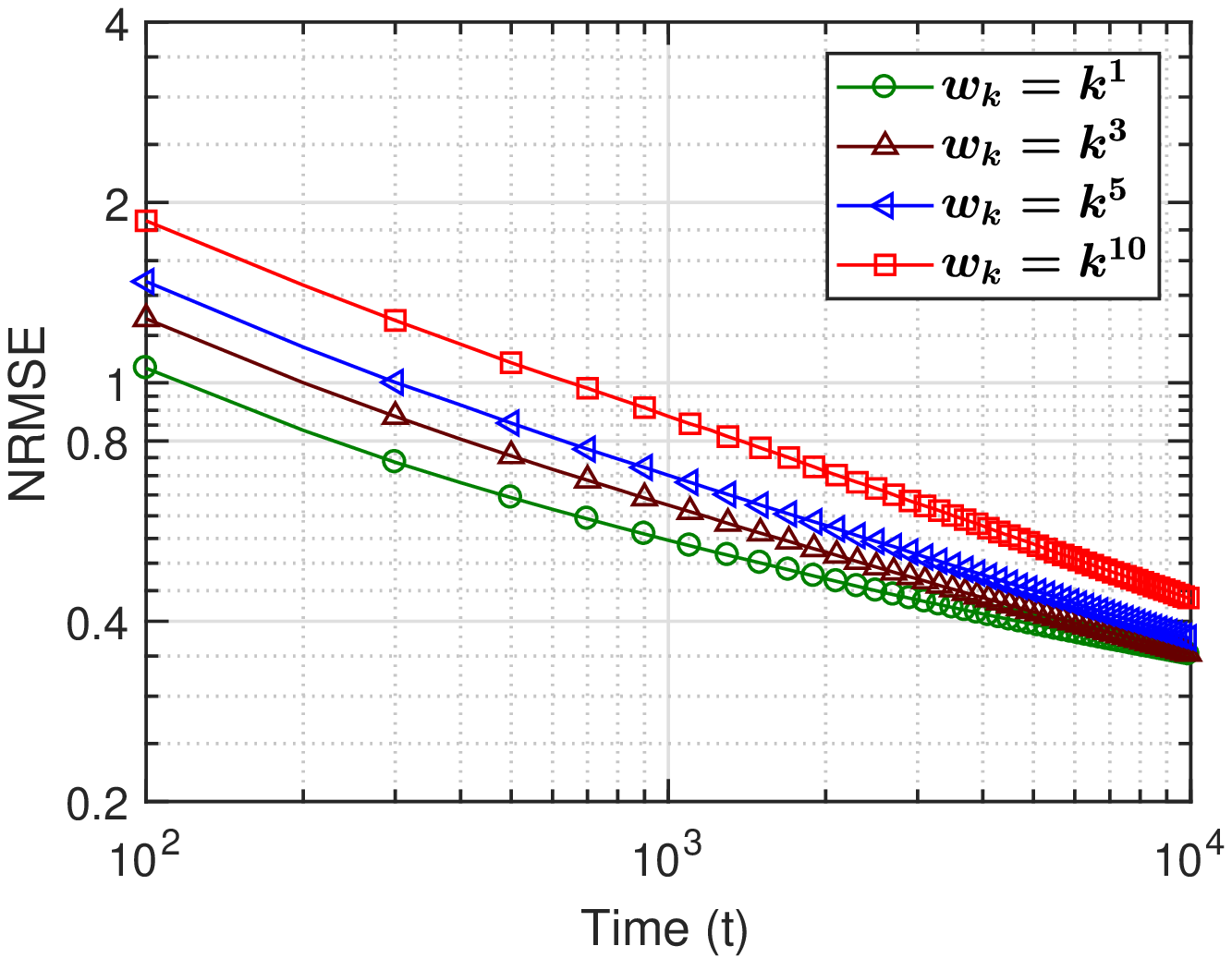}}
    \vspace{-3mm}
    \caption{The TVD and NRMSE results (on a log-log scale) of (a) MH-SRW and MH-MAX and (b) NMMC method with $w_k \!=\! k^\alpha$ ($\alpha \!=\! 3,5,10$) and $p \!=\! 0.01$ when applied to an \emph{`undirectified'} version of Gnutella graph.}\label{fig:mhrw}
    \vspace{-3mm}
\end{figure*}

One may ask what if our NMMC method is applied to an `undirected' graph and/or how it would be compared to the MH algorithm under the undirected graph for graph sampling. We again emphasize that our NMMC method is the first of its kind, applicable to directed graphs or in general when the proposed move is non-reciprocal. Our NMMC method is not meant to be compared with the MH algorithm, as the MH cannot even be applied to directed graphs and our algorithm is not an improvement over the MH algorithm.

Nonetheless, for the sake of completeness, we provide simulation results of the MH algorithm and our NMMC method when applied to an undirected graph. To this end, we use the undirected version of the Gnutella graph (its strongly connected component). For the MH algorithm, we consider the following two popular random-walk versions, since it still depends on the choice of $\q$. One is to use $Q_{ij} \!=\! 1/d_{\max}$ if $(i,j) \!\in\! \E$ and $Q_{ij} \!=\! 0$ if $(i,j) \!\not\in\! \E$ and $i \!\ne\! j$, with $Q_{ii} \!=\! 1 - d_i/d_{\max}$, where $d_i$ is the degree of node $i$ and $d_{\max}$ is the maximum degree of the undirected graph~\cite{VLDB2000,Boyd04-SIAM}. We call this version MH-MAX. The other is to use the simple random walk for $\q$, where $Q_{ij} \!=\! 1/d_i$ if $(i,j) \in \E$ and $Q_{ij} \!=\! 0$, otherwise~\cite{WillingerToN09,GjokaJSAC11,LeeSIGMETRICS12}. We call this version MH-SRW. Note that the transition matrix of the simple random walk is still used as the proposed chain $\q$ under our NMMC method, but the resulting acceptance probability in (\ref{eq:gamma-SRW}) is now given by $\gamma_{ij} \propto \frac{\pi(j)}{\pi(i)}\frac{d_i}{d_j}$, since the underlying graph is \emph{undirected}.

We here consider $\ppi \!=\! \uu$ and use 100 random walks. Letting $\{X_t\}$ be the trajectory by the MH-MAX (or MH-SRW), we consider $\hat{x}(i) \!:=\! \frac{1}{t}\sum_{k=1}^t\idc_{\{X_k = i\}}$ to estimate $\pi(i) \!=\! 1/n$. For performance metrics, we use the TVD and also the normalized root mean square error (NRMSE), which is defined as $\sqrt{\Ex\{(\hat{x}(i) \!-\! 1/n)^2\}}/(1/n)$ for each $i$. We report the average of the NRMSE over all $i \!\in\! \N$. Figure~\ref{fig:mhrw} shows the TVD and NRMSE results (on a log-log scale) of MH-MAX, MH-SRW, and our NMMC method with $w_k \!=\! k^\alpha$ ($\alpha \!=\! 1,3,5,10$) and $p \!=\! 0.01$. We see that all the TVD and NRMSE values decay to zero polynomially fast as we discussed in Section~\ref{se:simu}.

\subsection{PageRank vs. NMMC}\label{subse:page-rank}

Our NMMC method may be reminiscent of the famous PageRank algorithm and its associated `random surfer' model due to the seemingly common nature of following successive outgoing links at random and occasionally performing a random jump to a node in the graph. They are, however, fundamentally different, as explained below.

The random surfer model was introduced in the seminal work by Brin and Page~\cite{PageBrin98,PageBrin99} in order to \emph{interpret} the PageRank vector as the stationary distribution, say $\bm{\pi^*}$, of an `imaginary' random walk (or random surfer) on a Web graph. Namely, a random surfer keeps clicking successive outgoing links uniformly at random and occasionally gets bored of following links then jumps or teleports to a random page (by entering a new destination in the browser's URL line). In the long run, the proportion of the time the random surfer spends on a given page becomes the PageRank score of the page. We again emphasize that the random surfer model was just a random-walk ``interpretation" of the PageRank algorithm, \emph{not} an actual algorithmic operation implementable by actual crawlers on the graph. The rationale behind such a random-walk interpretation is to enable one to resort to the results of random walks on a graph, or the theory of Markov chains, in order to understand the properties of the PageRank algorithm such as convergence properties as in the seminal work~\cite{PageBrin98,PageBrin99} and follow-up studies~\cite{Haveliwala03,Carl2006}.

However, we are not concerned with speeding up the PageRank algorithm nor devising an efficient numerical method for computing $\bm{\pi^*}$. We are rather focused on a somewhat reverse problem, which is to obtain samples from a desired probability distribution $\bm{\pi}$ over a directed graph without requiring the global knowledge of the graph structure. It is also worth noting that unlike the operation of a random surfer/walk in the random surfer model that is to teleport to \emph{any} random node (even including not-yet-visited ones), our NMMC method is only limited to making a random jump (as a process of redistribution) to one of \emph{already-visited} nodes, which can be readily implemented in practice without acquiring the global knowledge. We do not assume a full access to the state space, which would be necessary to realize a teleportation in the random surfer model if one wishes to jump to any arbitrary page (associated with its own URL) in a Web graph, and would be equivalent to assuming that one can draw $i.i.d.$ samples directly from the space.

\subsection{Online In-degree Estimation}\label{subse:online}

We validate the feasibility of our NMMC method with an online estimate of the in-degree of each node, which is possibly latent information. As explained in~\ref{subse:qsd-mapping}, the in-degree information is often readily available as part of a user profile in online social networks such as Google++, Twitter, and Instagram. However, in other types of directed graphs like Web graphs, the in-degree information may not be available but needs to be discovered. To make our NMMC method applicable to the latter, we use a simple way of estimating the in-degree of node $i$. Each crawling agent, in the stage of proposing a movement from $i$ to $j$, can discover the presence of an incoming edge from $i$ into $j$. Whenever a new incoming edge is discovered, the agent increments the value of the in-degree of $j$, say $\hat{d}^-_j$, with $\hat{d}^-_j \!:=\! 1$ initially. All the agents share the estimate of the in-degree of each node. We report the resulting performance of our NMMC method in Figure~\ref{fig:indegree} and observe that our NMMC method performs well even with the simple online in-degree estimation. This manifests the robustness and versatility of our NMMC method in practice.

\begin{figure}[t!]
    \centering
    \vspace{-0mm}
    \subfigure[$\ppi = \uu$; $p = 0.1$]{\includegraphics[width=0.35\textwidth]{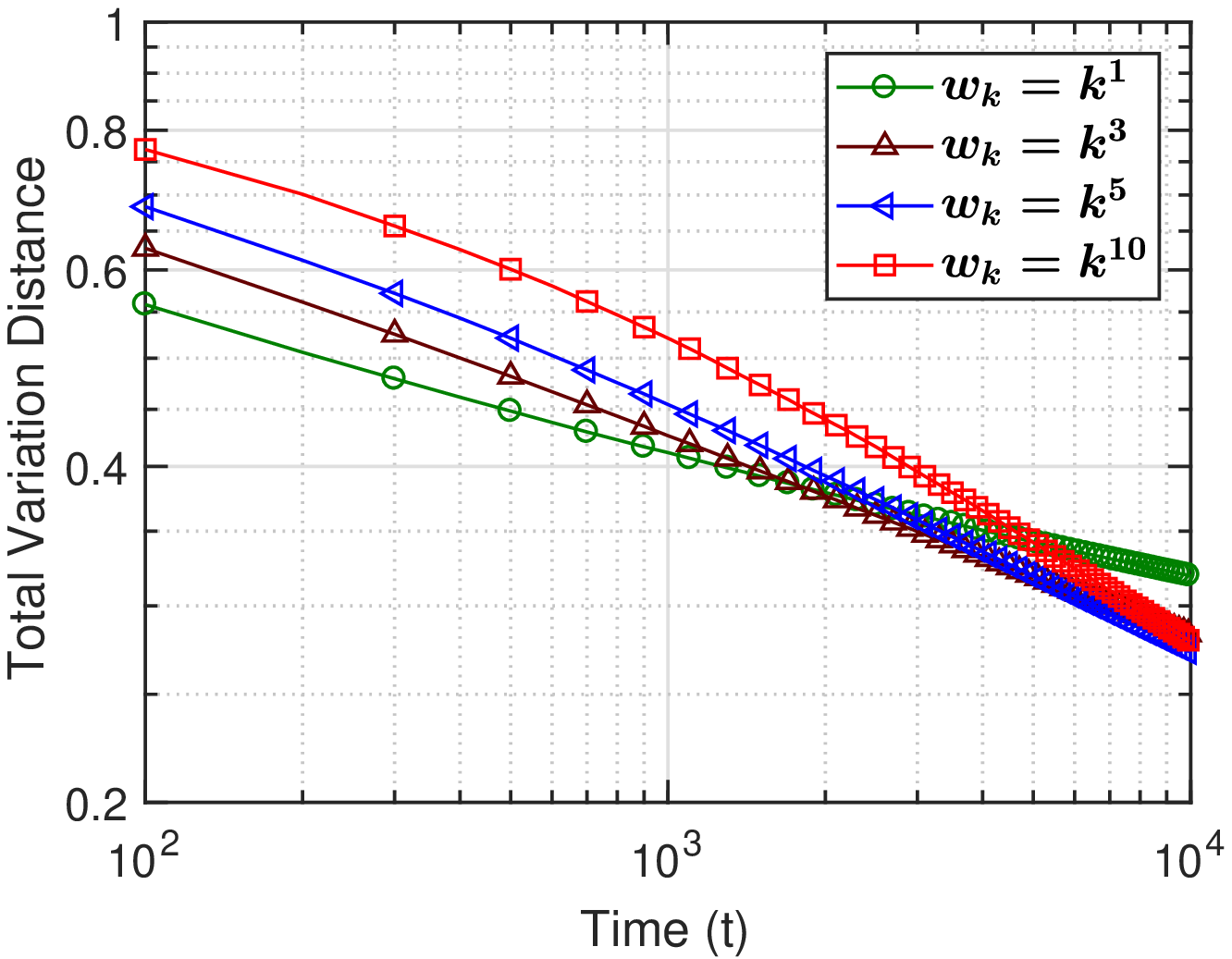}}
    \hspace{6mm}\subfigure[$\ppi = \uu$; $p = 0.01$]{\includegraphics[width=0.35\textwidth]{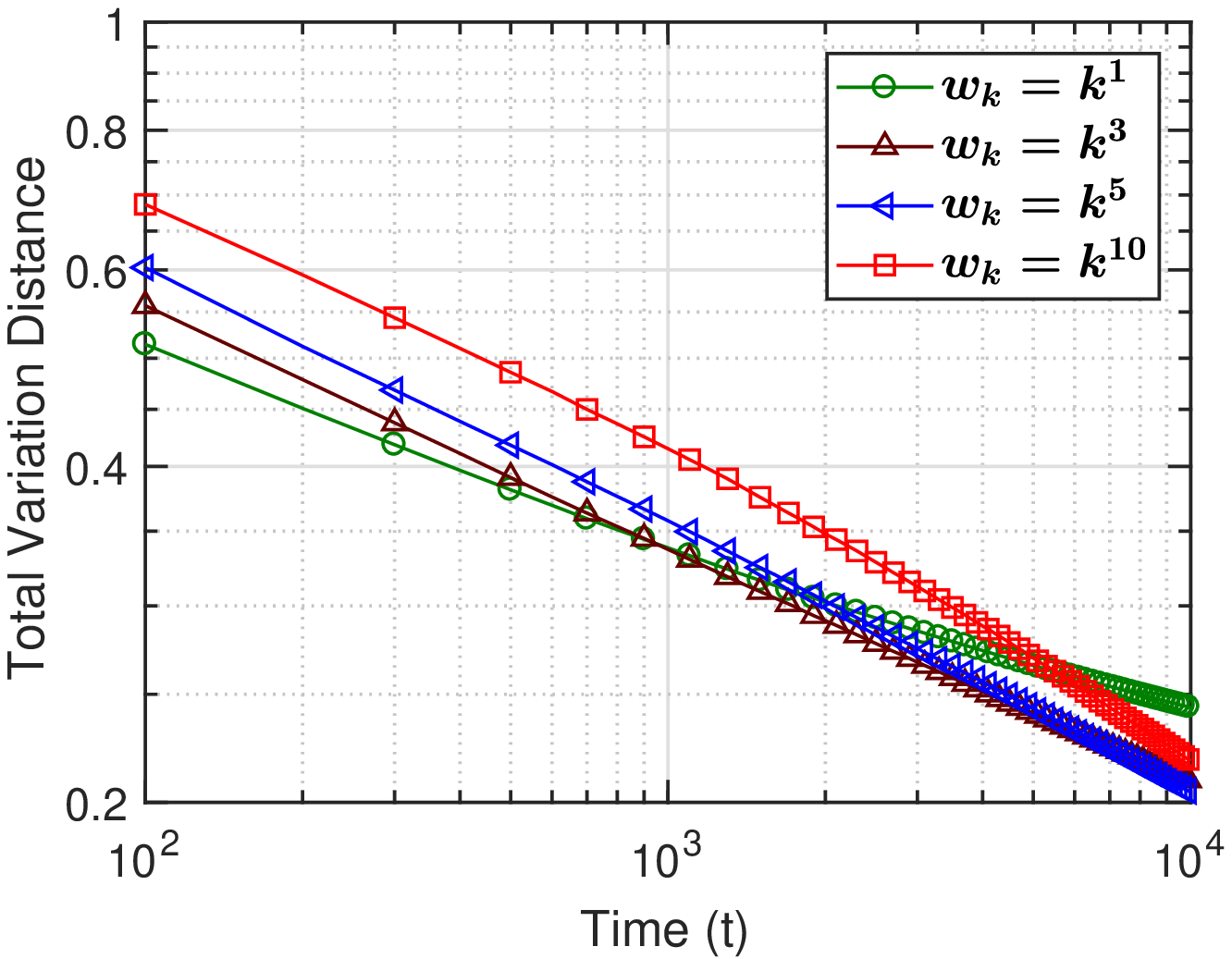}}
    \vspace{-4mm}
    \caption{The TVD results (on a log-log scale) obtained under Gnutella graph for the NMMC method with an online estimate of the in-degree of each node $i \!\in\! \N$.} \label{fig:indegree}
    \vspace{-3mm}
\end{figure}

\subsection{Going Beyond Strongly Connected Directed Graphs}\label{subse:nonscc}

We have assumed that the underlying graph $\GG$ is strongly connected so far. While this assumption has been made mainly for ease of exposition, it may not hold in practice. We thus discuss how we can relax this assumption so that our NMMC method becomes ready for practical use. We first note that the key assumption for the NMMC method is essentially the irreducibility of the `proposed' Markov chain $\q$ on transient states $\N$, as explained in Section~\ref{se:qsd-mapping}. This may not be equivalent to assuming that the underlying directed graph be strongly connected for sampling from an arbitrary probability distribution $\ppi$. It is worth noting that for the estimation of the EVC $\xx$ of the directed graph, the underlying graph needs to be strongly connected by definition, or to properly define the eigenvector centrality as a valid measure~\cite{Carl2000}. In this case, the assumption of strongly connected graphs is not a limitation at all. Thus, we below focus only on achieving $\ppi$.

If the directed graph $\GG$ is not strongly connected, the support of $\ppi$ would need to be on the set of all the nodes, say $\N_S$, that are ``reachable" by traversing directional edges when starting from an initial node, or a set of initial nodes, say $S$, that are reachable among themselves. Note that the ``unreachable" nodes $\N \setminus \N_S$ are simply invisible and can safely be ignored. The technical changes are then how to deal with (i) `dangling' nodes, which are the nodes with no forward links, and (ii) a set of nodes that are reachable but without outgoing edges to the other nodes in $\N_S$. Note that these two problems were resolved in the PageRank algorithm by adopting the teleportation to any node in the graph, yet to numerically compute the (unknown) PageRank scores. While such global teleportation is infeasible for sampling in practice, its \emph{personalized} PageRank version is practically implementable and can be adopted into our NMMC method.\footnote{The personalized PageRank scores are still unknown and need to be computed by the PageRank algorithm (or the power method), as long as the PageRank scores themselves are concerned.}

Fix the set of initial nodes $S$. The state space $\N$ then reduces to $\N_S$. For the proposed chain $\q$ on $\N_S$, we consider the following transition matrix: For $i,j \in \N_S$,
\begin{equation}
Q_{ij}  = (1-p)\frac{1}{|S|} + p\frac{1}{|S|} \idc_{\lt\{d_i^+ =0\rt\}}  + p \frac{A_{ij}}{d_i^+} \idc_{\lt\{d_i^+>0\rt\}}, \label{ppr}
\end{equation}
for some constant $p \!\in\! (0,1)$. At each node $i$, with probability $p$, the agent follows one of its outgoing edges (if exists) uniformly at random, and jumps to one of the nodes in $S$ uniformly with probability $1-p$. If node $i$ is a dangling node with no outgoing edges ($d_i^+ \!=\! 0$), then the agent always jumps to a random node in $S$. It is not difficult to see the irreducibility of $\q$ on $\N_S$ for our NMMC method. Then, the acceptance probability $\gamma_{ij}$ in (\ref{r-qij}) can properly be obtained. In Figure~\ref{fig:nonscc}, we present the TVD results for estimating the uniform $\uu$ on $\N_S$ using the dynamic NMMC method with $\q$ in (\ref{ppr}), where $\N_S$ is the set of nodes that are reachable from any node in the LSCC of Gnutella graph and $|\N_S| \!=\! 8,566$. We here use 300 independent crawling agents starting from 300 fixed, distinct nodes $S$ and $p \!=\! 0.95$. Figure~\ref{fig:nonscc} demonstrates the feasibility of our NMMC method even when the underlying graph is not strongly connected.

\begin{figure}[t!]
    \centering
    \vspace{-0mm}
    \subfigure[$\ppi = \uu$; $p = 0.1$]{\includegraphics[width=0.35\textwidth]{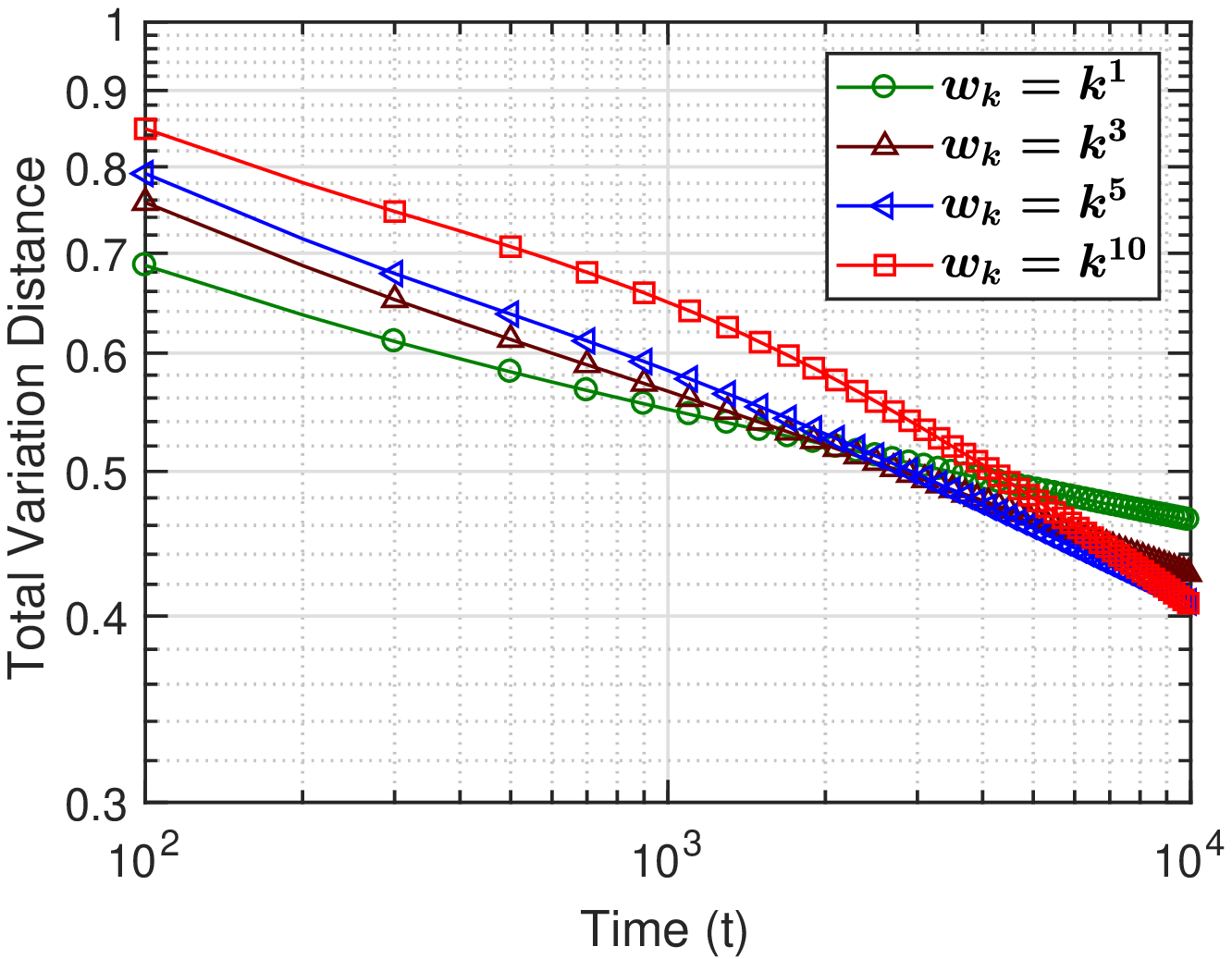}}
    \hspace{6mm}\subfigure[$\ppi = \uu$; $p = 0.01$]{\includegraphics[width=0.35\textwidth]{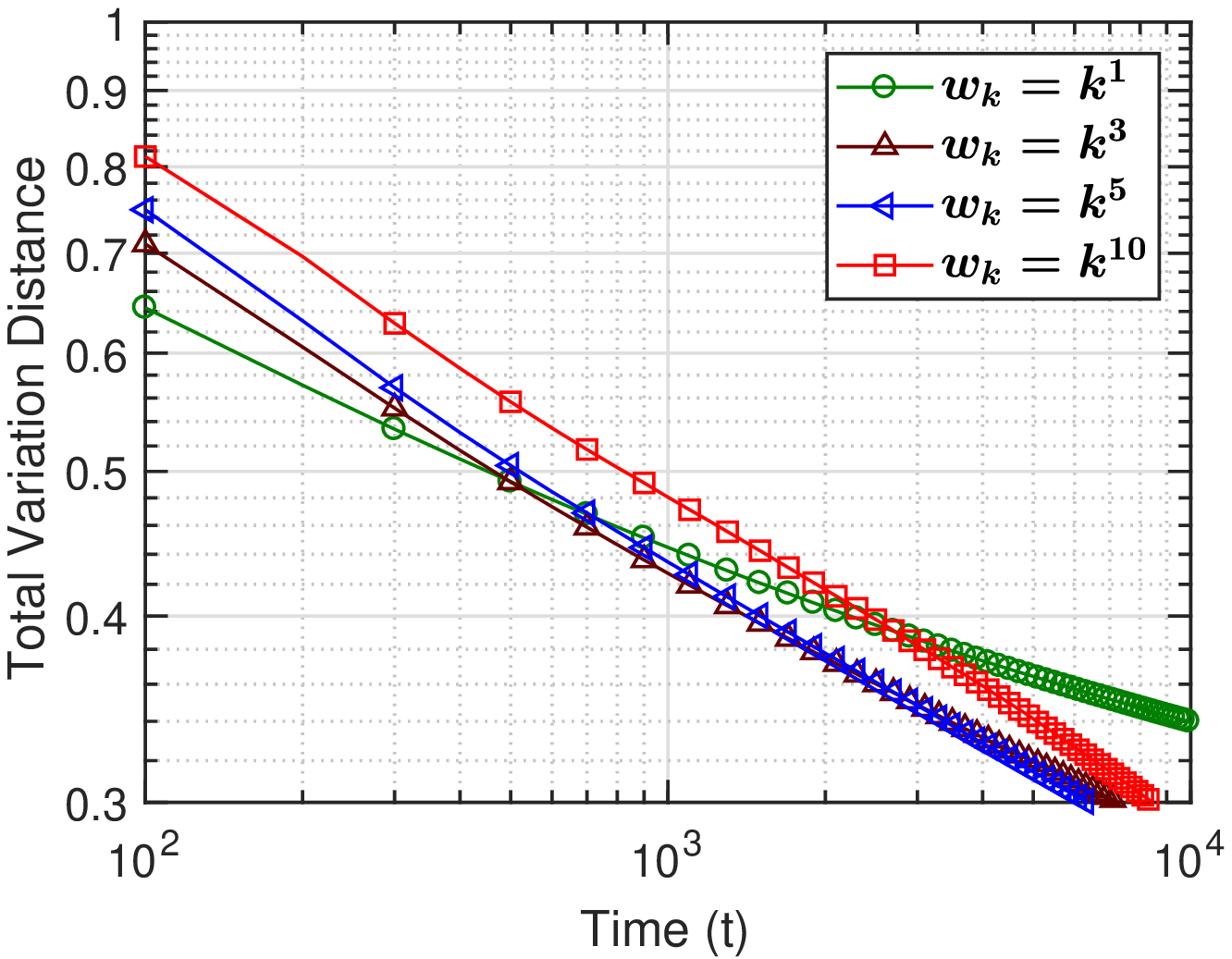}}
    \vspace{-4mm}
    \caption{The TVD results (on a log-log scale) of our NMMC method on the set of nodes $\N_S$, which are reachable from any node in the LSCC of Gnutella graph.} \label{fig:nonscc}
    \vspace{-3mm}
\end{figure}

\subsection{Query Cost vs. Performance for Graph Sampling}\label{subse:query-cost}

So far, all the simulation results have been shown with respect to the number of time steps, since they have the primary purpose of numerically validating the convergence of our NMMC method and evaluating its convergence speed. Recall that the ergodic theorem in (\ref{ergodic}) and the CLT in (\ref{clt}) hold for ergodic Markov chains with respect to the number of time steps. Nonetheless, when it comes to graph sampling via crawling, `query cost' becomes practically important. In general, crawling real-world networks is done through public yet restrictive API interfaces provided by the networks or their Web interfaces (by following URLs and hyperlink structures). Such external access is often rate-limited due to security concerns, privacy protection, infrastructure burden and commercial purposes~\cite{Vesdapunt-CIKM16,Xin-Eun-Infocom17}. For example, Twitter allows only 15 API requests to retrieve IDs of a user's followers or IDs of the friends that the specific user is following every 15 minutes~\cite{twitter_rate}.

From a practical point of view, it would be informative to see the performance of our NMMC method with respect to the query cost in terms of the API calls (or HTTP requests) required. We below present the performance of NMMC method with respect to the query cost under Gnutella and Slashdot graphs. To this end, we adopt the definition of the query cost commonly used in the literature~\cite{GjokaJSAC11,Bruno-Infocom12,History-vldb15,Rewiring-database16}, which is the number of \emph{unique} queries required. In other words, API calls are spent only when visiting a node (or querying its page/profile) \emph{for the first time}, and revisiting the same node since then  (duplicate queries) can be done immediately by retrieving the nodal information from local cache without consuming the query rate limit. This indicates that the redistribution mechanism of our NMMC method is achieved at no additional cost as long as query cost is concerned. Thus, the `due' redistributions are only advantageous to shape up the historical empirical distribution $\hmmu_t$ toward the target QSD. The local cache can also be easily shared among multiple crawling agents.

\begin{figure*}[t!]
    \centering
    \vspace{-0mm}
    \subfigure[$\ppi = \uu$; $p = 0.01$]{\includegraphics[width=0.31\textwidth]{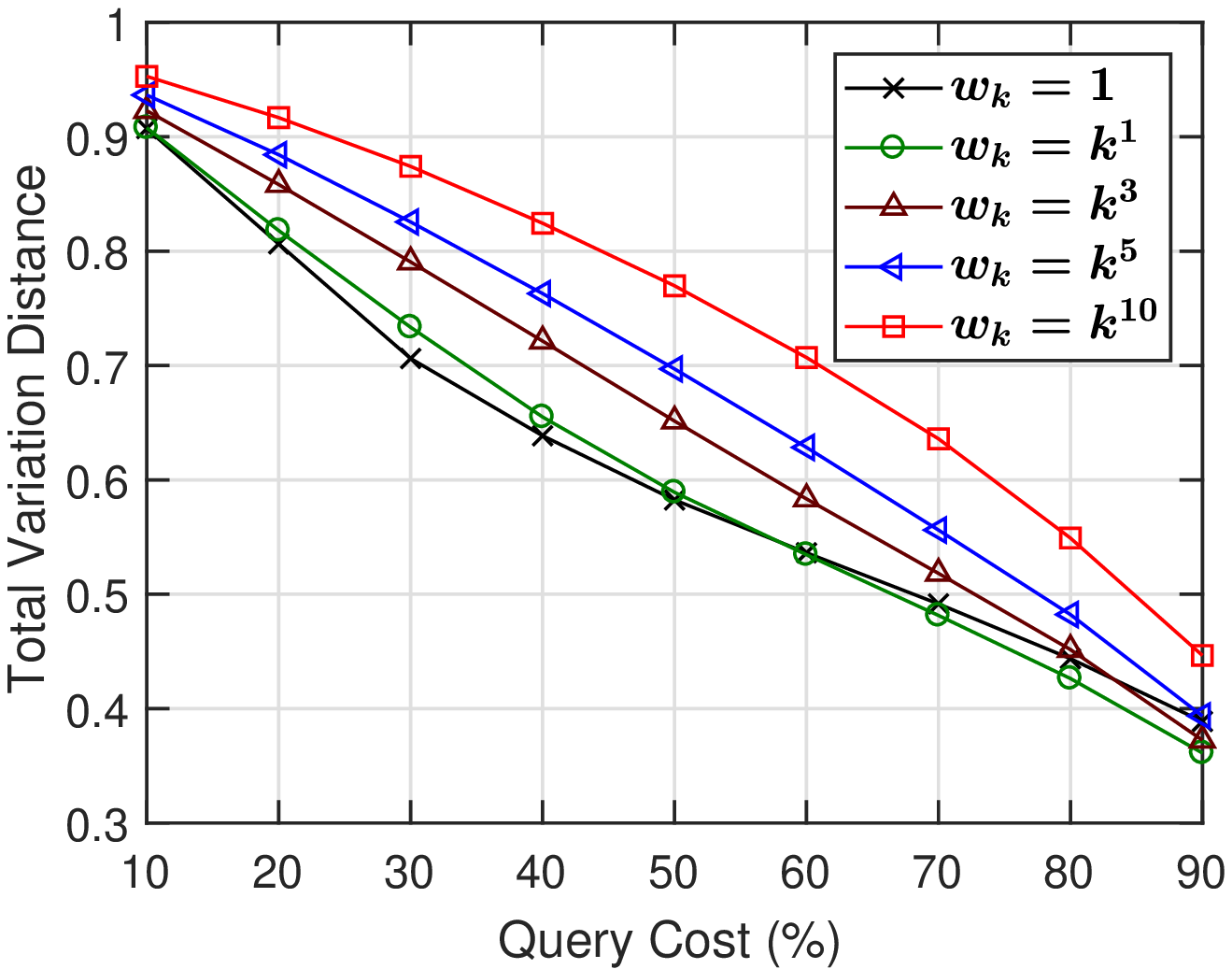}}
    \hspace{1mm}
    \subfigure[$\pi(i) \propto d^-_i$; $p = 0.01$]{\includegraphics[width=0.31\textwidth]{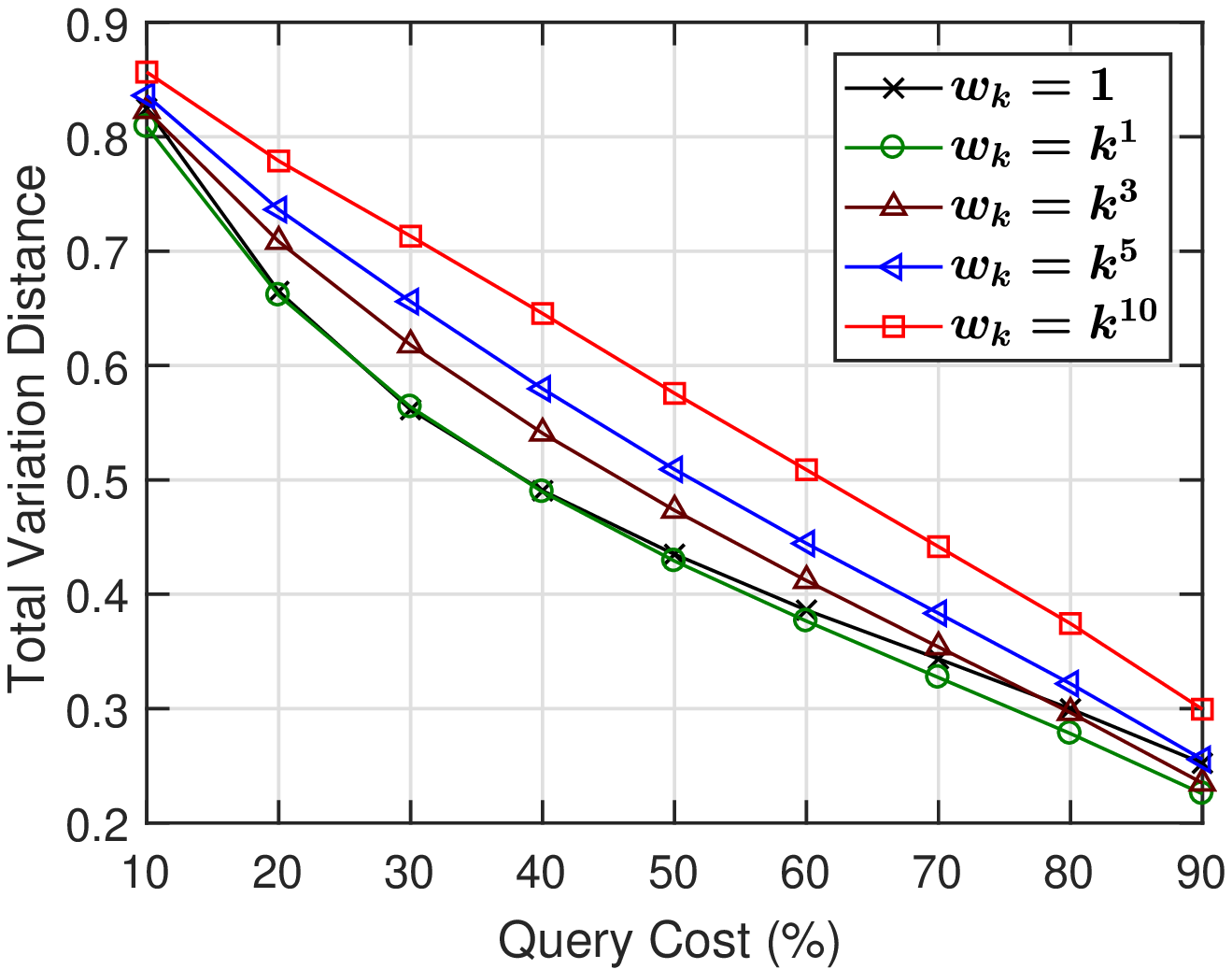}}
    \hspace{1mm}
    \subfigure[EVC $\xx$; $p = 0.01$]{\includegraphics[width=0.31\textwidth]{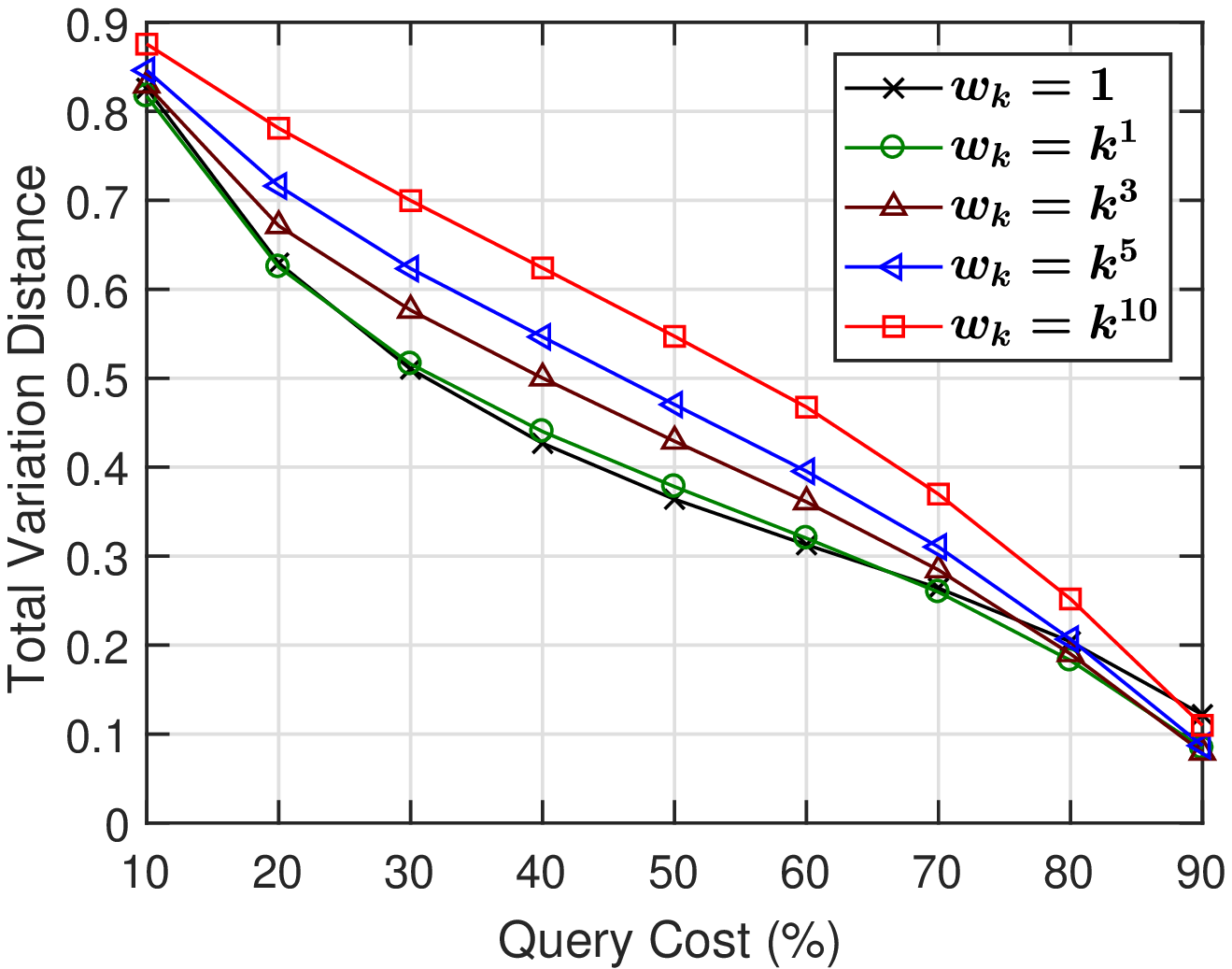}}
    \vspace{-3mm}
    \caption{The TVD results obtained under Gnutella graph as the query cost varies.} \label{fig:gnutella05-api}
    \vspace{-3mm}
\end{figure*}
\begin{figure*}[t!]
    \centering
    \vspace{-0mm}
    \subfigure[$\ppi = \uu$; $p = 0.01$]{\includegraphics[width=0.31\textwidth]{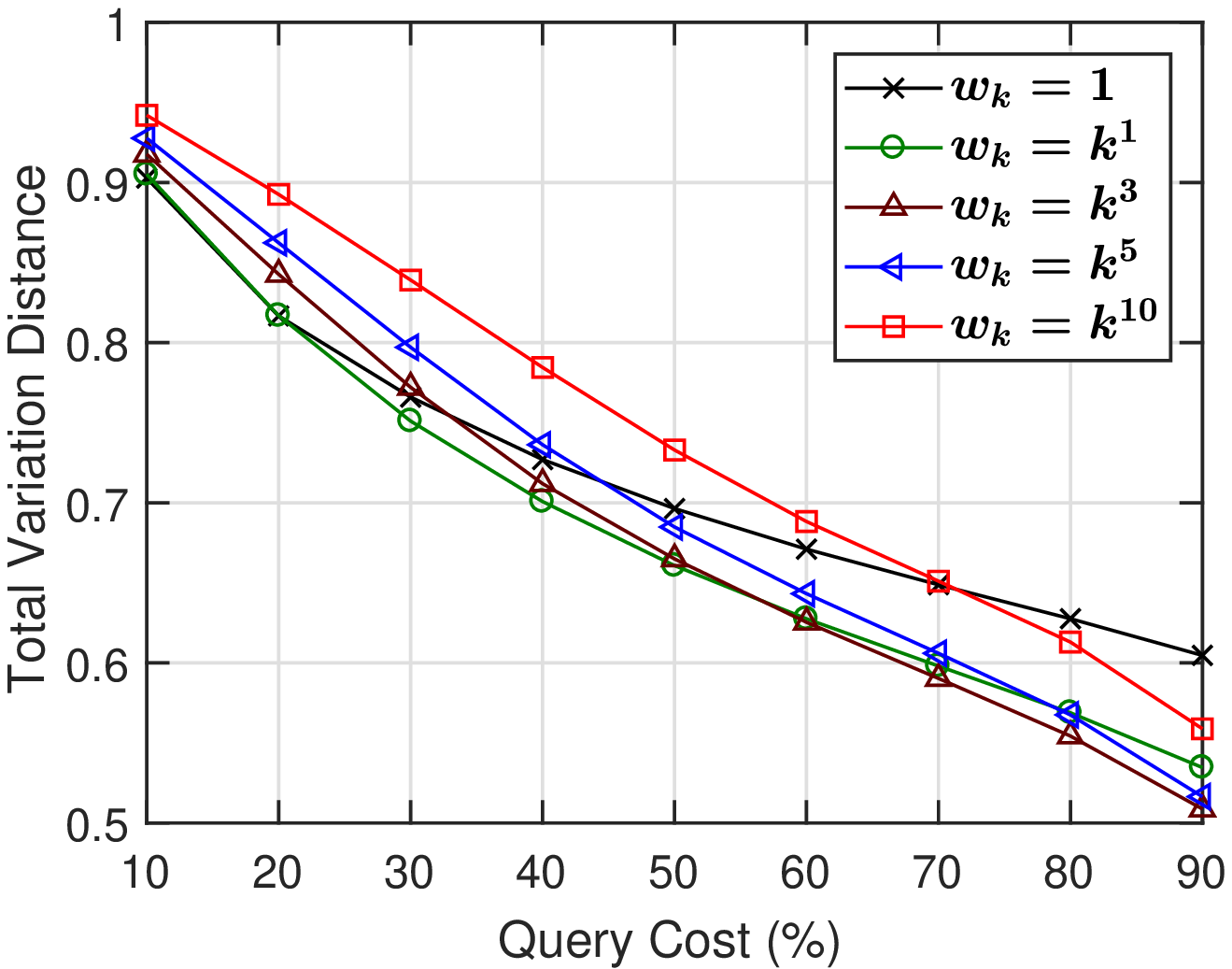}}
    \hspace{1mm}
    \subfigure[$\pi(i) \propto d^-_i$; $p = 0.01$]{\includegraphics[width=0.31\textwidth]{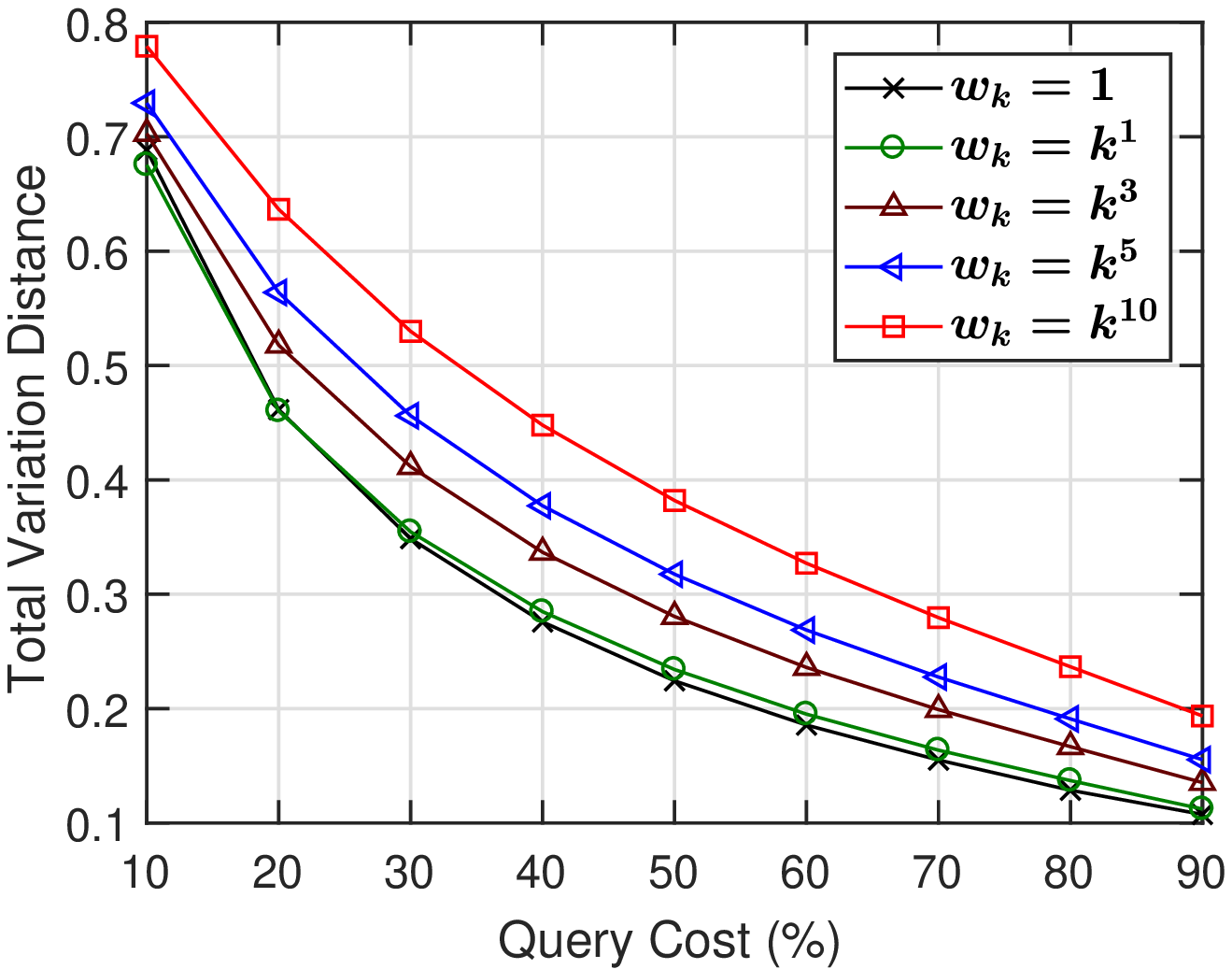}}
    \hspace{1mm}
    \subfigure[EVC $\xx$; $p = 0.01$]{\includegraphics[width=0.31\textwidth]{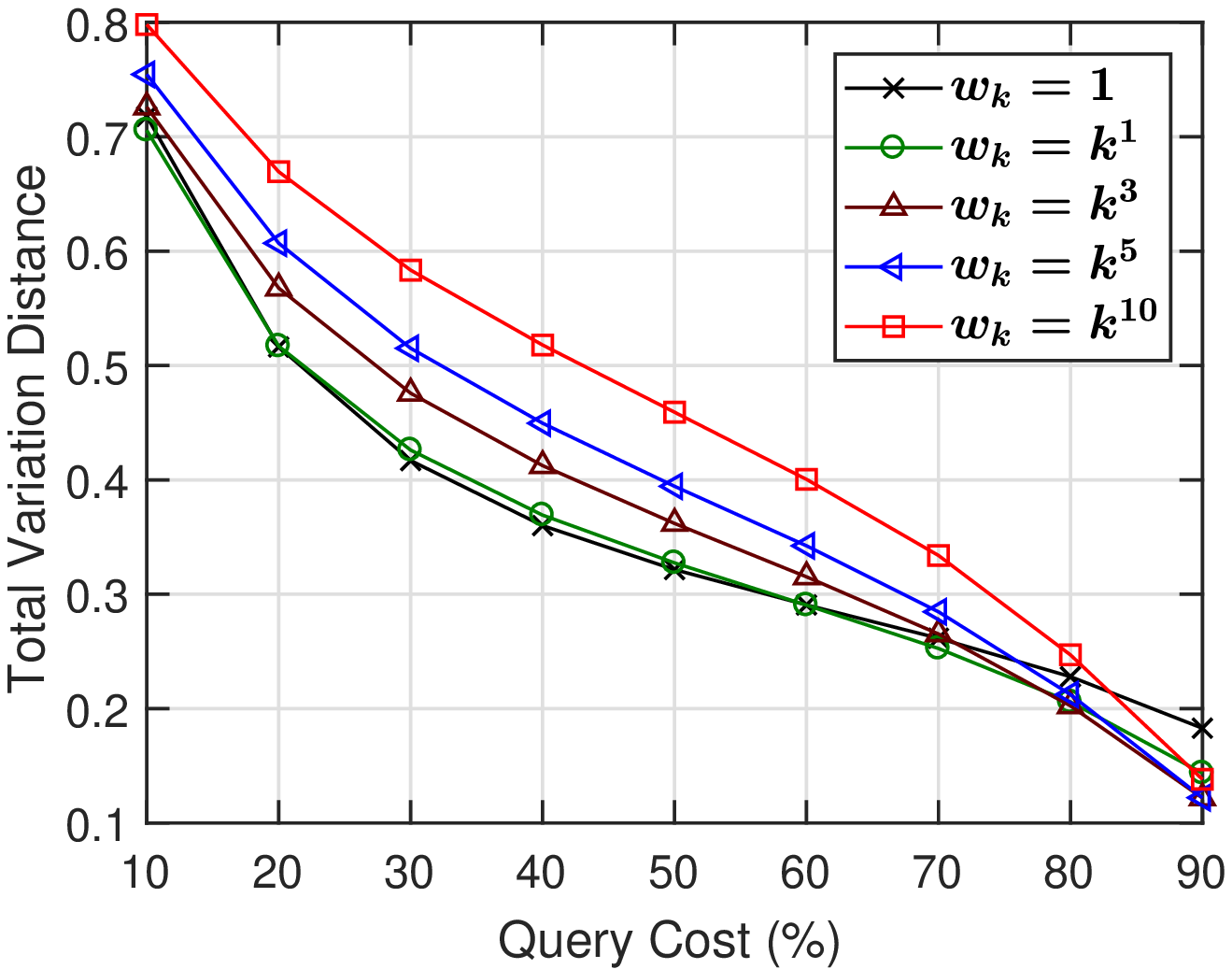}}
    \vspace{-3mm}
    \caption{The TVD results obtained under Slashdot graph with varying degrees of the query cost.} \label{fig:slashdot-api}
    \vspace{-3mm}
\end{figure*}

In Figures~\ref{fig:gnutella05-api} and~\ref{fig:slashdot-api}, we present the TVD results of the dynamic NMMC method with $k^\alpha$ ($\alpha \!=\! 0, 1,3,5,10$) and $p \!=\! 0.01$ under Gnutella and Slashdot graphs, respectively, as the query cost varies. Here, the query cost is defined as the number of unique queries normalized by the graph size (LSCC size) and given in percentage. The same numbers of agents are used as before for Figures~\ref{fig:gnutella05} and~\ref{fig:slashdot}, respectively. We see that the choice of $w_k \!=\! k$ exhibits the overall fastest decaying rate in its TVD curve for all three cases of the target QSD under both Gnutella and Slashdot graphs. This also indicates that higher redistribution weights toward more recently-visited nodes is no longer beneficial in improving the convergence speed when the query cost is concerned.

In a similar vein, we showed the TVD results of our NMMC method until $10^4$ time steps in Figure~\ref{fig:gnutella05}, which is greater than the LSCC size of the Gnutella graph. Considering 100 agents used, one might expect that the entire graph is already visited by that time. However, it has not been the case, since there are a lot of redistributions that are counted toward the number of time steps, but not toward the number of unique queries. The multiple agents can also visit the already-visited nodes many times, which is again not counted toward the number of unique queries. This is confirmed in Figure~\ref{fig:gnutella05-unique}, which shows the number of unique queries, which is normalized by the LSCC size, for all three cases of the target QSD on the Gnutella graph using the dynamic NMMC method with $k^\alpha$ ($\alpha \!=\! 0,1,3,5,10$) and $p \!=\! 0.01$.

\begin{figure*}[t!]
    \centering
    \vspace{-0mm}
    \subfigure[$\ppi = \uu$; $p = 0.01$]{\includegraphics[width=0.31\textwidth]{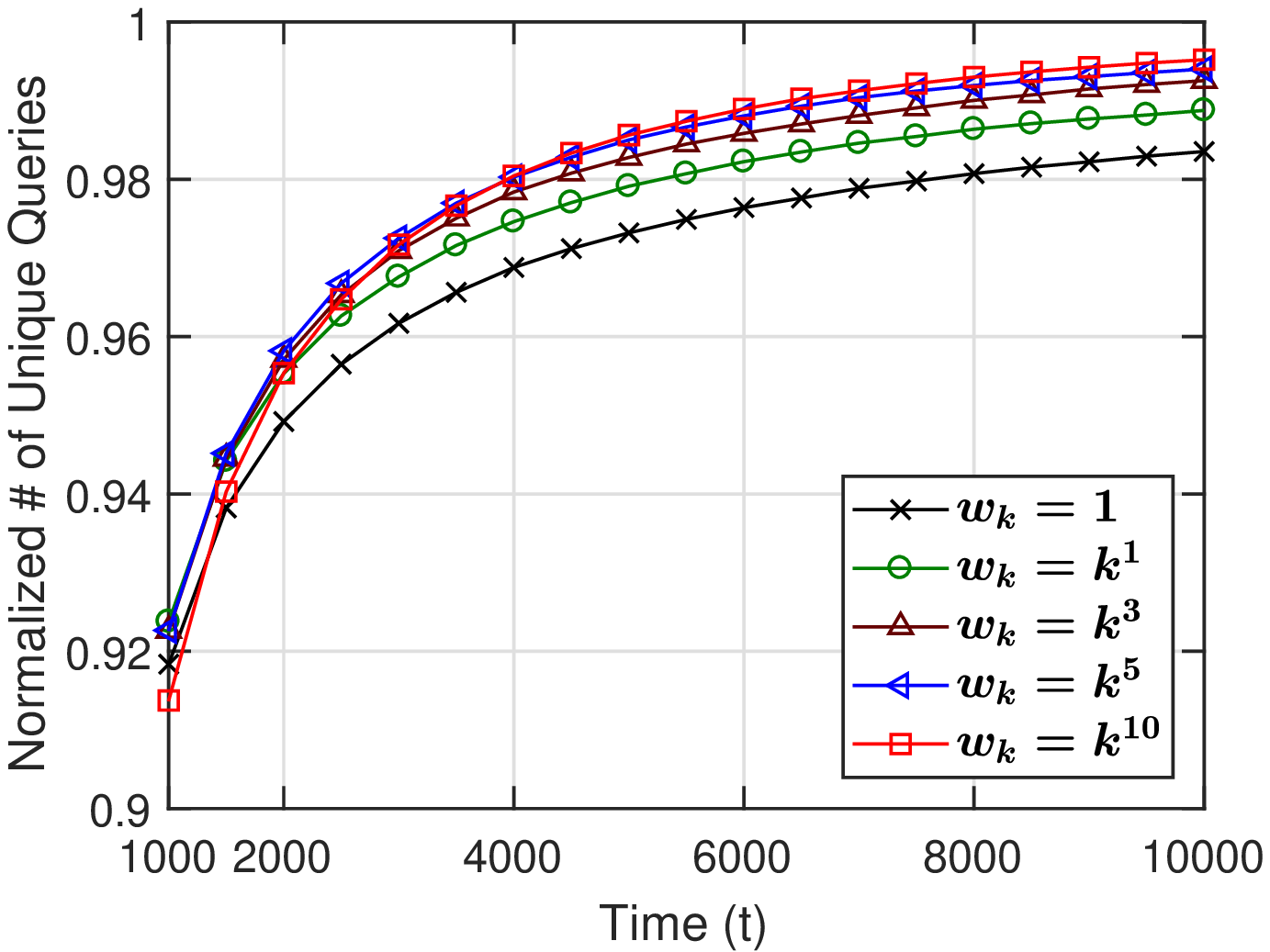}}
    \hspace{1mm}
    \subfigure[$\pi(i) \propto d^-_i$; $p = 0.01$]{\includegraphics[width=0.31\textwidth]{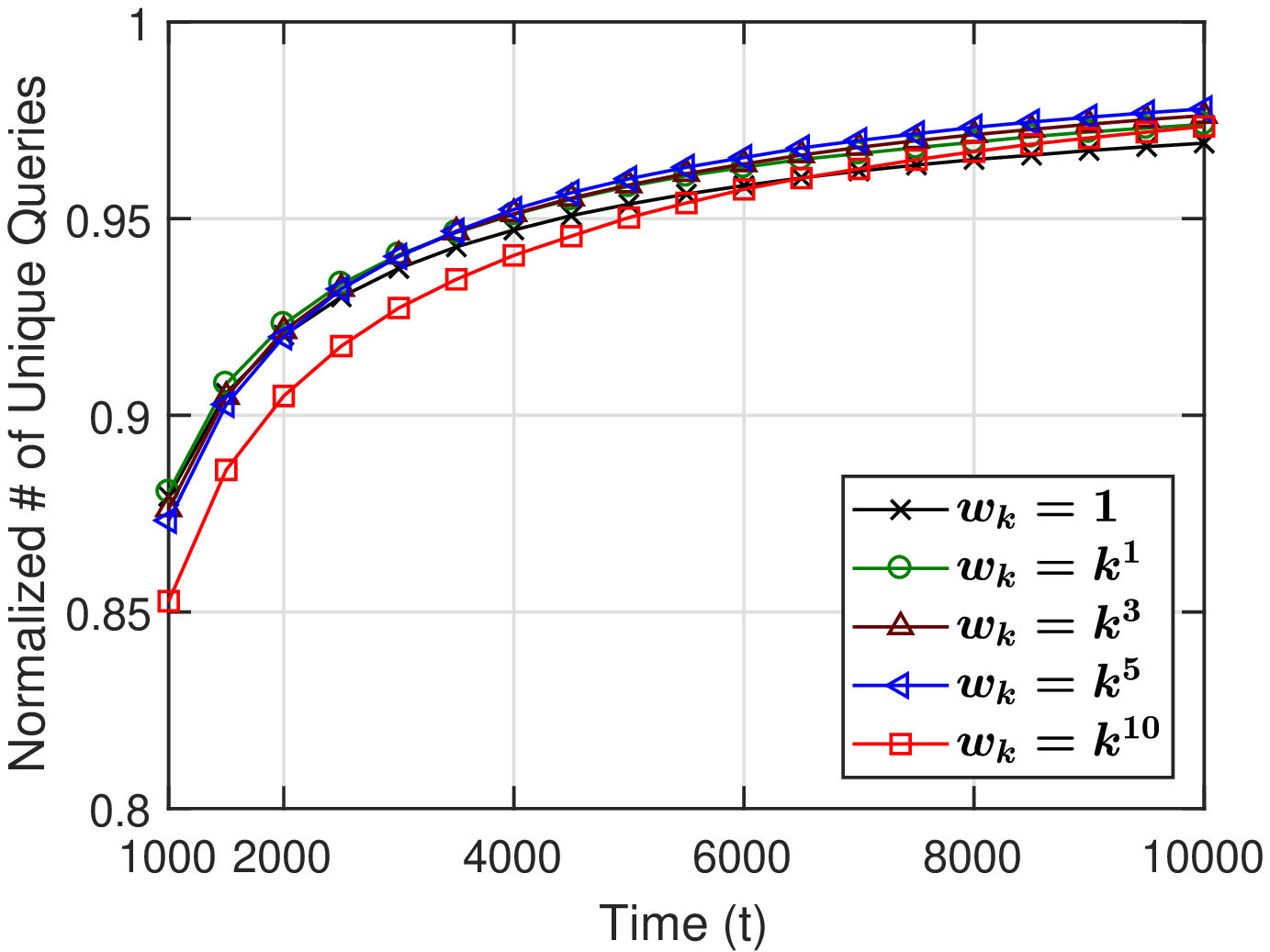}}
    \hspace{1mm}
    \subfigure[EVC $\xx$; $p = 0.01$]{\includegraphics[width=0.31\textwidth]{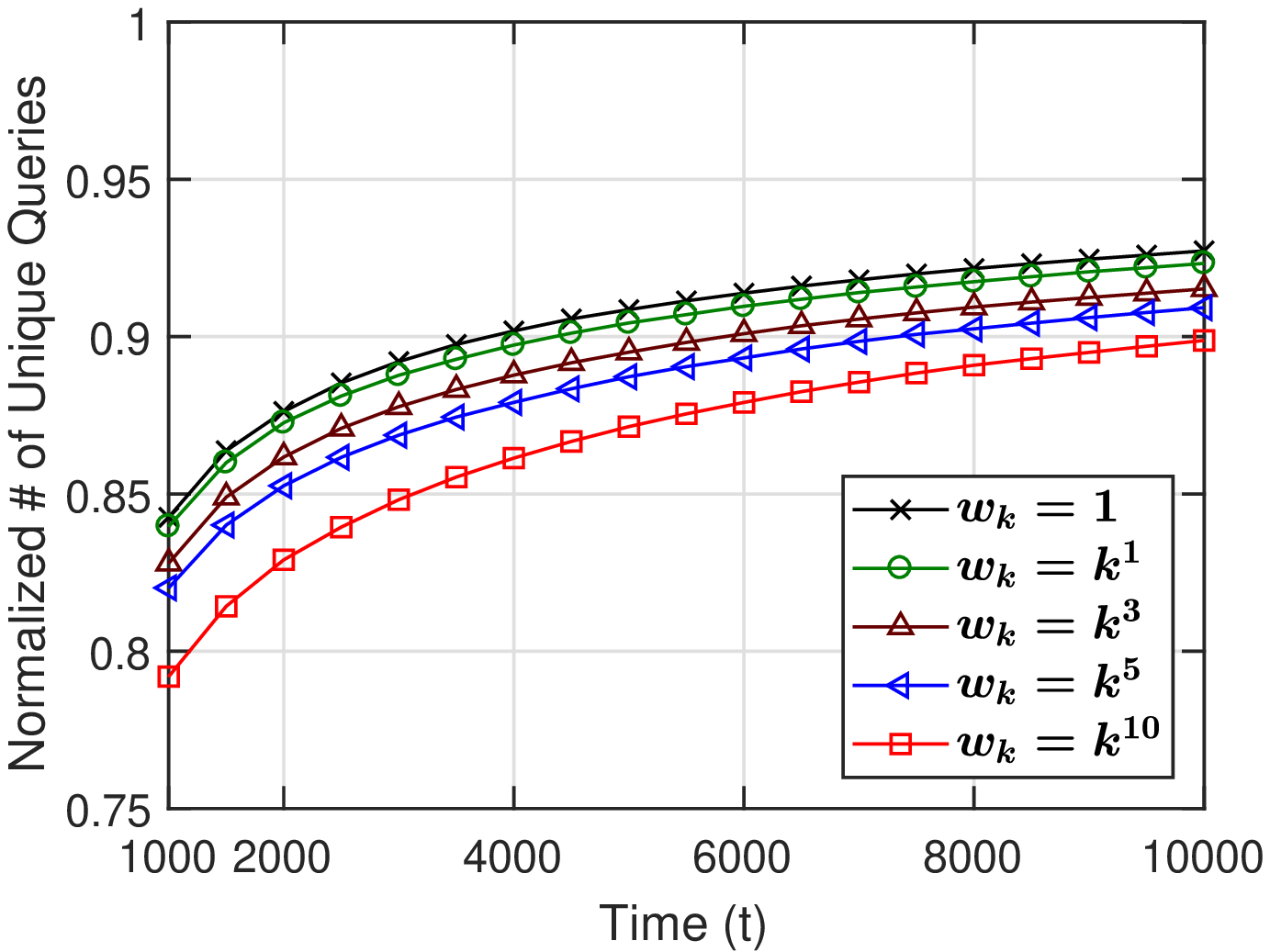}}
    \vspace{-3mm}
    \caption{The normalized number of unique queries used by the dynamic NMMC method on Gnutella graph.} \label{fig:gnutella05-unique}
    \vspace{-3mm}
\end{figure*}

\subsection{Comparison with the Latest Approach for Graph Sampling}\label{subse:comparison}

Lastly, we provide a comparison with the latest approach in~\cite{Bruno-Infocom12} for sampling directed graphs. Among only a few studies available in the literature on sampling from directed graphs, which shall be explained later in Section~\ref{se:related}, the latest practical solution is the one in~\cite{Bruno-Infocom12}, which is called Directed Unbiased Random Walk (DURW) algorithm. While our NMMC method is capable of achieving any distribution $\ppi$ on a directed graph and estimating its EVC $\xx$, which is much more than just achieving `uniform' samples from the directed graph, the comparison of our NMMC method to the DURW algorithm would be interesting in the context of uniform graph sampling and is provided below. We here focus on estimating $\ppi \!=\! \uu$.

The DURW algorithm has two main parts. The first part is to transform the underlying directed graph to an \emph{undirected} graph. At each newly visited node, it keeps outgoing edges only to `not-yet-visited' nodes and converts them as undirected edges. This edge-pruning idea was originally proposed in~\cite{VLDB2000}. Upon the transformed, undirected graph, the second part is to use (i) a weighted random walk with random jumps to \emph{everywhere} in the graph and (ii) the importance sampling. Specifically, the weighted random walk at node $i$ chooses one of its neighbors uniformly at random, with probability $\frac{d_i}{d_i + w}$, and jumps (or teleports) to any random one in the graph, with probability $\frac{w}{d_i + w}$, where $w$ is the weight parameter for choosing a random jump and $d_i$ is the degree of $i$. While the former operation consumes a unit cost if a new node is sampled, the latter consumes the cost $c \!>\! 1$. For a directed social network, assuming that its user-ID space is accessible, the cost $c$ of a random jump is the number of API calls required on average to find a valid ID for sampling. It clearly depends on how sparse the user-ID space may be, which is unknown and uncontrollable. Note that the random jumps may not be feasible at all in practice, e.g., for sampling Web graphs. In addition, letting $\bm{\pi'}$ be the stationary distribution of the weighted random walk with random jumps, one can see that $\pi'(i)$ is proportional to $d_i \!+\! w$. The importance sampling is then applied to correct the bias from the non-uniform $\bm{\pi'}$. Thus, letting $\{X_t\}$ be the trajectory by the DURW algorithm, it is to use the ratio estimator $\hat{x}(i) \!:=\! \frac{\sum_{k=1}^t\idc_{\{X_k = i\}}/\pi'(X_k)}{\sum_{k=1}^t 1/\pi'(X_k)}$ to estimate $\ppi \!=\! \uu$. Note that, instead of $\bm{\pi'}$, its unnormalized version can be used in practice. We refer to~\cite{Bruno-Infocom12} for more details.

\begin{figure*}[t!]
    \centering
    \vspace{-0mm}
    \subfigure[TVD; NMMC]{\includegraphics[width=0.31\textwidth]{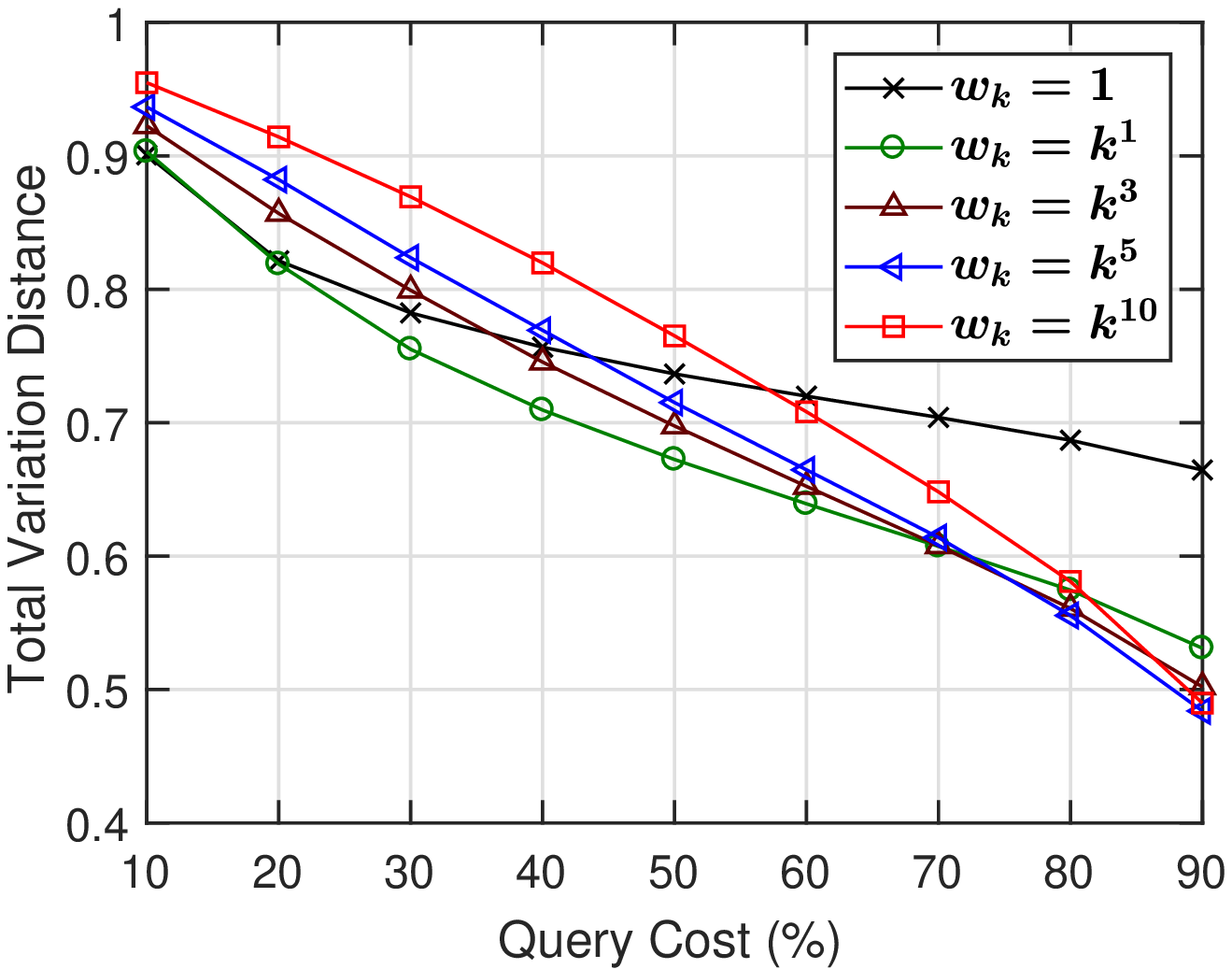}}
    \hspace{6mm}
    \subfigure[TVD; DURW]{\includegraphics[width=0.31\textwidth]{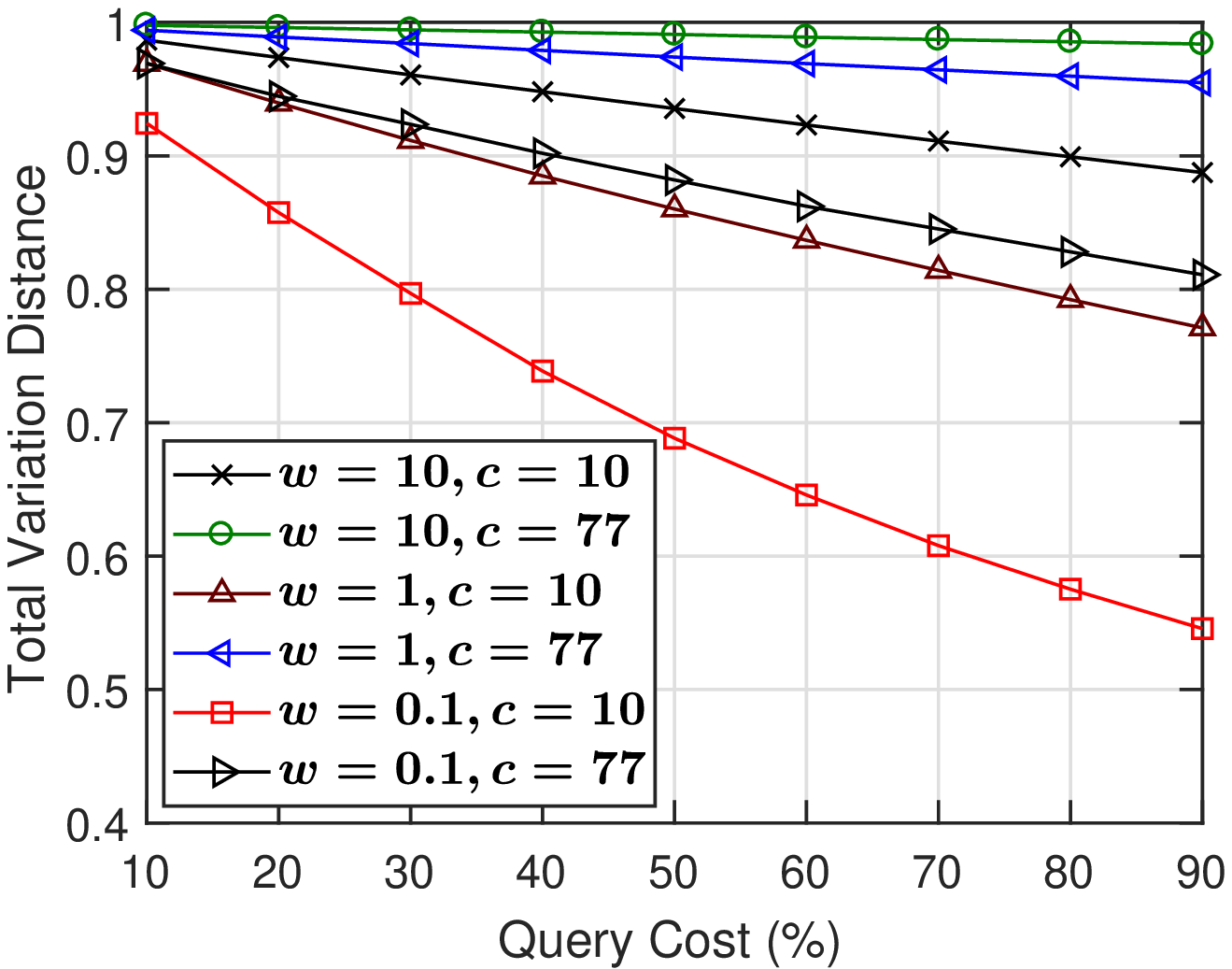}}
    \\\vspace{-1mm}
    \subfigure[NRMSE; NMMC]{\includegraphics[width=0.31\textwidth]{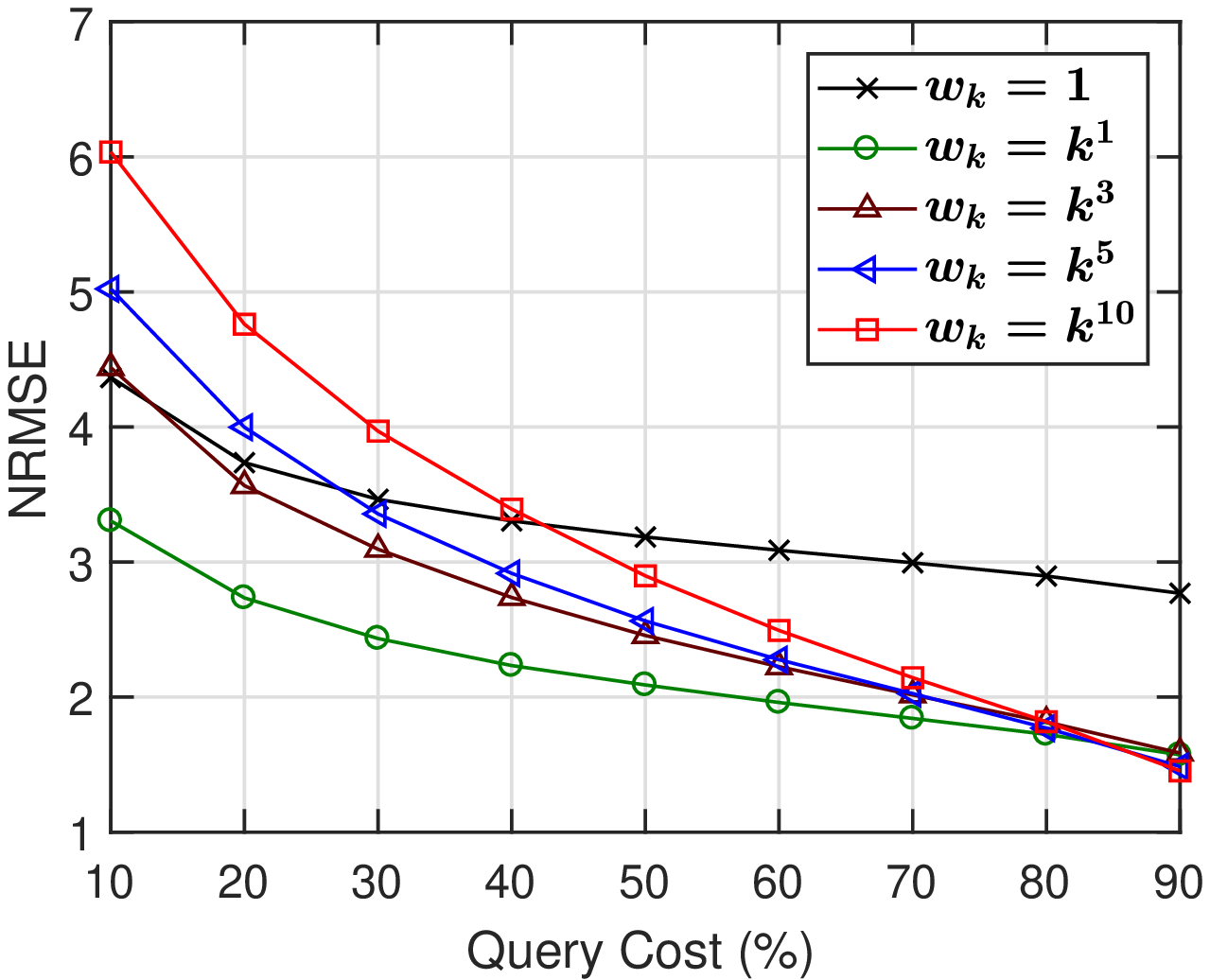}}
    \hspace{6mm}
    \subfigure[NRMSE; DURW]{\includegraphics[width=0.31\textwidth]{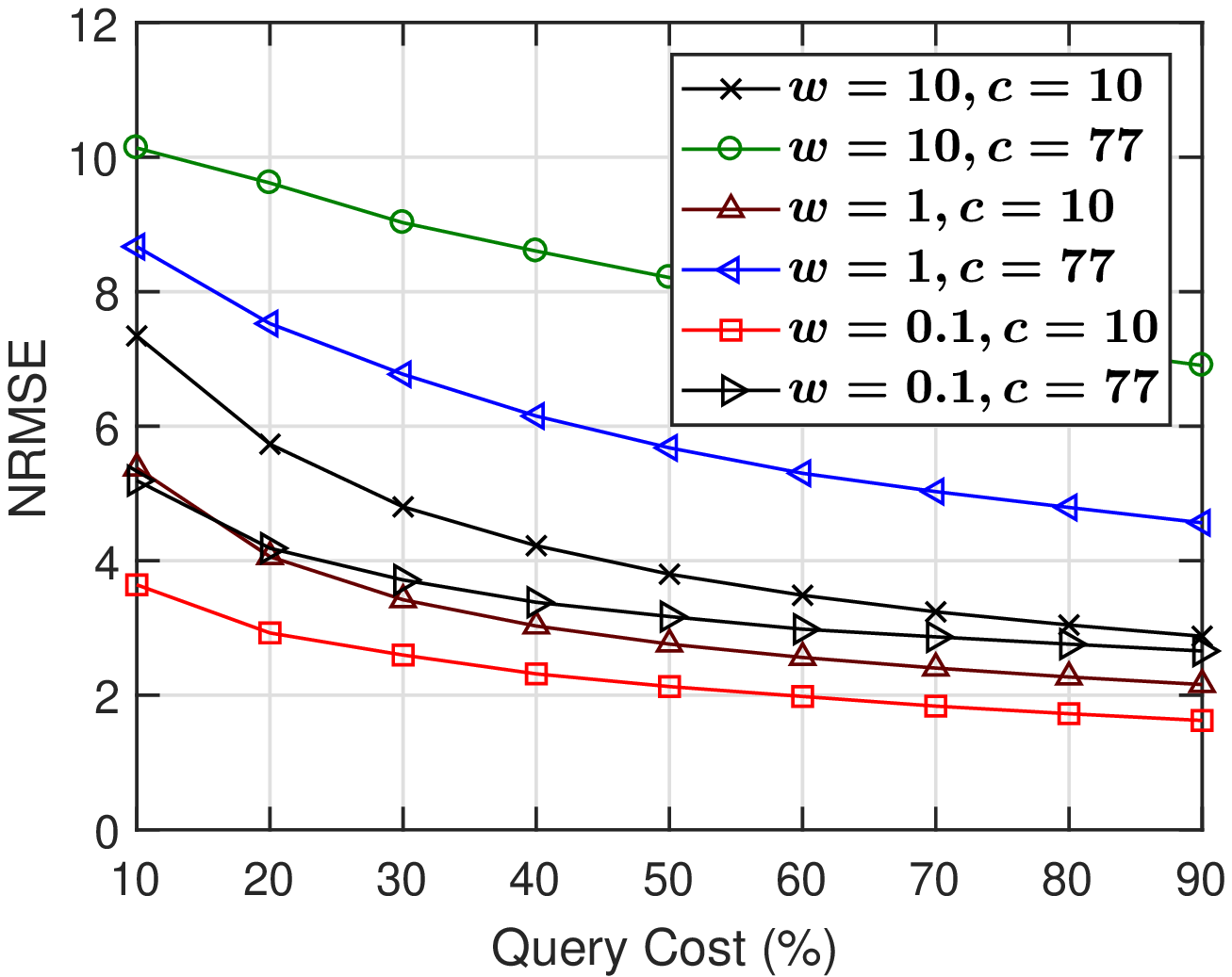}}
    \vspace{-3mm}
    \caption{The TVD and NRMSE results of our dynamic NMMC method and the DURW algorithm under Gnutella graph.}\label{fig:durw}
    \vspace{-4mm}
\end{figure*}

Figure~\ref{fig:durw} shows the TVD and NRMSE results of our NMMC method with $w_k \!=\! k^\alpha$ ($\alpha \!=\! 0, 1,3,5,10$) and $p \!=\! 0.01$, and the DURW algorithm with various choices of $w$ and $c$ under Gnutella graph, which are from~\cite{Bruno-Infocom12}. For example, the choice of $w \!=\! 0.1$ behaves more like the pure simple random walk, while the larger values of $w$ put more weights on random jumps. In addition, $c \!=\! 10$ means 10 API calls spent on average to sample a valid ID (or to teleport to a random node). Similarly for $c \!=\! 77$. In addition, we here consider a single random crawler. It is because in the DURW algorithm, each crawler has to build its own state space (or the transformed, undirected graph) as it explores the graph, which is clearly time-varying and depends on its sample path, hence the state space can never be shared with other crawlers. The query cost is again defined as the number of unique queries normalized by the graph size (LSCC size) and given in percentage. We report the average of the NRMSE over all $i \!\in\! \N$.

We see that our NMMC method generally performs better than the DURW algorithm. A possible explanation would be that the way to transform the underlying directed graph to an undirected graph makes the transformed one `tree-like'. Thus, traversing the transformed graph could get slower. While random jumps can speed up the graph traversal, the high cost of each random jump eventually outweighs the benefit of moving to everywhere in the graph. We also observe that carefully choosing the value of $w$ with given (yet unknown) cost $c$, the DURW algorithm would be comparable to our NMMC method. In other words, the DURW algorithm requires a good choice of $w$ a priori against unknown $c$, before running the random crawler to explore a directed graph by `undirectifying' directed edges on the fly. In contrast, our NMMC method performs well for a wide range of values of $\alpha$ for $w_k \!=\! k^\alpha$.

We finally emphasize that while we take (directed) graph sampling as a primary application of our NMMC method in this paper, the NMMC method can potentially find much broader scope of applications beyond the graph sampling as we have witnessed the widespread applications of the MH algorithm in many disciplines. The potential broader impact stems from the versatile applicability of the NMMC method for generating \emph{any} arbitrary probability distribution over the state spaces violating the reciprocity condition, where the MH algorithm and other standard MCMC methods are no longer applicable.

\section{Related Work}\label{se:related}

There have been a few studies on achieving uniform samples from \emph{directed} graphs, ranging from traditional Web graphs~\cite{VLDB2000,WWW2000,AAAI2001} to recent online social networks~\cite{Wang-Sigcomm10,Bruno-Infocom12,Gabielkov-Sigcomm14}. The key underlying methods have been mainly (explicitly or implicitly) the MH algorithm and importance sampling for Markov chains as in (\ref{importance}). The `non-uniform' random walks, along with the importance sampling, are typically the simple random walks and (personalized) PageRank walks~\cite{WWW2000,AAAI2001}. Unfortunately, however, their stationary distributions are simply \emph{unknown}, and the only viable approach in~\cite{WWW2000,AAAI2001} has been to empirically ``estimate" the stationary distribution (e.g., based on the visit counts obtained by running a separate random walk), which is still known to suffer from the bias toward the nodes with high in-degree~\cite{Bar-Yossef08}. Another simple way to get around the difficulty is to just convert every directed edge to an undirected one (excluding the duplicate ones)~\cite{AAAI2001,Wang-Sigcomm10,Gabielkov-Sigcomm14}, which technically has noting to do with sampling from directed graphs.

A more practical method was proposed in~\cite{VLDB2000} to construct an undirected graph on-the-fly from a target directed graph for sampling, only considering the inbound links for each node discovered before the node is visited for the first time, and then to apply, in essence, the MH algorithm with $Q_{ij} \!=\! 1/\max_{k}d^+_k$ over the transformed \emph{undirected} graph to obtain uniform samples. In short, the edge conversion is done with edge pruning. This idea has been adopted later in~\cite{Bruno-Infocom12} for a weighted random walk with `random jumps' (to anywhere in the undirected graph), equipped with the importance sampling due to its non-uniform stationary distribution (on the undirected graph). In contrast to all these, \emph{without} knowing the stationary distributions or requiring any transformation of the directed graph to another undirected version, our NMMC method is able to sample from any target distribution over the directed graph with theoretical guarantees. We recently came to know the work by Avrachenkov et al.~\cite{Avrachenkov2018} in which they briefly mention that their proposed reinforcement learning-based sampling technique can potentially be applied to strongly connected directed graphs, but without any theoretical and practical demonstrations. We here emphasize that our NMMC method is even capable of sampling from a given distribution $\bm{\pi}$ over the set of reachable nodes $\N_S$ of the directed graph, even if the graph is not strongly connected, not to mention the capability of estimating the \emph{EVC} $\xx$ of the (strongly connected) directed graph.

\section{Conclusion}
\label{se:conclu}

In this paper, we have shown how to go beyond the limitation set by the current MCMC methods when state spaces violate the reciprocity condition and the global knowledge is unavailable. Unlike the most studies in the literature that aim to improve a given MCMC algorithm or to propose a better one within the same framework, our approach is fundamentally different and unique, in that it works directly over any nonreversible proposed chain such as simple random walks on directed graphs without sacrificing their distributed nature in the implementation. Our NMMC framework builds upon our careful mapping from the proposed chain to the target distribution on a set of transient states, and entails the machinery of the quasi-stationary distribution of a suitably constructed transient chain and the induced random walks with reinforcement to relocate to positions from their past history, which are amenable to distributed implementation using only locally available information. We have also devised a number of ways to improve the overall speed of convergence and to enhance the practicability, by utilizing online estimates of in-degree and the required global constant and by relaxing the assumption of the strong connectedness of the underlying directed graph, to make our NMMC method readily applicable to most directed graphs in reality.

\section{Acknowledgments}
The authors thank the anonymous reviewers and especially our shepherd, Konstantin Avrachenkov, for many constructive comments and suggestions that greatly improved the quality of this paper. Do Young Eun is supported in part by National Science Foundation under Grant Nos. CNS-1423151 and CNS-1824518.

%
\bibliographystyle{ACM-Reference-Format}
\bibliography{ref}

\end{document}